\documentclass[12pt,english]{article}
\usepackage[T1]{fontenc}
\usepackage[utf8]{inputenc}
\usepackage[letterpaper]{geometry}
\usepackage{amsmath}
\usepackage{amsthm}
\usepackage{tikz}
 \usetikzlibrary{arrows}
\usepackage{graphicx}
\usepackage{amssymb}

\makeatletter

\newtheorem{theorem}{Theorem}[section]

\newtheorem{lemma}[theorem]{Lemma}

\newtheorem{definition}[theorem]{Definition}

\newtheorem{remark}[theorem]{Remark}

\numberwithin{equation}{section}

\makeatother

\newcommand{\eqnsection}{
\renewcommand{\theequation}{\thesection.\arabic{equation}}
 \makeatletter   \csname  @addtoreset\endcsname{equation}{section}
   \makeatother}

\def\qed{$\Box $}

\def\R{\mathbb{R}}

\def\E{\mathbb{E}}
\def\P{\mathbb{P}}
\def\0{\mathbf{0}}
\def\1{\mathbf{1}}

\def\Var{{\mathop {{\rm Var\, }}}}
\def\Cov{{\mathop {{\rm Cov\, }}}}

\makeatother

\usepackage{babel}

\makeatother

\begin{document}

\title{Time-dependency in hyperbolic Anderson model: Stratonovich regime}

\author{\sc Xia Chen\thanks {Supported in part by
the Simons Foundation \#585506.} }
 
\date{} 
\maketitle

\begin{abstract}

In this paper, the 
 hyperbolic Anderson equation
 generated by a time-dependent Gaussian noise is under
 investigation
 in  two fronts: The solvability and large-$t$ asymptotics.
 The investigation leads to
 a necessary and sufficient condition 
 for existence  and a precise large-$t$
 limit form for the expectation of the solution.
 Three major developments are made
 for achieving these goals: A universal bound
 for Stratonovich moment that guarantees the Stratonovich integrability 
 and ${\cal L}^2$-convergence 
 of the Stratonovich chaos expansion under the best possible condition, a representation 
  of the expected Stratonovich
 moments  in terms of a time-randomized Brownian intersection local time,
 and a large deviation principle for the time-randomized Brownian intersection local time.

\end{abstract}

\begin{quote} {\footnotesize

}

\end{quote}

\section{Introduction} \label{intro}

The model studied in this paper is named after the physicist 
Philip Warren Anderson (\cite{Anderson})  who adds a
multiplicative Gaussian noise
to the heat equation in his investigation of magnetic impurities embedded in metals.
Due to its close links to other
physical models such as KPZ equation (\cite{KPZ}),  
especially in the wake of the breakthrough  of    \cite{Hairer}, the study of this equation
 has been rapidly developed. Today, the equation is known as parabolic
Anderson model in literature. We refer the interested readers  to  to
\cite{HHNS} and the references therein for  general information on this subject.

Due to lack of analytic tools such as Feynman-Kac formula, much less have been
known on hyperbolic models. Until very recent, it was widely believed that as
$\partial u/\partial t$ being replaced by $\partial^2u/\partial t^2$,
the hyperbolic equation
has a wilder behavior than its parabolic counterpart. The recent progress shows otherwise:
The hyperbolic equations are solvable  under the conditions (\cite{BCC}, \cite{CDST-1}) generally weaker than
 those posted for the existence of parabolic systems in the same regime.
 
 It is also noticeable that most of the recent development cited above focus
on the setting of time-independent Gaussian noise. The reason behind is that the time-dependency of
the Gaussian field posts  serious challenges and limitations to
the currently available  tools and ideas in dealing with
hyperbolic Anderson models.

In this paper we consider the hyperbolic Anderson equation 
\begin{align}\label{intro-1}
\left\{\begin{array}{ll}\displaystyle
\frac{\partial^2 u}{\partial t^2}(t,x)=\Delta u(t,x)
+ \dot{W}(t, x) 
u(t,x)\,, \hskip.2in (t, x)\in \R^+\times\R^d\\\\
u(0, x)=1  \hskip.1in\hbox{and} \hskip.1in
\frac{\partial u}{\partial t}(0,x)=0\,, \hskip.2in x\in\R^d\end{array}\right.
\end{align}
run by a mean zero and possibly generalized
Gaussian noise $\dot{W}(t, x)$ 
with
the covariance function
\begin{align}\label{intro-2}
  \Cov \big(\dot{W}(s, x), \dot{W}(t, y)\big)
=\vert s-t\vert^{-\alpha_0}\gamma(x-y)\,, \hskip.2in s,t\in\R_+,\hskip.05in  x,y\in\R^d,
\end{align}
where $0<\alpha_0<1$. As a  covariance function,  the non-negative definiteness   of  $\gamma(\cdot)$   implies that it admits
a spectral measure $\mu(d\xi)$ on
$\R^d$ uniquely defined by the relation
\begin{align}\label{intro-3}
  \gamma(x)=\int_{\R^d} e^{i\xi\cdot x}\mu(d\xi)\,, \hskip.3in x\in\R^d\,. 
\end{align}
 
Throughout this work, we assume that $\gamma(\cdot)\ge 0$ and $d=1,2,3$. 
The system is set up in Stratonovich regime. Roughly speaking, it means
 that the equation (\ref{intro-1}) is the result of the approximation by 
classical  wave equations run by the smoothed
Gaussian noise $\dot{W}_{\epsilon, \delta}(t, x)$.  We shall
provide the details of the construction of the solution 
in Section \ref{M}.

Our first concern is to find the best condition for the existence of  solution. The conditions for solvability
are often formulated
in term of the  integrability of the spectral measure $\mu(d\xi)$.  
In a different set-up known as Skorohod regime,
Chen, Deya, Song and Tindel  (\cite{CDST} posted the optimal condition
\begin{align}\label{intro-4}
  \int_{\R^d}\bigg(\frac{1}{1+\vert\xi\vert^2}\bigg)^{3-\alpha_0\over 2}\mu(d\xi)<\infty
\end{align}
for the existence/uniqueness of the system (\ref{intro-1}).

In their follow-up paper \cite{CDST-1} for
Stratonovich regime, (\ref{intro-1}) is solved under
an apparently less optimal assumption in the dimensions $d=1,2$.

\begin{theorem}\label{th-1}  Let $d=1,2,3$.
\begin{enumerate}
\item[(1)]   Under  the condition
  \begin{align}\label{intro-5} 
   \int_{\R^d}\bigg(\frac{1}{1+\vert \xi\vert^2}\bigg)^{2-\alpha_0\over 2}\mu(d\xi)<\infty
 \end{align}  
the equation (\ref{intro-1}) has a   solution in the sense of
Definition \ref{d.mild_solution} given in Section \ref{M}.
\item[(2)] If the equation (\ref{intro-1}) has a  square integrable solution $u(t,x)$ that admits the  Stratonovich expansion (see (\ref{M-3}))
  for some $t>0$, then condition \eqref{intro-5}
  must be satisfied.  
\end{enumerate}
\end{theorem}

For the purpose of comparison, we introduce the parabolic Anderson model
\begin{align}\label{intro-6}
\left\{\begin{array}{ll}\displaystyle
\frac{\partial u}{\partial t}(t,x)=\Delta u(t,x)
+ \dot{W}(t, x) 
u(t,x)\,, \hskip.2in (t, x)\in \R^+\times\R^d\\\\
u(0, x)=1   \hskip.2in x\in\R^d. \end{array}\right.
\end{align}
It is well known that the condition (see, e,g. Theorem 4.6, \cite{Song})
for (\ref{intro-6}) to be solvable (in a sense same as 
Theorem \ref{th-1}) is
 \begin{align}\label{intro-7} 
  \int_{\R^d}\bigg(\frac{1}{1+\vert \xi\vert^2}\bigg)^{1-\alpha_0}\mu(d\xi)<\infty.
 \end{align}  
In other words, the hyperbolic Anderson model is solvable under a condition genuinely weaker than the one set for
its parabolic counterpart when it comes to the Gaussian filed that is fractional in time.

In next main theorem is about the increasing rate of $\E u(t,x)$ as $t\to\infty$.
To this end, we assume the
homogeneity for the covariance structure:  
\begin{align}\label{intro-8} 
  \gamma(cx)=c^{-\alpha}\gamma(x)\,, \hskip.2in x\in\R^d\hskip.1in c>0
 \end{align} 
 for some $\alpha>0$.
 Taking $f(\lambda)=(1+\lambda^2)^{-{2-\alpha_0\over 2}}$ and $v(d\xi)=\mu(d\xi)$ in   \cite[Lemma 3.10]{CDST}
yields 
$$
\int_{\R^d}\bigg(\frac{1}{ 1+\vert\xi\vert^2}\bigg)^{2-\alpha_0\over 2}
\mu(d\xi)=\alpha\mu\{\xi\in\R^d;\hskip.05in\vert\xi\vert\le 1\}
\int_0^\infty \bigg(\frac{1}{ 1+\rho^2}\bigg)^{2-\alpha_0\over 2}\frac{d\rho}{\rho^{1-\alpha}}
$$
as far as either side is finite.  This
shows that under the homogeneity (\ref{intro-8}) on the 
noise covariance condition,
the condition (\ref{intro-6}) becomes   ``$\alpha_0+\alpha<2$''.
  In addition (Remark 1.4, \cite{CDST}), the fact that $\gamma(\cdot)$ is non-negative and
  non-negative
  definite (for being qualified as co-variance function)
  requires that $\alpha\le d$. Further, the only setting where
``$\alpha=d$'' is allowed
under $\alpha<2$ is when $\alpha=d=1$, or when $\gamma (\cdot)$ is a
constant multiple of Dirac function.

\begin{theorem}\label{th-2} Under the homogeneity condition (\ref{intro-8})
  with 
  \begin{align}\label{intro-9}
  \alpha_0+\alpha<2
 \end{align}
 with $0<\alpha_0<1$ and $0<\alpha<d$ or with $0<\alpha_0<1$ and $\alpha=d=1$,
 \begin{align}\label{intro-10} 
&\lim_{t\to\infty}t^{-{4-\alpha-\alpha_0\over 3-\alpha}}\log\E u(t,x)
=(3-\alpha)\bigg({2(4-\alpha-2\alpha_0)^{4-\alpha-2\alpha_0\over 2}\over
(4-\alpha-\alpha_0)^{4-\alpha-\alpha_0}}\Big({{\cal M}\over 4-\alpha}\Big)^{4-\alpha\over 2}\bigg)^{1\over 3-\alpha}
\end{align}
where
\begin{align}\label{intro-11}
  {\cal M}&=\sup_{g\in {\cal A}_d}\bigg\{\bigg(\int_0^1\!\!\int_0^1\int_{\R^d\times\R^d}
  {\gamma(x-y)\over\vert s-r\vert^{\alpha_0}}
  g^2(s, x)g^2(r, y)dxdydsdr\bigg)^{1/2}\\
  &-\int_0^1\!\int_{\R^d}\vert\nabla_x g(s, x)\vert^2dxds\bigg\}\nonumber
\end{align}
(which is finite under $\alpha<2$ by Lemma 5.2,  \cite{Chen-5}) and
$$
{\cal A}_d=\bigg\{g(s, x);\hskip.1in g(s,\cdot)\in W^{1,2}(\R^d)\hskip.1in \hbox{and}\hskip.05in
\int_{\R^d}\vert g(s, x)\vert^2dx=1\hskip.05in\hbox{for every $0\le s\le 1$}\bigg\}\,. 
$$
\end{theorem}

To the author's best knowledge, Theorem \ref{th-2} appears to be the first time
that a  precise long term asymptotics are obtained for the hyperbolic Anderson models with
time-fractional random noise.  A result close to (\ref{intro-11}) found in literature is obtained by Balan and Conus
(Theorem 2.1, \cite{BC}) 
where the system (\ref{intro-1}) is set up in Skorohod regime and the bounds
$$
\limsup_{t\to\infty}t^{-{4-\alpha -\alpha_0\over 3-\alpha}}\log\E \vert u(t,x)\vert^p\le C_p<\infty\hskip.2in (p\ge 2)
$$
$$
\liminf_{t\to\infty}t^{-{4-\alpha -\alpha_0\over 3-\alpha}}\log\E u^2(t,x)>0
$$
are obtained.

In the case of parabolic Anderson model (\ref{intro-6}), the moment asymptotics (Theorem 6.1, \cite{CHSX})
follow the pattern
\begin{align}\label{intro-12}
\lim_{t\to\infty}t^{-{4-\alpha -2\alpha_0\over 2-\alpha}}\log\E u^p(t,x)\hskip.2in p=1,2,\cdots
\end{align}
under the scaling property (\ref{intro-8})) with
\begin{align}\label{intro-13}
2\alpha_0+\alpha<2.
\end{align}
Comparing (\ref{intro-10}) and (\ref{intro-9}) with (\ref{intro-12}) and (\ref{intro-13}), respectively,
one can see the contribution from the time component of the Gaussian field are different between
hyperbolic and parabolic settings. See Remark \ref{B-0} below for an explanation from a new perspective.

In the recent work (\cite{CH}) on the hyperbolic Anderson model with time-independent noise $\dot{W}(x)$,
the solvability is established under the Dalang's condition
 \begin{align}\label{intro-14}
  \int_{\R^d}{1\over 1+\vert\xi\vert^2}\mu(d\xi)<\infty
  \end{align}
 and  the long term moment asymptotics are established in the form of
\begin{align}\label{intro-15}
\lim_{t\to\infty}t^{-{4-\alpha\over 3-\alpha}}\log\E u^p(t,x)\hskip.2in p=1,2,\cdots
\end{align}
under the condition $\alpha <2$.

 On the side of idea development, the current work is partially motivated by \cite{CH} 
 where the Stratonovich moment is represented in terms of the intersection local times
 $$
 \int_0^t\!\int_0^t\gamma\big(B(s)-B(r)\big)dsdr\hskip.1in \hbox{and}\hskip.1in
  \int_0^t\!\int_0^t\gamma\big(B(s)-\widetilde{B}(r)\big)dsdr
 $$
 where $B(t)$ abd $\widetilde{B}(t)$ are independent $d$-dimensional Brownian motions
  (see Corollary 3.3, \cite{CH} for details).
  This connection is established on the simple fact ((3.6), \cite{CH}) that
 \begin{align}\label{intro-16}
\int_0^\infty e^{-\lambda t}G(t,x)dt={1\over 2}\int_0^\infty \exp\Big\{-{\lambda^2\over 2}t\Big\}p(t,x)dt\hskip.2in (\lambda>0)
\end{align}
where $p(t,x)$ is the Brownian semi-group
 \begin{align}\label{intro-17}
p(t,x)=(2\pi t)^{-d/2}\exp\Big\{-{\vert x\vert^2\over 2t}\Big\}\hskip.2in (t,x)\in\R_+\times\R^d
\end{align}
and where $G(t,x)$ is the fundamental solution (see (\ref{M-2}) below)
of the wave equation. Indeed, one of the crucial observations (Theorem \ref{th-5}) made in the current work is a
link  between
the Stratonovich moment and the time-randomized intersection local time
$$
\int_0^t\!\int_0^t\Big(\vert s-r\vert +i\big(\beta(s)-\beta(r)\big)\Big)^{-\alpha_0}\gamma\big(B(s)-B(r)\big)dsdr
$$
where $\beta(t)$ is an 1-dimensional Brownian motion independent of $B(t)$.
Accordingly, 
 a large deviation principle
(Theorem \ref{th-6}) for the time-randomized intersection local time is established that requires some new ideas
(see Remark \ref{T-0} and the discussion at the beginning of Section \ref{T}).

Another important task carried out in this paper is the legalization of Stratanovich expansion (\ref{M-3})
under the best condition (\ref{intro-5}). To this end, a universal
bound (Theorem \ref{th-3}) is established that is responsible for the Stratonovich integrability and ${\cal L}^2$-convergence of  the Stratonovich expansion. It should be pointed out that Stratonovich moment demands a
level of technology higher than its Skorohod counter part.

This paper also brings a different idea on the treatment of the time-covariance $\vert\cdot\vert^{-\alpha_0}$
(introduced in (\ref{intro-2})). There have been two ways in literature in handling $\vert\cdot\vert^{-\alpha_0}$.
In the work \cite{BC} and some of its follow-up papers, the Hardy–Littlewood–Sobolev inequality is used
(Lemma B.3, \cite{BC}) in the Skorohod regime to separate the time component. This strategy appears to be
powerful in the parabolic setting. In the hyperbolic setting, it does not function as good as it for the parabolic
equations, as it is less adoptive to the oscillation behavior of wave operator. 
Another existing practice (e.g., \cite{CDST}) is to perform the Fourier transform
$$
\vert u\vert^{-\alpha_0} =C\int_{\R}e^{i\lambda u}{d\lambda\over\vert\lambda\vert^{1-\alpha_0}}\hskip.2in u\in\R
$$
to $\vert\cdot\vert^{-\alpha_0}$. The use of Fourier transformation has been popular and effective in parabolic setting
(and hyperbolic/Skorohod setting as well).
For the hyperbolic equation in  the 
Stratonovich regime, it creates an annoying and un-controllable singularity.
 Instead, it is proposed in this work to use the Laplace transform
 \begin{align}\label{intro-18}
\vert u\vert^{-\alpha_0} =\Gamma(\alpha_0)^{-1}\int_0^\infty e^{-\lambda \vert u\vert}{d\lambda\over\lambda^{1-\alpha_0}}\hskip.2in u\in\R.
\end{align}

It should be pointed out that the time-dependency in hyperbolic Anderson model is far from fully understood. More specifically,
a deeper connection between the Stratonovich expansion and the Brownian intersection local times is very likely.
The progress should
lead to a fully understand of the intermittency of the system. Despite of its incompleteness, on the other hand,
Theorem \ref{th-2} strongly indicates the most possible pattern of intermittency
\begin{align}\label{intro-19} 
&\lim_{t\to\infty}t^{-{4-\alpha-\alpha_0\over 3-\alpha}}\log\E u^p(t,x)
=(3-\alpha)p^{4-\alpha\over 3-\alpha}\bigg({2(4-\alpha-2\alpha_0)^{4-\alpha-2\alpha_0\over 2}\over
(4-\alpha-\alpha_0)^{4-\alpha-\alpha_0}}\Big({{\cal M}\over 4-\alpha}\Big)^{4-\alpha\over 2}\bigg)^{1\over 3-\alpha}
\end{align}
for $p=1,2,\cdots$. Here we post it as conjecture and leave it for future study.

Here is the organization of the paper. In  next section (Section \ref{M}),  we set a way
to approximate the possibly generalized Gaussian field $\dot{W}(t,x)$,
introduce  the  multiple Stratonovich 
integral   and 
formally express the solution  as  Stratonovich expansion.
In Section \ref{bound}, we establish a universal bound for the moments in the Stratonovich expansion.
Using this bound we install the Sratonovich integrability 
for the  functions $g_n(\cdot, t,x)$ in Section \ref{S} and establish the convergence of the Stratonovich
expansion
(and therefore prove Theorem \ref{th-1}) in Section \ref{E}. In Section \ref{B}, we establish
a link between $\E S_{2n}\big(g_{2n}(\cdot, t,x)\big)$ and the $n$-th moment of a time-randomized
Brownian intersection local time under time exponentiation.  In Section \ref{T}, we prove
a large deviation principle for the time-randomized
Brownian intersection local time. Using the results  from Section \ref{B} and \ref{T}, we prove 
Theorem \ref{th-2} in Section \ref{P}. In the appendix, we conduct some elementary
calculation for the bounds established in Section \ref{bound}.

\section{Stratonovich expansion and approximations} \label{M}
As usual by the Duhamel principle the mathematical definition of the hyperbolic Anderson equation (\ref{intro-1}) will be  the following mild form
\begin{align}\label{M-1}
	u(t,x)=1+\int_{\R^d}\!\int_0^tG(t-s, x-y)u(s,y)W(ds, dy)\,, 
\end{align}
where 
\begin{enumerate}
	\item[(i)] $G(t, x)$ is the fundamental solution  defined by the
	deterministic wave equation
	\begin{align}\label{M-2}
		\left\{\begin{array}{ll}\displaystyle \frac{\partial^2 G}{\partial t^2}(t,x)
			=\Delta G(t,x) \\\\
			\displaystyle G(0, x)=0\hskip.1in\hbox{and}\hskip.1in\frac{\partial G}{\partial t}(0,x)
			=\delta_0(x)\,, \hskip.2in x\in\R^d\,.\end{array}\right.
	\end{align} 
		\item[(ii)] the stochastic integral on the right hand side of (\ref{M-1})
	is interpreted in the sense of Stratonovich  (see discussion below for details).   
\end{enumerate}

\subsection{Green's function}

The fundamental solution $G(t,x)$ associated with  (\ref{M-2}) plays a key role in determining the behavior of
the system (\ref{M-1}). Let us recall some  basic facts. 
Taking Fourier  transform in (\ref{M-2}) we get the expression for the 
fundamental solution
\begin{align}\label{M-11}
	{\cal F}\big(G(t,\cdot)\big)(\xi)=\int_{\R^d}G(t,x)e^{i\xi\cdot x}dx=\frac{\sin(\vert\xi\vert t)}{\vert\xi\vert}\,, \hskip.2in
	(t,\xi)\in\R_+\times\R^d
\end{align}
in its Fourier transform form. In particular,
\begin{align}\label{M-11-1}
\int_{\R^d}G(t,x)dx={\cal F}\big(G(t,\cdot)\big)(0)=t
\end{align}
In the dimensions $d=1,2,3$, the fundamental solution $G(t,x)$ 
itself can be expressed explicitly as
\begin{align}\label{M-12}
	G(t,x)=
	\begin{cases}\displaystyle\frac12 1_{\{\vert x\vert\le t\}}&\hskip.3in d=1\\\\
		\displaystyle \frac{1}{ 2\pi}\frac{ 1_{\{\vert x\vert\le t\}}}{\sqrt{t^2-\vert x\vert^2}}&\hskip.3in d=2\\\\
		\displaystyle\frac{1}{ 4\pi t}\sigma_t(dx) &\hskip.3in d=3\,,  
	\end{cases}
\end{align}
where $\sigma_t(dx)$ is the surface measure on the sphere $\{x\in\R^3;\hskip.05in\vert x\vert=t\}$.
We limit our attention to $d=1,2,3$ in this work because the treatment developed   here
requires $G(t,x)\ge 0$.  A scaling property we frequently use 
(especially in the proof of Theorem \ref{th-2}) is
\begin{align}\label{M-13}
	G(t,x)=t^{-(d-1)}G(1,t^{-1}x)\,, \hskip.2in (t,x)\in\R^+\times\R^d\,. 
\end{align}

\subsection{Smoothed version of $\dot{W}(t,x)$} \label{W}

Generally speaking, the smoothed version of the generalized Gaussian field $\dot{W}(t,x)$ can be
any family $\big\{\dot{W}_{\epsilon,\delta}(t,x);\hskip.1in\epsilon,\delta>0\big\}$ of mean-zero Gaussian
fields on $\R_+\times\R^d$, living in the same
probability space $(\Omega, {\cal A},\P)$ as $W(t,x)$ does, such that for each $\epsilon, \delta>0$, $\dot{W}_{\epsilon,\delta}(t,x)$
is defined point-wise  even with path-continuity (if needed), and of the form of covariance function
\begin{align}\label{W-1}
\Cov\Big(\dot{W}_{\epsilon,\delta}(s,x), \dot{W}_{\epsilon,\delta}(t,y)\Big)
=\gamma_{\delta}^0(s-t)\gamma_\epsilon(x-y)\hskip.2in (s,x), (t,y)\in\R_+\times\R^d
\end{align}
that satisfies
\begin{align}\label{W-2}
&\lim_{\epsilon,\delta\to 0^+}\int_{(\R_+\times\R^d)^2}\gamma_{\delta}^0(s-t)\gamma_\epsilon(x-y)f(s,x)g(t,y)dsdtdxdy\\
&=\int_{(\R_+\times\R^d)^2}\vert s-t\vert^{-\alpha_0}\gamma(x-y)f(s,x)g(t,y)dsdtdxdy\hskip.2in f,g\in{\cal S}(\R_+\times\R^d)
\nonumber
\end{align}
where ${\cal S}(\R_+\times\R^d)$ is the Schwartz space of all functions on $\R_+\times\R^d$ that are infinitely
differentiable and rapidly decay to zero at infinity.

A popular construction of $\dot{W}_{\epsilon,\delta}$ in literature is by convolution:
$$
W_{\epsilon,\delta}(t,x)=\int_{\R^{d+1}}p_0(\delta, t-u)p(\epsilon, x-y)\dot{W}(u,y)dudy
\hskip.2in (t,x)\in\R_+\times\R^d
$$
where $p(t,x)$ is the Brownian semi-group defined on $\R^d$ that is given in (\ref{intro-17}) and $p_0(t,x)$ is the
Brownian semi-group on $\R$.

To promote the use of the Laplace transform proposed in (\ref{intro-18}), we shall construct $W_{\epsilon,\delta}(t,x)$
in a slightly different way such that
\begin{align}\label{W-3}
\Cov\Big(\dot{W}_{\epsilon,\delta}(s,x), \dot{W}_{\epsilon,\delta}(t,y)\Big)
=\bigg(\Gamma(\alpha_0)^{-1}\int_0^{\delta^{-1}}e^{-\lambda\vert s-t\vert}{d\lambda\over\lambda^{1-\alpha_0}}\bigg)
\gamma_{2\epsilon}(x-y)
\end{align}
for $(s,x), (t,y)\in\R_+\times\R^d$, where
$$
\gamma_\epsilon(x)=\int_{\R^d}p(\epsilon, x-y)\gamma(y)dy\hskip.2in x\in\R^d.
$$

If exists (i.e., if they can live in the same probability space with $W(t,x)$), the family 
$\{W_{\epsilon,\delta}(\cdot,\cdot);\hskip.1in\epsilon,\delta>0\}$ meets all requirements
as the smoothed version of $\dot{W}(\cdot,\cdot)$: First, by the expression
\begin{align}\label{W-4}
\Gamma(\alpha_0)^{-1}\int_0^{\delta^{-1}}e^{-\lambda\vert s-t\vert}{d\lambda\over\lambda^{1-\alpha_0}}\
=\Gamma(\alpha_0)^{-1}\E^\kappa\int_0^{\delta^{-1}}e^{i\lambda\kappa( s-t)}{d\lambda\over\lambda^{1-\alpha_0}}\
\end{align}
the function of $s-t$ on the left is non-negative definite and
therefore qualified to be used as covariance
function, where $\kappa$ is a standard 1-dimensional Cauchy random variable.

Second, (\ref{W-2}) holds in light of (\ref{intro-18}).

Third, by the relation
$$
\begin{aligned}
&\E\Big(\dot{W}_{\epsilon,\delta}(s,x)-\dot{W}_{\epsilon,\delta}(t,y)\Big)^2\\
&=2\Gamma(\alpha_0)^{-1}\Bigg\{\bigg(\int_0^{\delta^{-1}}{d\lambda\over\lambda^{1-\alpha_0}}\bigg)\gamma_\epsilon(0)
-\bigg(\int_0^{\delta^{-1}}e^{-\lambda\vert s-t\vert}{d\lambda\over\lambda^{1-\alpha_0}}\bigg)
\gamma_\epsilon(x-y)\Bigg\}\\
&\le C_{\epsilon,\delta}\Big\{\vert s-t\vert +\vert x-y\vert\Big\}
\end{aligned}
$$
and therefore by normality
$$
\E\Big(\dot{W}_{\epsilon,\delta}(s,x)-\dot{W}_{\epsilon,\delta}(t,y)\Big)^{2n}
\le {(2n)!\over 2^nn!}C^n_{\epsilon,\delta}\Big\{\vert s-t\vert +\vert x-y\vert\Big\}^n
$$
for any integer $n\ge 1$.
A standard use of Kolmogorov continuity theorem (Theorem D.7, p.313, \cite{Chen-1}), the Gaussian
field $W_{\epsilon,\delta}(t,x)$ has a continuous modification on $\R_+\times\R^d$.

To have the family $\big\{\dot{W}_{\epsilon,\delta};\hskip.1in\epsilon,\delta>0\big\}$ live in the same
probability space. We start with the following simple observation: Given $0<T_1<\cdots <T_n<\cdots$, let
$\Delta_k\dot{W}(t,x)$ ($k=1, 2,\cdots $) be independent mean-zero (possibly generalized)
Gaussian fields on $\R_+\times\R^d$
with the covariance functions
$$
\Cov\Big(\Delta_k\dot{W}(s,x), \Delta_k\dot{W}(t,y)\Big)=\bigg(\Gamma(\alpha_0)^{-1}\int_{T_{k-1}}^{T_k}e^{-\lambda\vert s-t\vert}{d\lambda\over\lambda^{1-\alpha_0}}\bigg)
\gamma(x-y)
$$
for $(s,x), (t,y)\in\R_+\times\R^d$, where we follow the convention $T_0=0$. Set
$$
\dot{W}_{T_n}(t,x)=\sum_{k=1}^n\Delta_k\dot{W}(t,x)\hskip.2in n=1,2,\cdots.
$$
It is straightforward to see that
$$
\Cov\Big(\dot{W}_{T_n}(s,x), \dot{W}_{T_n}(t,y)\Big)=\bigg(\Gamma(\alpha_0)^{-1}\int_0^{T_n}e^{-\lambda\vert s-t\vert}{d\lambda\over\lambda^{1-\alpha_0}}\bigg)
\gamma(x-y).
$$
Without changing its distribution, one can re-define the Gaussian field $W(t,x)$ as
$$
\dot{W}(t,x)=\sum_{k=1}^\infty\Delta_k\dot{W}(t,x)
$$
for any monotonic sequence $\{T_n\}$ satisfying $T_n\to\infty$ ($n\to\infty$).

By Kolmogorov consistence extension, we can extend $\{\dot{W}_{T_n}(\cdot, \cdot);\hskip.1in n=1,2,\cdots\}$ to
the bigger family $\{\dot{W}_T(\cdot, \cdot);\hskip.1in T>0\}$. Then we adopt the new notation that replaces 
$\dot{W}_{\delta^{-1}}(\cdot, \cdot)$ by $\dot{W}_{\delta}(\cdot, \cdot)$. Finally we define
$$
\dot{W}_{\epsilon,\delta}(t,x)=\int_{\R^d}p(\epsilon, x-y)\dot{W}_\delta(t, y)dy
$$
which satisfies (\ref{W-3}). More generally
\begin{align}\label{W-5}
\Cov\Big(\dot{W}_{\epsilon,\delta}(s,x), \dot{W}_{\tilde{\epsilon},\tilde{\delta}}(t,y)\Big)
=\bigg(\Gamma(\alpha_0)^{-1}\int_0^{\delta^{-1}\wedge\tilde{\delta}^{-1}}e^{-\lambda\vert s-t\vert}{d\lambda\over\lambda^{1-\alpha_0}}\bigg)
\gamma_{\epsilon+\tilde{\epsilon}}(x-y)
\end{align}
for any $\epsilon, \tilde{\epsilon}, \delta, \tilde{\delta}>0$.

\subsection{Stratonovich integral}\label{D}

Given a random field $\Psi(t, x)$ ($(t, x)\in\R_+\times\R^d$) such that
$$
\int_{\R_+\times\R^d}\Psi(t, x)\dot{W}_{\epsilon,\delta}(t, x)dtdx\in{\cal L}^2(\Omega, {\cal F},\P)\hskip.2in \forall\varepsilon>0\,, 
$$
where $\dot{W}_{\epsilon,\delta}(t,x)$ is constructed in Section \ref{W},
define  the Stratonovich integral of $\Psi(t, x)$ as 
\begin{equation}
	\int_{\R_+\times\R^d}\Psi(t, x)W(dt, dx)\buildrel \Delta\over
	=\lim_{\epsilon,\delta\to 0^+}\int_{\R_+\times\R^d}\Psi(t, x)
	\dot{W}_{\epsilon,\delta}(t,x)dtdx
	\label{e.def_integral} 
\end{equation} 
whenever  such limit exists in ${\cal L}^2(\Omega, {\cal F},\P)$. We can also use the convergence in probability in above definition. But as in most works  on SPDE, ${\cal L}^2(\Omega, {\cal F},\P)$ norm is   easier to deal with so that we choose the ${\cal L}^2(\Omega, {\cal F},\P)$ convergence 
throughout  this work.  Notice that   this definition 
implicates  that $u(t,x) $  as  a
solution to (\ref{M-1}) is in $  {\cal L}^2(\Omega, {\cal F},\P)$   for all $(t,x)\in\R^+\times\R^d$.  After  defining the Stratonovich integral, we can give the following definition about  the solution.
\begin{definition}\label{d.mild_solution} 
	A random field $\{u(t,x)\,, t\ge 0\,, x\in \R^d\}$  is called a mild solution 
	to \eqref{intro-1} if  for any $(t,x)\R_+\times\R^d$, the random field 
	$$\Psi(s,y)\equiv G(t-s, x-y)u(s,y)1_{[0,t]}(s)
	$$  
	is Stratonovich integrable and if  \eqref{M-1} is satisfied. 
\end{definition}

To prove Theorem \ref{th-1}, we shall use the Stratonovich 
expansion (see \cite{humeyer}, \cite{hubook} and
references therein for the multiple
Stratonovich integrals). 
According to the algorithm in \cite{DMT},
formally  iterating (\ref{M-1})
infinitely many  times we   have  heuristically a solution candidate 
\begin{align}\label{M-3}
	u(t,x)=\sum_{n=0}^\infty S_n\big(g_n(\cdot,t,x)\big)
\end{align}
with $S_0\big(g_0(\cdot, t,x)\big)=   1$. Here is how the notation
$S_n\big(g_n(\cdot,t,x)\big)$ is justified:
The iteration procedure creates the recurrent relation
\begin{equation}
	S_{n+1}\big(g_{n+1}(\cdot, t,x)\big)= \int_{\R^d}\!\int_0^t
	G(t-s, x-y)S_n\big(g_n(\cdot, s,y)\big)
	{W}(ds,dy) \,. 
	\label{e.2.9}
\end{equation} 
Iterating this relation formally we have
\begin{align}\label{M-4}
	&S_n\big(g_n(\cdot,t,x)\big)\\
	&=\int_{(\R^d)^n}\int_{[0,t]_<^n}
	G(t-s_n, x-x_n)\cdots G(s_2-s_1, x_2-x_1)W(ds_1, dx_1)\cdots W(ds_n, dx_n)\nonumber\\
	&=\int_{(\R_+\times\R^d)^n}g_n(s_1,\cdots, s_n, x_1,\cdots, x_n, t, x)W(ds_1,dx_1)\cdots W(ds_n, dx_n)
		\hskip.2in \hbox{(say)}\,, \nonumber                               
\end{align}
where $[0,t]_<^n:=\left\{ (s_1, \cdots, s_n)\in [0, t]^n \ \  \hbox{satisfies } 
  \ \  0<s_1<s_2<\cdots<s_n<t\right\}$,
and the conventions $x_{n+1}=x$ and $s_{n+1}=t$ are adopted.
Thus, the notation ``$S_n\big(g_n(\cdot,t,x)\big)$'' is reasonably introduced
for
an  $n$-multiple Gaussian integral of the integrand
\begin{align}\label{M-14}
	&g_n(s_1,\cdots, s_n, x_1,\cdots, x_n, t,x)\\
	&=\Big(G(t-s_n, x-x_n)\cdots G(s_2-s_1, x_2-x_1)\Big)
	1_{[0,t]_<^n}(s_1,\cdots, s_n)\nonumber
\end{align}
($n=1,2,\cdots$). In Section \ref{S}, the Stratonovich
integrability of $g_n(\cdot, t,x)$
shall be rigorously established
in Theorem \ref{th-4}.

The above construction indicates  that the  existence of  the system (\ref{M-1}) 
can be implied by  the convergence
of the random series defined by  (\ref{M-3}) in an  appropriate form. 
This will be justified rigorously in Section \ref{E} as part of the proof
of Theorem \ref{th-1}.

\begin{definition}\label{def-2.2}
	Let $f:(\R_+\times\R^d)^n\to \R $ be measurable
	such that   for every $\epsilon, \delta>0$ 
	$$
	\int_{(\R_+\times\R^d)^{n}}f(s_1,\cdots, s_n, x_1\cdots, x_{n})\bigg(\prod_{k=1}^n\dot{W}_{\epsilon, \delta}(s_k, x_k)\bigg)
	ds_1\cdots ds_ndx_1\cdots dx_n\in   {\cal L}^2(\Omega, {\cal F}, \P)\,.
	$$
	Then we define the $n$-multiple Stratonovich integral of $f$ as
	\begin{align}\label{M-6}
		S_n(f):=&   \int_{(\R_+\times\R^d)^{n}}  f(s_1,\cdots, s_n, x_1\cdots, x_{n})W(ds_1, dx_1)\cdots W(ds_n, dx_n)\\
	 =&\lim_{\epsilon,\delta\to 0^+}
		\int_{(\R_+\times\R^d)^{n}}f(s_1, \cdots, s_n, x_1\cdots, x_{n})\
		\bigg(\prod_{k=1}^n\dot{W}_{\epsilon,\delta}(s_k, x_k)\bigg)
		ds_1\cdots ds_ndx_1\cdots dx_n \nonumber  
	\end{align}
	whenever the limit exists $  {\cal L}^2(\Omega, {\cal F}, \P)$. 
	
\end{definition}
\begin{remark}Along with the set-up of our model, the Stratonovich integrand $f$ is
	given as a generalized function in the dimension three ($d=3$). Indeed  (\cite{pipiras}),  Definition \ref{def-2.2} can be extended
	to the setting of  generalized functions $f$.  A detail is provided near the end of this section for the construction
	needed in $d=3$.
\end{remark} 

The following lemma provides a convenient test
of Stratonovich integrability that we shall use   in this work. 

\begin{lemma}\label{L-0} The $n$-multiple Stratonovich integral $S_n(f)$ 
	exists
	if and only if the limit
	$$
	\begin{aligned}
		&\lim_{\stackrel{(\epsilon,\delta)\to 0^+}{\scriptstyle(\epsilon',\delta')\to 0^+}}
		\E\bigg\{\int_{(\R^d)^n}f(s_1,\cdots, s_n, x_1,\cdots x_n)\bigg(\prod_{k=1}^n\dot{W}_{\epsilon,\delta}(s_k, x_k)\bigg)ds_1\cdots ds_ndx_1\cdots dx_n\bigg\}\\
		&\times\bigg\{\int_{(\R_+\times\R^d)^n}f(s_1,\cdots, ds_n, x_1,\cdots x_n)\bigg(\prod_{k=1}^n
		\dot{W}_{\epsilon',\delta'}(s_k, x_k)\bigg)
		ds_1\cdots ds_ndx_1\cdots dx_n\bigg\}
	\end{aligned}
	$$
	exists
\end{lemma}

\begin{proof} The existence of the limit in (\ref{M-6}) is
	another way to say that the family
	$$
	{\cal Z}_{\epsilon,\delta}
	=\int_{(\R_+\times\R^d)^n}f(s_1,\cdots, s_n, x_1,\cdots x_n)\bigg(\prod_{k=1}^n\dot{W}_{\epsilon, \delta}(s_k, x_k)\bigg)ds_1\cdots ds_ndx_1\cdots dx_n
	$$
	is a Cauchy sequence in ${\cal L}^2(\Omega, {\cal F},\P)$ as $\epsilon, \delta\to 0^+$, 
	which is equivalent to the lemma.
\end{proof}

Definition \ref{def-2.2} can be extended to a random field
$f(s_1,\cdots, s_n, x_1,\cdots, x_n)$ in an obvious way.  Most of the time in this paper,
however,
we deal with a deterministic integrand and demand some effective ways to compute the expectation of   multiple Stratonovich
integral  of deterministic integrands. 

\begin{lemma}\label{L-1} In the setting of deterministic integrand, the ${\cal L}^2$-convergence in Definition \ref{def-2.2}
can be replaced ${\cal L}^p$-convergence for any $p\ge 2$.
\end{lemma}

\proof This has been proved in the setting of time-independent noise (Theorem 6.2 and Remark 6.3, \cite{CH}). In
the context of the general integrand, one can identify  $f(s_1,\cdots, s_n, x_1,\cdots x_n)$ with $f(x_1,\cdots, x_n)$ and the time-dependent
Guassian field $\dot{W}(t,x)$ with the time-independent Gaussian field $\dot{W}(x)$ by viewing the time variable as the extra
dimension of the space variables. Then Theorem 6.2 and Remark 6.3, \cite{CH} applies to our setting. \qed

This lemma brings some convenience in computation. An example of such is Fubini theorem: Given
the integers $l_1,\cdots, l_m\ge 1$ and the $l_j$-multiple time-space variate functions
$f_j$ ($1\le j\le m$), the Stratonovich integrabilities of
$f_1,\cdots, f_m$
implies the Stratonovich integrability of $f_1\otimes\cdots \otimes f_m$
and
\begin{align}\label{M-16}
	S_{l_1+\cdots +l_m}\big(f_1\otimes\cdots \otimes f_m\big)
	=\prod_{j=1}^mS_{l_j}(f_j)\,. 
\end{align}

Let us
recall  an identity \cite[p.201, Lemma 5.2.6]{MR} known as Wick's formula
which states that
\begin{align}\label{M-8}
	\left\{\begin{array}{ll}\displaystyle
		\E\prod_{k=1}^{2n}g_k=\sum_{{\cal D}\in\Pi_n}\prod_{(j,k)\in{\cal D}}\E g_jg_k\\\\
		\displaystyle\E\prod_{k=1}^{2n-1}g_k=0\,, \end{array}\right.
\end{align}
where $(g_1,\cdots,g_{2n})$ is a
mean zero normal vector,
and $\Pi_n$ is the set  of all pair partitions  of $\{1,2,\cdots, 2n\}$. As a side remark,
$\#(\Pi_n)={(2n)!\over 2^n n!}$.
Applying  (\ref{M-8})  to $g_k=\dot{W}_{\epsilon, \delta}(s_k, x_k)$ in the case of deterministic integrand $f$,
and taking the $(\epsilon, \delta)$-limit, we have
\begin{align}\label{M-9}
\E S_{2n-1}(f) =0
\end{align}
and
\begin{align}\label{M-10}
	\E S_{2n}(f)&=\sum_{{\cal D}\in\Pi_n}
	\int_{(\R_+\times\R^d)^{2n}}f(s_1,\cdots, s_{2n}, x_1,\cdots, x_{2n})\\
	&\times\bigg(\prod_{(j,k)\in {\cal D}}\vert s_j-s_k\vert^{-\alpha_0}\gamma(x_j-x_k)\bigg)
	ds_1\cdots ds_{2n}dx_1\cdots dx_{2n}\nonumber
\end{align}
under the Stratonovich integrability.
In particular, the expectation of a $(2n)$-multiple Stratonovich
integral is non-negative if the integrand is non-negative. Take
$f=g_{2n}(\cdot, t, 0)$ for example. In the case of integrability
$$
\begin{aligned}
\E S_{2n}\big(g_{2n}(\cdot, t, 0)\big)
&=\sum_{{\cal D}\in\Pi_n}\int_{[0,t]_<^{2n}}\!\int_{(\R^d)^{2n}}
\Big(G(t-s_{2n}, -x_{2n})\cdots G(s_2-s_1, x_2-x_1)\Big)\cr
&\times\bigg(\prod_{(j,k)\in {\cal D}}\vert s_j-s_k\vert^{-\alpha_0}\gamma(x_j-x_k)\bigg)
 ds_1\cdots ds_{2n}dx_1\cdots dx_{2n}.
\end{aligned}
$$
With the substitutions $s_k\mapsto t-s_{2n-k+1}$ and $x_k\mapsto -x_{2n-k+1}$.
($k=1,2,\cdots, 2n$) performed on the right hand side,
\begin{align}\label{M-10-20}
&\E S_{2n}\big(g_{2n}(\cdot, t, 0)\big)
=\sum_{{\cal D}\in\Pi_n}\
\int_{[0,t]_<^{2n}}\!\int_{(\R^d)^{2n}}\bigg(\prod_{l=1}^{2n}G(s_{l}-s_{l-1}, x_{l}-x_{l-1})\bigg)\\
&\times\bigg(\prod_{(j,k)\in {\cal D}}\vert s_{2n-j+1}-s_{2n-k+1}\vert^{-\alpha_0}\gamma(x_{2n-j+1}-x_{2n-k+1})\bigg)
 ds_1\cdots ds_{2n}dx_1\cdots dx_{2n}\nonumber\\
 &=\sum_{{\cal D}\in\Pi_n}\int_{[0,t]_<^{2n}}\!\int_{(\R^d)^{2n}}
	\bigg(\prod_{l=1}^{2n}G(s_l-s_{l-1}, x_l-x_{l-1})\bigg)\nonumber\\
	&\times\bigg(\prod_{(j,k)\in {\cal D}}\vert s_j-s_k\vert^{-\alpha_0}\gamma(x_j-x_k)\bigg)
	ds_1\cdots ds_{2n}dx_1\cdots dx_{2n}\nonumber
\end{align}
where we follow the convention on the right hand side that $s_0=0$ and $x_0=0$.

Since  the condition (\ref{intro-5})
encompasses the cases where the covariance function $\gamma(\cdot)$ exists only as a
generalized function (e.g., $\gamma(\cdot)=\delta_0(\cdot)$ in $d=1$), the meaning of the multiple integrals on
the right hand side of (\ref{M-10}) needs to be clarified.
Indeed, by (\ref{W-3}) and (\ref{M-8})
\begin{align}\label{M-30}
	&\E\int_{(\R_+\times\R^d)^{2n}} f(s_1,\cdots, s_{2n}, x_1,\cdots, x_{2n})\bigg(
	\prod_{k=1} ^{2n} \dot{W}_{\epsilon,\delta}
	(s_k, x_k)\bigg)ds_1\cdots ds_{2n}dx_1\cdots dx_{2n}\\
	&=\sum_{{\cal D}\in\Pi_n} \int_{(\R_+\times\R^d)^{2n}}f(s_1,\cdots, s_{2n}, x_1,\cdots, x_{2n})
	\bigg(\prod_{(j,k)\in {\cal D}}\gamma_\delta^0(s_j-s_k)\gamma_{2\epsilon}(x_j-x_k)\bigg)\nonumber\\
	&\times ds_1\cdots ds_{2n}dx_1\cdots dx_{2n}\nonumber
\end{align}
where
\begin{align}\label{M-17}
\gamma_\delta^0(u)
=\Gamma(\alpha_0)^{-1}\int_0^{\delta^{-1}}e^{-\lambda\vert u\vert}{d\lambda\over\lambda^{1-\alpha_0}}
\hskip.1in\hbox{and}\hskip.1in
\gamma_\epsilon(x)=\int_{\R^d}\gamma(y)p_\epsilon(x-y)dy.
\end{align}
Inspired by (\ref{M-6}) and in light of (\ref{M-30}), we therefore define 
\begin{align}\label{M-15}
	&\int_{(\R_+\times\R^d)^{2n}}f(s_1,\cdots, s_{2n}, x_1,\cdots, x_{2n})
	\bigg(\prod_{(j,k)\in {\cal D}}\vert s_j-s_k\vert^{-\alpha_0}\gamma(x_j-x_k)\bigg)\\
	&\times ds_1\cdots ds_{2n}dx_1\cdots dx_{2n}\nonumber\\
	&\buildrel\Delta\over =\lim_{\epsilon,\delta\to 0^+}\int_{(\R_+\times\R^d)^{2n}}f(s_1,\cdots, s_{2n}, x_1,\cdots, x_{2n})
	\bigg(\prod_{(j,k)\in {\cal D}}\gamma_\delta^0(s_j-s_k)\gamma_{\epsilon}(x_j-x_k)\bigg)\nonumber\\
	&\times ds_1\cdots ds_{2n}dx_1\cdots dx_{2n}\nonumber
\end{align}
whenever the limit exists.

To end this section   we take the  chance to address an inconvenient fact from (\ref{M-12}) where $G(t,x)$ is
defined as a measure rather than a function in 3-dimensional Euclidean space. In this case,  the integrand
$g_n(\cdot, t, x)$ introduced in (\ref{M-14}) exists only as a generalized function. For any $\eta>0$,
set 
$$
G_\eta(t,x)=\int_{\R^3}p_3(\eta, x-y)G(t,y)dy
$$
where $p_3(t,x)$ is the Brownian semi-group on $\R^3$. Write
$$
\begin{aligned}
&g_{n,\eta}(s_1,\cdots, s_n, x_1,\cdots, x_n, t, x)\\
&=\Big(G_\eta(t-s_n, x-x_n)\cdots G_\eta(s_2-s_1, x_2-x_1)\Big)
1_{[0, t]<^n}(s_1,\cdots, s_n).
\end{aligned}
$$
The $n$-multiple Stratonovich
integral is defined in the following two steps: 
$$
\begin{aligned}
&\int_{(\R_+\times\R^d)^n} g_{n}(s_1,\cdots, s_n, x_1,\cdots, x_n, t, x)\bigg(\prod_{k=1}^n
\dot{W}_{\epsilon,\delta}(s_k, x_k)\bigg)ds_1\cdots ds_n dx_1\cdots dx_n\\
&:=\lim_{\eta\to 0^+}\int_{(\R_+\times\R^d)^n} g_{n,\eta}(s_1,\cdots, s_n, x_1,\cdots, x_n, t, x)\bigg(\prod_{k=1}^n
\dot{W}_{\epsilon,\delta}(s_k, x_k)\bigg)ds_1\cdots ds_n dx_1\cdots dx_n
\end{aligned}
$$
for any $\epsilon, \delta>0$, and
$$
\begin{aligned}
&S\big(g_n(\cdot, t, x)\big):=\int_{(\R_+\times\R^d)^n}g_{n}(s_1,\cdots, s_n, x_1,\cdots, x_n, t, x)
W(ds_1,x_1)\cdots W(s_n, x_n)\\
&:=\lim_{\epsilon,\delta\to 0^+}
\int_{(\R_+\times\R^d)^n} g_{n}(s_1,\cdots, s_n, x_1,\cdots, x_n, t, x)\bigg(\prod_{k=1}^n
\dot{W}_{\epsilon,\delta}(s_k, x_k)\bigg)ds_1\cdots ds_n dx_1\cdots dx_n
\end{aligned}
$$
whenever both limits exist in ${\cal L}^2(\Omega, {\cal A}, \P)$.
It will be verified  later that   the limits do exist
under the assumption (\ref{intro-5}). Therefore, it can be treated  together with dimensions $d=1,2$.

\section{Stratonovich  moment bound}\label{bound}

Let the pair partition ${\cal D}\in \Pi_n$ be fixed. Let $\mu_{j,k}(d\xi)$ ($(j,k)\in{\cal D}$) be the finite measures on $\R^d$ such that
\begin{align}\label{D-1}
\gamma_{j,k}(x)\equiv\int_{\R^d}e^{i\xi\cdot x}\mu_{j,k}(d\xi)\ge 0\hskip.2in x\in\R^d.
\end{align}
Since $\mu_{j,k}$ are finite, $\gamma_{j,k}(x)$ are defined point-wise.

Let $A_{j,k}\subset \R_+$ ($(j,k)\in{\cal D}$) be measurable and define
\begin{align}\label{D-2}
\gamma^0_{j,k}(u)=\Gamma(\alpha_0)^{-1}\int_{A_{j,k}}e^{-\lambda \vert u\vert}{d\lambda\over \lambda^{1-\alpha_0}}\hskip.2in
u\in\R.
\end{align}
Clearly, $\gamma_{j,k}^0(\cdot)$ is non-negative and non-negative definite (see the trick played in (\ref{W-4}) for the second claim).

For any non-negative function ${\cal G}(t,x)$ on $\R_+\times\R^d$,
 we introduce the notations
\begin{align}\label{D-3}
\|{\cal G}\|^{(0)}\equiv\int_{\R_+\times\R^d}{\cal G}(t,x)dtdx\hskip.2in (\hbox{It is $(j,k)$-index free}!),
\end{align}
\begin{align}\label{D-4}
\|{\cal G}\|_{j,k}^{(1)}\equiv\Gamma(\alpha_0)^{-1}\int_{A_{j,k}\times\R^d}
\bigg\vert \int_{\R_+\times\R^d}e^{-\lambda t+i\xi\cdot x}{\cal G}(t,x)dtdx
\bigg\vert{d\lambda\over \lambda^{1-\alpha_0}}\mu_{j,k}(d\xi),
\end{align}
\begin{align}\label{D-5}
\|{\cal G}\|_{j,k}^{(2)}\equiv\bigg(\int_{(\R_+\times\R^d)^2}\gamma_{j,k}^0(s-t)
\gamma_{j,k}(x-y)
{\cal G}(s,x){\cal G}(t,x)dsdxdtdy\bigg)^{1/2}.
\end{align}

The aim of this section is to provide a meaningful bound for the multiple integral
$$
\begin{aligned}
&\int_{(\R_+)_<^{n_1}\times(\R_+)_<^{n_2}}ds_1\cdots ds_{2n}\int_{(\R^d)^{n_1}\times(\R^d)^{n_2}} dx_1\cdots dx_{2n}\\
&\times\bigg(\prod_{\rho=1}^2\prod_{l\in I_\rho}G_l(s_{l}-s_{l-1}, x_{l}-x_{l-1})\bigg)
\prod_{(j,k)\in{\cal D}}\gamma^0_{j,k}(s_j-s_k)\gamma_{j,k}(x_j-x_k)
\end{aligned}
$$
with the non-negative functions $G_l(t,x)$ ($l=1,2\cdots, 2n$) such that
$\|G_l\|^{(0)}, \|G\|_{j,k}^{(1)}, \|G\|_{j,k}^{(2)}<\infty$, where $n_1$ and $n_2$ are non-negative integers with
$2n=n_1+n_2$, $I_1=\{1,\cdots, n_1\}$ and $I_2=\{n_1+1,\cdot, 2n\}$, and where  the following conventions
are adopted: $s_0=0$, $x_0=0$,  and $s_{n_1}=0$, $x_{n_1}=0$ in the expression
$G_{n_1+1}(s_{n_1+1}-s_{n_1}, x_{n_1+1}-x_{n_1})$.

\begin{theorem}\label{th-3} There is
a partition $\{Q_0, Q_1, Q_2\}$ of $\{1,\cdots, 2n\}$, possibly depending on $n_1, n_2$ and ${\cal D}$,
such that $\#(Q_2)$ is even and $\#(Q_0)=\#(Q_1)$ and that
\begin{align}\label{D-6}
&\int_{(\R_+)_<^{n_1}\times(\R_+)_<^{n_2}}
ds_1\cdots ds_{2n}\int_{(\R^d)^{n_1}\times(\R^d)^{n_2}} dx_1\cdots dx_{2n}\\
&\times\bigg(\prod_{\rho=1}^2\prod_{l\in I_\rho}G_l(s_l-s_{l-1}, x_l-x_{l-1})\bigg)
\prod_{(j,k)\in{\cal D}}\gamma^0_{j,k}(s_j-s_k)\gamma_{j,k}(x_j-x_k)\nonumber\\
&\le\bigg(\prod_{l\in Q_0}\|G_l\|^{(0)}\bigg)\bigg(\prod_{l\in Q_1}2\|G_l\|_{(\cdot,\cdot)}^{(1)}\bigg)
\bigg(\prod_{l\in Q_2}\|G_l\|_{(\cdot,\cdot)}^{(2)}\bigg)\nonumber
\end{align}
where the subscripts $\{(j,k)\in {\cal D}\}$ are distributed  between $Q_1$-product and $Q_2$-product in the following way:
Each $(j,k)$ in $Q_1$ appears exact once (so the number of $(j,k)$ in $Q_1$-product is equal to $\#(Q_1)$, while 
each index $(j,k)$ in $Q_2$-product appears exactly twice (so the number of $(j,k)$ in  $Q_2$
is equal to $2^{-1}\#(Q_2)$). 
\end{theorem}

\proof We carry out argument by induction on $n$: When $n=1$, there are three possible forms for the left hand:
"$n_1=n_2=1$", "$n_1=2$ and $n_2=0$" or "$n_1=0$ and $n_2=2$".

When $n_1=n_2=1$, the left hand of (\ref{D-6}) is
$$
\int_{\R_+^2}ds_1ds_2\int_{(\R^d)^{2}}dxdy G_1(s_1, x)G_2(s_2,y)\gamma^0_{1,2}(s_1-s_2)
\gamma_{1,2}(x-y)\le\|G_1\|_{1,2}^{(2)}\|G_2\|_{1,2}^{\{2\}}
$$
where the inequality follows from Cauchy-Schwartz's inequality. So the claim holds with $Q_0=Q_1=\phi$ and
$Q_2=\{1,2\}$.

\medskip

When $n_1=2$ and $n_2=0$, the left is
$$
\begin{aligned}
&\int\!\int_{\{s_1<s_2\}}ds_1ds_2\int_{(\R^d)^{2}}dxdy G_1(s_1, x)G_2(s_2-s_1, y-x)
\gamma_{j,k}^0(s_1-s_2)\gamma_{1,2}(x-y)\\
&=\bigg(\int_{\R_+\times\R^d}G_1(s_1,x)ds_1dx\bigg)\bigg(\int_{\R_+\times\R^d}G_2(t,y)
\gamma_{1,2}^0(t)
\gamma_{1,2}(y)dtdy\bigg).
\end{aligned}
$$
Notice
$$
\begin{aligned}
&\int_{\R_+\times\R^d}G_2(t,y)\gamma_{1,2}^0(t)
\gamma_{j,k}(y)dtdy\cr
&=\Gamma(\alpha_0)^{-1}\int_{\R_+\times\R^d}G_2(t,y)\bigg(\int_{A_{1,2}\times\R^d}e^{-\lambda t+i\xi\cdot y} 
{d\lambda\over\lambda^{1-\alpha_0}}\mu_{1,2}(d\xi)\bigg)dtdy\\
&=\Gamma(\alpha_0)^{-1}\int_{A_{1,2}\times\R^d}\bigg(\int_{\R_+\times\R^d}G_2(t,y)e^{-\lambda t+i\xi\cdot y} dtdy\bigg)
{d\lambda\over\lambda^{1-\alpha_0}}\mu_{1,2}(d\xi)\le \|G_2\|_{1,2}^{(1)}.
\end{aligned}
$$
In summary
$$
\begin{aligned}
&\int\!\int_{\{s_1<s_2\}}ds_1ds_2\int_{(\R^d)^{2}}dxdy G_1(s_1, x)G_2(s_2-s_1, y-x)\vert s_1-s_2\vert^{-\alpha_0}\gamma(x-y)\\
&\le \|G_1\|^{(0)}\|G_2\|_{1,2}^{(1)}.
\end{aligned}
$$
Thus, the claim holds with $Q_0=\{1\}$, $Q_1=\{2\}$ and $Q_2=\phi$.

Similarly,  when $n_1=0$ and $n_2=2$,
the claim holds with the bound $\|G_1\|_{1,2}^{(1)}\|G_2\|^{(0)}$
and $Q_0=\{2\}$, $Q_1=\{1\}$ and $Q_2=\phi$.

\bigskip

Assume the claim holds for $n-1$. 
We now verified it for $n$. Assume 
that $j_0$ and $k_0$ are paired with $2n$ and $n_1$, respectively, i.e.
$(j_0, 2n), (k_0, n_1)\in {\cal D}$. The idea is to separate 
$$
G_{2n}(s_{2n}-s_{2n-1}, x_{2n}-x_{2n-1})\gamma^0_{j_0, 2n}(s_{j_0}-s_{2n})\gamma_{j_0,2n}(x_{j_0}-x_{2n})
$$
or
$$
G_{n_1}(s_{n_1}-s_{n_1-1}, x_{n_1}-x_{n_1-1})
\gamma^0_{k_0, n_1}(s_{k_0}-s_{n_1})\gamma_{k_0,n_1}(x_{k_0}-x_{n_1}),
$$
whichever possible, from the multiple time-space integral.

We consider the following three possible cases: Case 1:
$j_0\in I_2$ or $k_0\in I_1$.
In other words,  at least one of $n_1$ and $2n$ has a
domestic pair. In the remaining settings,
both $n_1$ and $2n$ have inter-group pairs. We shall treat it in the following
two different cases: Case 2:
$j_0=n_1$ and $k_0=2n$, i.e., $(n_1, 2n)\in{\cal D}$; and Case 3:
$1\le j_0<n_1<k_0<2n$.

\medskip

{\bf Case 1}: We actually claim a better bound
\begin{align}\label{D-7}
&\int_{(\R_+)_<^{n_1}\times(\R_+)_<^{n_2}}
ds_1\cdots ds_{2n}\int_{(\R^d)^{n_1}\times(\R^d)^{n_2}} dx_1\cdots dx_{2n}\\
&\times\bigg(\prod_{\rho=1}^2\prod_{l\in I_\rho}G_l(s_l-s_{l-1}, x_l-x_{l-1})\bigg)
\prod_{(j,k)\in{\cal D}}\gamma^0_{j,k}(s_j-s_k)\gamma_{j,k}(x_j-x_k)\nonumber\\
&\le{1\over 2}\bigg(\prod_{l\in Q_0}\|G_l\|^{(0)}\bigg)\bigg(\prod_{l\in Q_1}2\|G_l\|_{(\cdot,\cdot)}^{(1)}\bigg)
\bigg(\prod_{l\in Q_2}\|G_l\|_{(\cdot,\cdot)}^{(2)}\bigg)\nonumber
\end{align}
in this case for the argument needed in Case 3.

Due to similarity, we only consider the case $j_0\in I_2$.  
For $s_{j_0}\le s_{2n-1}\le s_{2n}$,
$$
\gamma^0_{j_0, 2n}(s_{2n}-s_{j_0})\le\gamma_{j_0, 2n}^0(s_{2n}-s_{2n-1}).
$$
Thus, the left hand side of (\ref{D-7}) yields to the bound
$$
\begin{aligned}
&\int_{(\R_+)_<^{n_1}\times(\R_+)_<^{n_2-1}}ds_1\cdots ds_{2n-1}\int_{(\R^d)^{n_1}\times(\R^d)^{n_2-1}} dx_1\cdots dx_{2n-1}\\
&\times\bigg(\prod_{\rho=1}^2\prod_{l\in I'_\rho}G_l(s_l-s_{l-1}, x_l-x_{l-1})\bigg)
\bigg(\prod_{(j,k)\in{\cal D}'}\gamma^0_{j,k}(s_j-s_k)\gamma_{j,k}(x_j-x_k)\bigg)
\int_{s_{2n-1}}^\infty ds_{2n}\\
&\times\int_{\R^d}dx_{2n}G_{2n}(s_{2n}-s_{2n-1}, x_{2n}-x_{2n-1})\gamma_{j_0, 2n}^0(s_{2n}-s_{2n-1})\gamma_{j_0, 2n}(x_{2n}-x_{j_0})
\end{aligned}
$$
where $I_1'=I_1$ and $I_2'=I_2\setminus \{2n\}=\{n_1+1,\cdots, 2n-1\}$
and ${\cal D}'={\cal D}\setminus (j_0,2n)\in \Pi_{n-1}$.
Notice
$$
\begin{aligned}
&\int_{s_{2n-1}}^\infty ds_{2n}\int_{\R^d}dx_{2n}G_{2n}(s_{2n}-s_{2n-1}, x_{2n}-x_{2n-1})\gamma^0_{j_0,2n}(s_{2n}-s_{2n-1})\gamma_{j_0,2n}(x_{2n}-x_{j_0})\\
&=\int_0^\infty ds\int_{\R^d}dx_{2n}G_{2n}(s,  x_{2n}-x_{2n-1})\gamma^0_{j_0,2n}(s)
\int_{\R^d}e^{i\xi\cdot(x_{2n}-x_{j_0})}
\mu_{j_0, 2n}(d\xi)\\
&=\Gamma(\alpha_0)^{-1}\int_{A_{j_0, 2n}\times\R^d}{d\lambda\over\lambda^{1-\alpha_0}}\mu(d\xi)\exp\Big\{i\xi\cdot (x_{2n-1}-x_{j_0})\Big\}\int_{\R_+\times\R^d}G_{2n}(s,x)e^{-\lambda s+i\xi\cdot x}dsdx\\
&\le \Gamma(\alpha_0)^{-1}\int_{A_{j_0,2n}\times\R^d}{d\lambda\over\lambda^{1-\alpha_0}}\mu(d\xi)\bigg\vert\int_{\R_+\times\R^d}G_{2n}(s,x)e^{-\lambda s+i\xi\cdot x}dsdx\bigg\vert=\|G_{2n}\|_{j_0,2n}^{(1)}.
\end{aligned}
$$
In summary, the left hand of (\ref{D-7}) yields to the bound
\begin{align}\label{D-8}
&\|G_{2n}\|_{j_0,2n}^{(1)} \int_{(\R_+)_<^{n_1}\times(\R_+)_<^{n_2-1}}ds_1\cdots ds_{2n-1}\int_{(\R^d)^{n_1}\times(\R^d)^{n_2-1}} dx_1\cdots dx_{2n-1}\nonumber\\
&\times\bigg(\prod_{\rho=1}^2\prod_{l\in I'_\rho}G_l(s_l-s_{l-1}, x_l-x_{l-1})\bigg)
\bigg(\prod_{(j,k)\in{\cal D}'}\gamma_{j,k}^0(s_j-s_k)\gamma_{j,k}(x_j-x_k)\bigg).
\end{align}
Denote $\tilde{I}_1=I_1$, $\tilde{I}_2=I_2\setminus \{j_0, 2n\}$. When $j_0=2n-1$,
the right hand side is equal to
$$
\begin{aligned}
&\|G_{2n}\|_{j_0, 2n}^{(1)}\|G_{2n-1}\|^{(0)}\int_{(\R_+)_<^{n_1}\times(\R_+)_<^{n_2-2}}ds_1\cdots ds_{2n-2}\int_{(\R^d)^{n_1}\times(\R^d)^{n_2-2}} dx_1\cdots dx_{2n-2}\\
&\times\bigg(\prod_{\rho=1}^2\prod_{l\in I'_\rho}G_l(s_l-s_{l-1}, x_l-x_{l-1})\bigg)
\bigg(\prod_{(j,k)\in{\cal D}'}\gamma_{j,k}^0(s_j-s_k)\gamma_{j,k}(x_j-x_k)\bigg).
\end{aligned}
$$
Applying the induction assumption, we yield the bound
$$
\begin{aligned}
&\|G_{2n}\|_{j_0,2n}^{(1)}\|G_{2n-1}\|^{(0)}\bigg(\prod_{l\in \tilde{Q}_0}\|G_l\|^{(0)}\bigg)
\bigg(\prod_{l\in \tilde{Q}_1}2\|G_l\|_{(\cdot, \cdot)}^{(1)}\bigg)\bigg(\prod_{l\in \tilde{Q}_2}
\|G_l\|_{(\cdot,\cdot)}^{(2)}\bigg)\\
&={1\over 2}\bigg(\prod_{l\in Q_0}\|G_l\|^{(0)}\bigg)\bigg(\prod_{l\in Q_1}2\|G_l\|_{(\cdot,\cdot)}^{(1)}\bigg)
\bigg(\prod_{l\in Q_2}\|G_l\|_{(\cdot,\cdot)}^{(2)}\bigg)\nonumber
\end{aligned}
$$
where $\tilde{Q}_0$, $\tilde{Q}_1$, $\tilde{Q}_2$ form a partition of $\{1,\cdots, 2(n-1)\}$
obeying the rule described by Theorem \ref{th-3}. Setting
$Q_0=\tilde{Q}_0\cup\{2n-1\}$,  $Q_1=\tilde{Q}_1\cup\{2n\}$ and $Q_2=\tilde{Q}_2$
we have proved (\ref{D-7}).
\medskip

We now assume $j_0<2n-1$. Set
$$
d\tilde{\bf s}=ds_1\cdots ds_{j_0-1}ds_{j_0+1}\cdots d_{2n-1}\hskip.1in\hbox{and}
\hskip.1in d\tilde{\bf x}=dx_1\cdots dx_{j_0-1}dx_{j_0+1}\cdots d_{2n-1}.
$$
The bound in (\ref{D-8}) can be written as
$$
\begin{aligned}
&\|G_{2n}\|_{j_0,2n}^{(1)}\int_{(\R_+)_<^{n_1}\times(\R_+)_<^{n_2-2}}d\tilde{\bf s}\int_{(\R^d)^{n_1}\times(\R^d)^{n_2-2}} 
d\tilde{\bf x}\bigg(\prod_{\rho=1}^2
\prod_{l\in \tilde{I}_\rho}G_l(s_l-s_{l-1}, x_l-x_{l-1})\bigg)\\
&\times\bigg(\int_{s_{j_0-1}}^{s_{j_0+1}}
\int_{\R^d}
G_{j_0}(s_{j_0}-s_{j_0-1}, x_{j_0}-x_{j_0-1})G_{j_0+1}(s_{j_0+1}-s_{j_0}, x_{j_0+1}-x_{j_01})dx_{j_0}ds_{j_0}\bigg)\cr
&\times\bigg(\prod_{(j,k)\in{\cal D}'}\gamma^0_{j,k}(s_j-s_k)\gamma_{j,k}(x_j-x_k)\bigg).
\end{aligned}
$$
Notice
$$
\begin{aligned}
&\int_{s_{j_0-1}}^{s_{j_0+1}}
\int_{\R^d}
G_{j_0}(s_{j_0}-s_{j_0-1}, x_{j_0}-x_{j_0-1})G_{j_0+1}(s_{j_0+1}-s_{j_0}, x_{j_0+1}-x_{j_0})dx_{j_0}ds_{j_0}\\
&=\tilde{G}_{j_0, j_0+1}(s_{j_0+1}-s_{j_0-1}, x_{j_0+1}-x_{j_0-1})
\end{aligned}
$$
where
$$
\tilde{G}_{j_0, j_0+1}(t,x)=\int_0^t\!\!\int_{\R^d}G_{j_0}(s, y)G_{j_0+1}(t-s, x-y)dyds.
$$
Therefore, by (\ref{D-8})
$$
\begin{aligned}
&\int_{(\R_+)_<^{n_1}\times(\R_+)_<^{n_2}}ds_1\cdots ds_{2n}
\int_{(\R^d)^{n_1}\times(\R^d)^{n_2}} dx_1\cdots dx_{2n}\cr
&\times\bigg(\prod_{\rho=1}^2\prod_{l\in I_\rho}G_l(s_l-s_{l-1}, x_l-x_{l-1})\bigg)
\prod_{(j,k)\in{\cal D}}\gamma_{j,k}^0(s_j-s_k)\gamma_{j,k}(x_j-x_k)\\
&\le\|G_{2n}\|_{j_0, 2n}^{(1)} \int_{(\R_+)_<^{n_1}\times(\R_+)_<^{n_2-2}}d\tilde{\bf s}\int_{(\R^d)^{n_1}\times(\R^d)^{n_2-2}} 
d\tilde{\bf x}\bigg(\prod_{\rho=1}^2
\prod_{l\in \tilde{I}_\rho}G_l(s_l-s_{l-1}, x_l-x_{l-1})\bigg)\\
&\times\tilde{G}_{j_0, j_0+1}(s_{j_0+1}-s_{j_0-1}, x_{j_0+1}-x_{j_0-1})
\bigg(\prod_{(j,k)\in{\cal D}'}\gamma_{j,k}^0s_j-s_k)\gamma_{j,k}(x_j-x_k)\bigg).
\end{aligned}
$$
Applying the induction assumption with the functions ($2(n-1)$ of them)
$$
G_1,\cdots, G_{j_0-1}, \tilde{G}_{j_0, j_0+1}, 
G_{j_0+2},\cdots, G_{2n-1}
$$
with ${\cal D}'={\cal D}\setminus\{j_0,2n\}$ and with $2(n-1)=n_1+(n_2-2)$, we have one of the three possible
bounds:
\begin{align}\label{D-9}
&\|G_{2n}\|_{j_0, 2n}^{(1)}\|\tilde{G}_{j_0, j_0+1}\|_{(\cdot, \cdot)}^{(i)}
\bigg(\prod_{l\in \tilde{Q}^{(0)}}\|G_l\|^{(0)}\bigg)\bigg(\prod_{l\in \tilde{Q}_1}2\|G_l\|_{(\cdot, \cdot)}^{(1)}\bigg)
\bigg(\prod_{l\in \tilde{Q}_2}\|G_l\|_{(\cdot, \cdot)}^{(2)}\bigg)\\
&={1\over 2}\|\tilde{G}_{j_0, j_0+1}\|_{(\cdot, \cdot)}^{(i)}
\bigg(\prod_{l\in \tilde{Q}^{(0)}}\|G_l\|^{(0)}\bigg)\bigg(\prod_{l\in \tilde{Q}_1\cup\{2n\}}2\|G_l\|_{(\cdot, \cdot)}^{(1)}\bigg)
\bigg(\prod_{l\in \tilde{Q}_2}\|G_l\|_{(\cdot, \cdot)}^{(2)}\bigg)
\hskip.2in i=0, 1,2\nonumber
\end{align}
where
$\tilde{Q}_0$, $\tilde{Q}_1$, $\tilde{Q}_2$ form a partition of $\{1,\cdots, 2n-1\}\setminus\{j_0, j_0+1\}$.
When $i=0$, $\|\tilde{G}_{j_0, j_0+1}\|_{(\cdot, \cdot)}^{(0)}=\|\tilde{G}_{j_0, j_0+1}\|^{(0)}$ 
is free of $(j,k)$-index.
As $i=1, 2$,
$\|\tilde{G}_{j_0, j_0+1}\|^{(i)}=\|\tilde{G}_{j_0, j_0+1}\|^{(i)}_{j_1,k_1}$ for some 
$(j_1, k_1)\in{\cal D}'$ (recall that ${\cal D}'$ is a pair partition on $\{1,\cdots, 2n\}\setminus\{j_0,2n\}$).

\medskip
When $i=0$, by induction assumption $1+\#(\tilde{Q}_0)=\#(\tilde{Q}_1)$, $\#(\tilde{Q}_2)$ is even.
Recall that all $(j,k)\in {\cal D}'$ have been assigned into
$\tilde{Q}_1$ and $\tilde{Q}_2$ according to the statement of Theorem \ref{th-3}.
In particular, the number of $(j,k)\in {\cal D}'$ in $\tilde{Q}_1$- product is $\#(\tilde{Q}_1)$
and the number of $(j,k)\in {\cal D}'$ in $\tilde{Q}_2$- product is $2^{-1}\#(\tilde{Q}_2)$.
Further
$$
\|\tilde{G}_{j_0, j_0+1}\|_0=\|G_{j_0}\|^{(0)}\|G_{j_0+1}\|^{(0)}.
$$
The bound (\ref{D-7}) has been verified with $Q_0=\tilde{Q}_0\cap\{j_0, j_0+1\}$, $Q_1=\tilde{Q}_1\cup\{2n\}$
and $Q_2=\tilde{Q}_2$.

\medskip
When $i=1$, $\#(\tilde{Q}_0)=1+\#(\tilde{Q}_1)$, $\#(\tilde{Q}_2)$ is even. 
Recall that $\|\tilde{G}_{j_0, j_0+1}\|^{(1)}=\|\tilde{G}_{j_0, j_0+1}\|^{(1)}_{j_1,k_1}$ for some 
$(j_1, k_1)\in{\cal D}'={\cal D}\setminus\{j_0, 2n\}$.
All $(j,k)\in {\cal D}'\setminus\{(j_1,k_1)\}$ have been assigned into
$\tilde{Q}_1$ and $\tilde{Q}_2$  in a way that $\#(\tilde{Q}_1)-1$ of them
are in $\tilde{Q}_1$- product and remaining of them (the number is $2^{-1}\#(\tilde{Q}_2)$)
are in $\tilde{Q}_2$-product (each of them appeasers twice).
$$
\begin{aligned}
\|\tilde{G}_{j_0, j_0+1}\|_{j_1.k_1}^{(1)}
&=\Gamma(\alpha_0)^{-1}\int_{A_{j_1, k_1}\times\R^d}
\bigg\vert \int_{\R_+\times\R^d} G_{j_0}(t,x)e^{-\lambda t+i\xi\cdot x}dtdx\bigg\vert\cr
&\times\bigg\vert \int_{\R_+\times\R^d} G_{j_0+1}(t,x)e^{-\lambda t+i\xi\cdot x}dtdx\bigg\vert{d\lambda\over\lambda^{1-\alpha_0}}\mu(d\xi)\\
&\le\bigg\{\Gamma(\alpha_0)^{-1}\int_{A_{j_1, k_1}\times\R^d}
\bigg\vert \int_{\R_+\times\R^d} G_{j_0}(t,x)e^{-\lambda t+i\xi\cdot x}dtdx\bigg\vert^2{d\lambda\over\lambda^{1-\alpha_0}}\mu(d\xi)\bigg\}^{1/2}\\
&\times\bigg\{\Gamma(\alpha_0)^{-1}\int_{A_{j_1, k_1}\times\R^d}
\bigg\vert \int_{\R_+\times\R^d} G_{j_0+1}(t,x)e^{-\lambda t+i\xi\cdot x}dtdx\bigg\vert^2{d\lambda\over\lambda^{1-\alpha_0}}\mu(d\xi)\bigg\}^{1/2}.
\end{aligned}
$$
Notice
$$
\begin{aligned}
&\Gamma(\alpha_0)^{-1}\int_{A_{j_1, k_1}\times\R^d}
\bigg\vert \int_{\R_+\times\R^d} G_{j_0}(t,x)e^{-\lambda t+i\xi\cdot x}dtdx\bigg\vert^2{d\lambda\over\lambda^{1-\alpha_0}}\mu(d\xi)\\
&=\int_{(\R_+\times\R^d)^2}\gamma^0_{j_1,k_1}(s+t)\gamma_{j,k}(x-y)G_{j_0}(s,x)G_{j_0}(t,x)dsdtdxdy\\
&\le \big(\|G_{j_0}\|_{j_1,k_1}^{(2)}\big)^2.
\end{aligned}
$$
Same thing happens to $G_{j_0+1}$-factor. Thus,
$$
\|\tilde{G}_{j_0, j_0+1}\|_{j_1, k_1}^{(1)}\le \|G_{j_0}\|_{j_1,k_1}^{(2)}
\|G_{j_0+1}\|_{j_1, k_1}^{(2)}.
$$
So (\ref{D-7}) has been varified with $Q_0=\tilde{Q}_0$, $Q_1=\tilde{Q}_1\cup\{2n\}$
and $Q_2=\tilde{Q}_2\cup\{j_0, j_0+1\}$.

When $i=2$, $\#(\tilde{Q}_0)=\#(\tilde{Q}_1)$, $\#(\tilde{Q}_2)$ is odd.
The the pairs $(j,k)\in {\cal D}'$ are distributed in a way that
$\#(\tilde{Q}_1)$ of them are in $\tilde{Q}_1$-product and remaining of them
are in $\tilde{Q}_2$-product. All $(j,k)$ in $\tilde{Q}_2$-product appears twice
except $(j_1,k_1)$ (which appears once). Consequently,
the number of $(j,k)\in {\cal D}'$ in $\tilde{Q}_1$ is $\#(\tilde{Q}_1)$
and the number of $(j,k)\in {\cal D}'$ in $\tilde{Q}_2$  is $2^{-1}(\#(\tilde{Q}_2)+1)$.
Further
$$
\begin{aligned}
&\|\tilde{G}_{j_0, j_0+1}\|_{j_1, k_1}^{(2)}\\
&=\bigg\{\int_{(\R_+\times\R^d)^2}\gamma^0_{j_1, k_1}(s-r)\gamma_{j_1,k_1}(x-y)
\tilde{G}_{j_0, j_0+1}(s,x)\tilde{G}_{j_0, j_0+1}(r,y)dsdxdrdy\bigg\}^{1/2}\\
&=\bigg\{\int_{\R^{d+1}}\bigg\vert\int_0^\infty\!\!\int_{\R^d}e^{i\lambda t+i\xi\cdot x}
\tilde{G}_{j_0, j_0+1}(t,x)dxdt\bigg\vert^2\mu_{j_1,k_1}^0(d\lambda)
\mu_{j,k}(d\xi)\bigg\}^{1/2}
\end{aligned}
$$
where $\mu_{j_1,k_1}^0(d\lambda)$ is the spectral measure of $\gamma_{j_1, k_1}^0(\cdot)$
(Recall that $\gamma_{j_1, k_1}^0(\cdot)$ is non-negative-definite).

Notice
$$
\begin{aligned}
&\bigg\vert\int_0^\infty\!\!\int_{\R^d}e^{i\lambda t+i\xi\cdot x}
\tilde{G}_{j_0, j_0+1}(t,x)dxdt\bigg\vert\\
&=\bigg\vert\bigg(\int_0^\infty\!\!\int_{\R^d}e^{i\lambda t+i\xi\cdot x}
G_{j_0}(t,x)dxdt\bigg)\bigg(\int_0^\infty\!\!\int_{\R^d}e^{i\lambda t+i\xi\cdot x}
G_{j_0+1}(t,x)dxdt\bigg)\bigg\vert\\
&\le\|G_{j_0}\|^{(0)}\bigg\vert\int_0^\infty\!\!\int_{\R^d}e^{i\lambda t+i\xi\cdot x}
G_{j_0+1}(t,x)dxdt\bigg\vert.
\end{aligned}
$$
Therefore,
$$
\begin{aligned}
&\|\tilde{G}_{j_0, j_0+1}\|_{j_1, k_1}^{(2)}\cr
&\le \|G_{j_0}\|^{(0)}
\bigg\{\int_{\R^{d+1}}\bigg\vert\int_0^\infty\!\!\int_{\R^d}e^{i\lambda t+i\xi\cdot x}
G_{j_0+1}(t,x)dxdt\bigg\vert^2\mu_{j_1, k_1}^0(d\lambda)
\mu_{j_1, k_1}(d\xi)\bigg\}^{1/2}\\
&=\|G_{j_0}\|^{(0)}\|G_{j_0+1}\|_{j_1, k_1}^{(2)}.
\end{aligned}
$$
The bound (\ref{D-7}) has been varified with $Q_0=\tilde{Q}_0\cup\{j_0\}$, $Q_1=\tilde{Q}_1\cup\{2n\}$
and $Q_2=\tilde{Q}_2\cup\{j_0+1\}$.

\medskip

{\bf Case 2}. Write $\tilde{I}_1=\{1,\cdots, n_1-1\}$, $\tilde{I}_2=\{n_1+1,\cdots, 2n-1\}$ and
$$
d\tilde{\bf s}=ds_1\cdots ds_{n_1-1}ds_{n_1+1}\cdots ds_{2n-1}\hskip.1in\hbox{and}\hskip.1in
d\tilde{\bf x}=dx_1\cdots dx_{n_1-1}dx_{n_1+1}\cdots dx_{2n-1}.
$$

Since $(n_1, 2n)\in{\cal D}$ in this case, the left hand side of (\ref{D-6}) is equal to
$$
\begin{aligned}
&\int_{(\R_+)_<^{n_1-1}\times(\R_+)_<^{n_2-1}}d\tilde{\bf s}\int_{(\R^d)^{n_1-1}\times
(\R^d)^{n_2-1}} d\tilde{\bf x}\\
&\times\bigg(\prod_{\rho=1}^2\prod_{l\in \tilde{I}_\rho}G_l(s_l-s_{l-1}, x_l-x_{l-1})\bigg)
\prod_{(j,k)\in\tilde{\cal D}}\gamma_{j,k}^0 s_j-s_k)\gamma_{j,k}(x_j-x_k)\cr
&\times\int_{s_{n_1-1}}^\infty\int_{s_{2n-1}}^\infty\int_{\R^d\times\R^d}ds_{n_1}
ds_{2n}dx_{n_1}dx_{2n}\gamma_{n_1,2n}^0 (s_{2n}-s_{n_1})
\gamma_{n_1, 2n}(x_{2n}-x_{n_1})\\
&\times G_{n_1}(s_{n_1}-s_{n_1-1}, x_{n_1}-x_{n_1-1})G_{2n}(s_{2n}-s_{2n-1}, x_{2n}-x_{2n-1}).
\end{aligned}
$$
where $\tilde{D}={\cal D}\setminus\{n_1, 2n\}$. Notice
$$
\begin{aligned}
&\int_{s_{n_1-1}}^\infty\int_{s_{2n-1}}^\infty\int_{\R^d\times\R^d}ds_{n_1}
ds_{2n}dx_{n_1}dx_{2n}\gamma_{n_1, 2n}^0 (s_{2n}-s_{n_1})\gamma_{n_1, 2n}(x_{2n}-x_{n_1})\\
&\times G_{n_1}(s_{n_1}-s_{n_1-1}, x_{n_1}-x_{n_1-1})G_{2n}(s_{2n}-s_{2n-1}, x_{2n}-x_{2n-1})\cr
&=\int_0^\infty\int_0^\infty\int_{\R^d\times\R^d}dsdrdxdy
\gamma_{n_1, 2n}^0 (s-r)\gamma_{n_1, 2n}(x-y)
G_{n_1}(s, x-x_{n_1-1})G_{2n}(r, y-x_{2n-1})\\
&\le \|G_{n_1}\|_{n_1,2n}^{(2)}\|G_{2n}\|_{n_1, 2n}^{(2)}
\end{aligned}
$$
where last step follows from Cauchy-Schwartz inequality and then shift-invariance of the space variables.

By the induction assumption, we have the bound
$$
 \|G_{n_1}\|_{n_1,2n}^{(2)}\|G_{2n}\|_{n_1, 2n}^{(2)}
 \bigg(\prod_{l\in \tilde{Q}_0}\|G_l\|^{(0)}\bigg)\bigg(\prod_{l\in \tilde{Q}_1}2\|G_l\|^{(1)}_{(\cdot, \cdot)}\bigg)
\bigg(\prod_{l\in \tilde{Q}_2}\|G_l\|_{(\cdot,\cdot)}^{(2)}\bigg)
$$
where $\tilde{Q}_0$, $\tilde{Q}_1$, $\tilde{Q}_2$ form a partition of $\{1,\cdots, 2n\}\setminus\{n_1,2n\}$
with $\#(\tilde{Q}_0)=\#(\tilde{Q}_1)$ and $\#(\tilde{Q}_2)$ being even. 
All $(j,k)\in{\cal D}'$ are distributed according to the statement of Theorem \ref{th-3}.
In particular, the number of $(j,k)$ in $\tilde{Q}_1$ is $\#(\tilde{Q}_1)$
and number of of $(j,k)$ in $\tilde{Q}_2$ is $2^{-1}\#(\tilde{Q}_2)$.
The bound(\ref{D-6}) has been varified
with $Q_0=\tilde{Q}_0$, $Q_1=\tilde{Q}_1$ and $Q_2=\tilde{Q}_2\cup\{n_1, 2n\}$.

\medskip

{\bf Case 3}. Recall that this is the case when $(j_0, 2n), (k_0, n_1)\in {\cal D}$ with
$j_0\in I_1$, $k_1\in I_2$ and $j_0\not =n_1$, $k_0\not =2n$.
The idea essentially comes from the treatment in Case 1.
The extra obstacle comes from the fact that there is no restrictive order between $s_{j_0}$ and $s_{2n}$, 
nor between
$s_{k_0}$ and $s_{n_1}$, as $(j_0, 2n)$ and $(k_0, n_1)$ are inter-group pairs.
To fix it, we break the left hand side of (\ref{D-6}) into two parts by the indicators $1_{\{s_{j_0}\le s_{2n-1}\}}$
and $1_{\{s_{j_0}> s_{2n-1}\}}$

On $\{s_{2n-1}\ge s_{j_0}\}$, 
$\gamma^0_{j_0, 2n}(s_{2n}-s_{j_0})\le \gamma^0_{j_0, 2n}(s_{2n}-s_{2n-1})$.
By a strategy same as the one used for (\ref{D-9}),  the first part of the decomposition
 yields the bound in (\ref{D-7}), i.e.,
 $$
 {1\over 2}\bigg(\prod_{l\in Q_0}\|G_l\|^{(0)}\bigg)
\bigg(\prod_{l\in Q_1}2\|G_l\|_{(\cdot,\cdot)}^{(1)}\bigg)\bigg(\prod_{l\in Q_2}\|G_l\|_{(\cdot,\cdot)}^{(2)}
\bigg).
$$

On $\{s_{2n-1}< s_{j_0}\}$, on the other hand,  $s_{n_1-1}\ge s_{j_0}>s_{2n-1}\ge s_{k_0}$. Therefore
$\gamma^0_{n_1, k_0}(s_{n_1}-s_{k_0})\le\gamma^0_{n_1, k_0}(s_{n_1}-s_{n_1-1})$.
Repeating the same procedure (with $2n$ being replaced by
$n_1$), the second part of the decomposition has the bound (\ref{D-7}), i.e.,
 $$
 {1\over 2}\bigg(\prod_{l\in Q'_0}\|G_l\|^{(0)}\bigg)
\bigg(\prod_{l\in Q'_1}2\|G_l\|_{(\cdot,\cdot)}^{(1)}\bigg)\bigg(\prod_{l\in Q'_2}\|G_l\|_{(\cdot,\cdot)}^{(2)}\bigg)
$$
for the partition $Q_1'$, $Q_2'$, $Q_3'$ of $\{1,\cdots, 2n\}$ that meets all requirement given in
Theorem \ref{th-3}.

Therefore, the left hand side of (\ref{D-6}) is less than or equal to
$$
\begin{aligned}
&{1\over 2}\Bigg\{\bigg(\prod_{l\in Q_0}\|G_l\|^{(0)}\bigg)
\bigg(\prod_{l\in Q_1}2\|G_l\|_{(\cdot,\cdot)}^{(1)}\bigg)\bigg(\prod_{l\in Q_2}\|G_l\|_{(\cdot,\cdot)}^{(2)}\bigg)\\
&+\bigg(\prod_{l\in Q'_0}\|G_l\|^{(0)}\bigg)
\bigg(\prod_{l\in Q'_1}2\|G_l\|_{(\cdot,\cdot)}^{(1)}\bigg)
\bigg(\prod_{l\in Q'_2}\|G_l\|_{(\cdot,\cdot)}^{(2)}\bigg)\Bigg\}.
\end{aligned}
$$
Between the partitions $\{Q_0, Q_1, Q_2\}$ and $\{Q'_0, Q'_1, Q'_2\}$, choosing  the one that
produces the larger product completes the induction. \qed

\section{Stratonovich integrability of $g_n(\cdot, t,x)$}\label{S}

Let $n\ge 1$ be fixed and recall that $g_n(\cdot, t,x)$ is defined in (\ref{M-14}).
Through this section, we adopt the following notations:
$$
\epsilon=(\epsilon_1,\cdots,\epsilon_n), \hskip.1in\tilde{\epsilon}=(\epsilon_{n+1},\cdots, \epsilon_{2n})\hskip.1in\hbox{and}\hskip.1in\delta=(\delta_1,\cdots,\delta_n), \hskip.1in\tilde{\delta}=(\tilde{\delta}_{n+1}, \cdots\tilde{\delta}_{2n})
$$
$$
\bar{\epsilon}=(\epsilon, \tilde{\epsilon})=(\epsilon_1,\cdots,\epsilon_{2n})\hskip.1in \hbox{and}\hskip.1in
\bar{\delta}=(\delta, \tilde{\delta})=(\delta_1,\cdots,\delta_{2n})
$$
for $\epsilon_1,\cdots,\epsilon_{2n}, \delta_1,\cdots,\delta_{2n}>0$.  The notation "$\epsilon\to 0^+$
means $\epsilon_1,\cdots,\epsilon_n\to 0^+$. The notations "$\tilde{\epsilon}\to 0^+$", "$\delta\to 0^+$",
"$\tilde{\delta}\to 0^+$", "$\bar{\epsilon}\to 0^+$" and "$\bar{\delta}\to 0^+$" are used in obvious way.

Set
\begin{align}\label{S-1}
S_{n,\epsilon,\delta}\big(g_n(\cdot, t, x)\big)
&=\int_{(\R_+\times\R^d)^n}g_n(s_1,\cdots, s_n, x_1,\cdots, x_n, t, x)\\
&\times\bigg(\prod_{k=1}^n\dot{W}_{\epsilon_k,\delta_k}(s_k, x_k)\bigg)
ds_1\cdots ds_ndx_1\cdots x_n.\nonumber
\end{align}

The main goal of this section is 

\begin{theorem}\label{th-4} Under the condition (\ref{intro-5}), the limit
\begin{align}\label{S-2}
\lim_{\epsilon,\delta\to 0^+}S_{n,\epsilon,\delta}\big(g_n(\cdot, t, x)\big)
\end{align}
exists in ${\cal L}^2(\Omega, {\cal A}, \P)$ for each $n\ge 1$, $t>0$ and $x\in\R^d$.
Consequently, $g_n(\cdot, t, x)$
is $n$-multiple Stratonovich integrable. Further,
\begin{align}\label{S-3}
&\E\Big[S_n\big(g_n(\cdot, t, x)\big)\Big]^2=\sum_{{\cal D}\in\Pi_n}\int_{[0,t]_<^n\times[0,t]_<^n}
\int_{(\R^d)^n\times(\R^d)^n}dx_1\cdots dx_{2n}\\
&\times\bigg(\prod_{\rho=1}^2\prod_{l\in I_\rho}G(s_l-s_{l-1}, x_l-x_{l-1})\bigg)\prod_{(j,k)\in{\cal D}}
\vert s_j-s_k\vert^{-\alpha_0}\gamma(x_j-x_k)\nonumber
\end{align}
where, $I_1=\{1,\cdots, n\}$, $I_2=\{n+1,\cdots 2n\}$, where, to simplify the notation, we follow the convention on the right hand side that $s_0=0$, $x_0=0$,
$s_n=0$, $x_n=0$ in the expression $G(s_{n+1}-s_n, x_{n+1}-x_n)$.
\end{theorem}

Notice that under the initial condition given in (\ref{intro-1}), 
$$
S_{n,\epsilon,\delta}\big(g_n(\cdot, t, x)\big)\buildrel d\over =S_{n,\epsilon,\delta}\big(g_n(\cdot, t, 0)\big)
$$
we may take $x=0$ in the proof of Theorem \ref{th-4}. By Lemma \ref{L-0}, to establish (\ref{S-2}), all we need is to prove
that the limit
\begin{align}\label{S-4}
\lim_{\bar{\epsilon},\bar{\delta}\to 0^+}
\E S_{n,\epsilon,\delta}\big(g_n(\cdot, t, 0)\big)S_{n,\tilde{\epsilon}, \tilde{\delta}}\big(g_n(\cdot, t, 0)\big)
\end{align}
exists. By Wick's formula (\ref{M-8}) and the covariance identity (\ref{W-5})
$$
\begin{aligned}
&\E S_{n,\epsilon,\delta}\big(g_n(\cdot, t, 0)\big)S_{n,\tilde{\epsilon}, \tilde{\delta}}\big(g_n(\cdot, t, 0)\big)\\
&=\sum_{{\cal D}\in\Pi_n}\int_{(\R_+\times\R^d)^{2n}}g_n(s_1,\cdots, s_n, x_1,\cdots, x_n, t, 0)
g_n(s_{n+1},\cdots, s_{2n}, x_{n+1},\cdots, x_{2n}, t, 0)\nonumber\\
&\times\prod_{(j,k)\in{\cal D}}\gamma^0_{\delta_j\vee\delta_k}(s_j-s_k)\gamma_{\epsilon_j+\epsilon_k}(x_j-x_k)\\
&=\sum_{{\cal D}\in\Pi_n}\int_{[0, t]_<^n\times[0, t]_<^n}ds_1\cdots ds_{2n}\int_{(\R^d)^n\times(\R^d)^n}
dx_1\cdots dx_{2n}\Big(G(t-s_n, -x_n)\cdots G(s_2-s_1, x_2-x_1)\Big)\\
&\times\Big(G(t-s_{2n}, -x_{2n})\cdots 
G(s_{n+2}-s_{n+1}, x_{n+2}-x_{n+1})\Big)\prod_{(j,k)\in{\cal D}}
\gamma_{\delta_j\vee\delta_k}^0(s_j-s_k)\gamma_{\epsilon_j+\epsilon_k}(x_j-x_k)\nonumber
\end{aligned}
$$
where $\gamma_\delta^0(\cdot)$ and $\gamma_\epsilon(\cdot)$ are defined in (\ref{M-17}).

Under the substitution $s_l\mapsto t-s_{n-l+1}$, $x_l\mapsto -x_{n-l+1}$ ($1\le l\le n$)
and $s_l\mapsto t-s_{2n-l+1}$, $x_l\mapsto -x_{2n-l+1}$ ($n+1\le l\le 2n$), the right hand side 
is equal to
$$
\begin{aligned}
&\sum_{{\cal D}\in\Pi_n}\int_{[0, t]_<^n\times[0,t]_<^n}ds_1\cdots ds_{2n}\int_{(\R^d)^{2n}}dx_1\cdots dx_{2n}
\bigg(\prod_{\rho=1}^2\prod_{l\in I_\rho}G(s_l-s_{l-1}, x_l-x_{l-1})\bigg)\\
&\times\prod_{(j,k)\in{\cal D}}\gamma_{\delta_{\sigma(j)}\vee\delta_{\sigma(k)}}(s_{\sigma(j)}-s_{\sigma(k)})
\gamma_{\epsilon_{\sigma(j)}+\epsilon_{\sigma(k)}}(x_{\sigma(j)}-x_{\sigma(k)})
\end{aligned}
$$
where $\sigma$ is the permutation on
$\{1, 2,\cdots, 2n\}$ such that $\sigma(l)=n-l+1$ ($l\in I_1$) and $\sigma(l)=2n-l+1$ ($l\in I_2$). Since
$\sigma(I_1)=I_1$ and $\sigma(I_2)=I_2$, the equality continues to be equal to
$$
\begin{aligned}
&\sum_{{\cal D}\in\Pi_n}\int_{[0, t]_<^n\times[0,t]_<^n}ds_1\cdots ds_{2n}\int_{(\R^d)^{2n}}dx_1\cdots dx_{2n}
\bigg(\prod_{\rho=1}^2\prod_{l\in I_\rho}G(s_l-s_{l-1}, x_l-x_{l-1})\bigg)\\
&\times\prod_{(j,k)\in{\cal D}}\gamma_{\delta_{j\vee\delta_k}}(s_j-s_k)
\gamma_{\epsilon_j+\epsilon_k}(x_j-x_k).
\end{aligned}
$$
In summary,
\begin{align}\label{S-5}
&\E S_{n,\epsilon, \delta}\big(g_n(\cdot, t, 0)\big)S_{n,\tilde{\epsilon}, \tilde{\delta}}\big(g_n(\cdot, t, 0)\big)\\
&=\sum_{{\cal D}\in\Pi_n}\int_{[0, t]_<^n\times[0,t]_<^n}ds_1\cdots ds_{2n}\int_{(\R^d)^{2n}}dx_1\cdots dx_{2n}
\bigg(\prod_{\rho=1}^2\prod_{l\in I_\rho}G(s_l-s_{l-1}, x_l-x_{l-1})\bigg)\nonumber\\
&\times\prod_{(j,k)\in{\cal D}}\gamma_{\delta_{j\vee\delta_k}}(s_j-s_k)
\gamma_{\epsilon_j+\epsilon_k}(x_j-x_k).\nonumber
\end{align}

 Let $\theta_1, \theta_2>0$ be fixed but arbitrary and set
$$
\tilde{g}_\rho(s_1,\cdots, s_n, x_1,\cdots, x_n)
=\bigg(\prod_{l=1}^ne^{-\theta_\rho (s_l-s_{l-1})}
G(s_l-s_{l-1}, x_l-x_{l-1})\bigg)1_{(\R_+)_<^n}(s_1,\cdots, s_n)
$$
with $\rho=1,2$ and the convention $s_0=0$ and $x_0=0$.
Consider the function
\begin{align}\label{S-5-1}
&F_{\bar{\epsilon}}(t_1,\cdots, t_{2n})\equiv\int_0^{t_1}\!\cdots\!\int_0^{t_{2n}}ds_1,\cdots ds_{2n}
\int_{(\R^d)^{2n}}dx_1\cdots dx_{2n}\\
&\times\tilde{g}_1(s_1,\cdots, s_n, x_1,\cdots, x_n)\tilde{g}_2(s_{n+1},\cdots, s_{2n}, x_{n+1},\cdots, x_{2n})
\prod_{(j,k)\in{\cal D}}
\gamma_{\epsilon_j+\epsilon_k}(x_j-x_k).\nonumber
\end{align}

\begin{lemma}\label{L-2} There is a constant $C$ independent of $\bar{\epsilon}$
  such that
  \begin{align}\label{S-6}
F_{\bar{\epsilon}}(+\infty,\cdots, +\infty)\le C<\infty.
\end{align}
Further, the distribution function (up to normalization) $F_{\bar{\epsilon}}$ weakly converges as $\bar{\epsilon}\to 0^+$.
\end{lemma}

\proof Let $\lambda_1,\cdots,\lambda_{2n}>0$ be fixed but arbitrary.
$$
\begin{aligned}&\int_{\R_+^{2n}}\exp\Big\{-\sum_{l=1}^{2n}\lambda_l t_l\Big\}F_{\bar{\epsilon}}(dt_1,\cdots, dt_{2n})\\
&=\int_{(\R^d)^{2n}}d{\bf x}
\bigg(\prod_{(j,k)\in{\cal D}}\gamma_{\epsilon_j+\epsilon_k}(x_j-x_k)\bigg)
\int_{\R_+^{2n}}\exp\Big\{-\sum_{l=1}^{2n}\lambda_l t_l\Big\}\\
&\times\tilde{g}_1(t_1,\cdots, t_n, x_1,\cdots, x_n)\tilde{g}_2(t_{n+1},\cdots, t_{2n}, x_{n+1},\cdots, x_{2n})
dt_1\cdots dt_{2n}.
\end{aligned}
$$
By our set-up
$$
\begin{aligned}
&\int_{\R_+^{2n}}\exp\Big\{-\sum_{l=1}^{2n}\lambda_l t_l\Big\}
\tilde{g}_1(t_1,\cdots, t_n, x_1,\cdots, x_n)\tilde{g}_2(t_{n+1},\cdots, t_{2n}, x_{n+1},\cdots, x_{2n})
dt_1\cdots dt_{2n}\\
&=\int_{(\R_+)_<^{n}\times(\R_+)_<^n}dt_1\cdots dt_{2n}\exp\Big\{-\sum_{l=1}^{2n}\lambda_l t_l\Big\}
\bigg(\prod_{\rho=1}^2\prod_{l\in I_\rho}e^{-\theta_\rho (t_l-t_{l-1})}
G(t_l-t_{l-1}, x_l-x_{l-1})\bigg).
\end{aligned}
$$
Write
$$
\sum_{l=1}^{2n}\lambda_l t_l=\sum_{\rho=1}^2\sum_{l\in I_\rho}\lambda_l t_l
=\sum_{\rho=1}^2\sum_{l\in I_\rho}c_l(t_l-t_{l-1})
$$
where
$$
c_l=\sum_{k=l}^n\lambda_k\hskip.1in (1\le l\le n)\hskip.1in\hbox{and}\hskip .1in
c_l=\sum_{k=n+l}^{2n}\lambda_k\hskip.1in (n+1\le l\le 2n).
$$
The right hand side is equal to
$$
\begin{aligned}
&\int_{(\R_+)_<^{n}\times(\R_+)_<^n}dt_1\cdots dt_{2n}\bigg(\prod_{\rho=1}^2\prod_{l\in I_\rho}
e^{-(c_l+\theta_\rho) (t_l-t_{l-1})}G(t_l-t_{l-1}, x_l-x_{l-1})\bigg)\\
&=\prod_{\rho=1}^2\prod_{l\in I_\rho}\int_0^\infty dt e^{-(c_l+\theta_\rho) t}
G(t, x_l-x_{l-1})\\
&=\Big({1\over 2}\Big)^{2n}\prod_{\rho=1}^2\prod_{l\in I_\rho}\int_0^\infty dt 
\exp\Big\{-{1\over 2}(c_l+\theta_\rho)^2 t\Big\}p(t, x_l-x_{l-1})\cr
&=\Big({1\over 2}\Big)^{2n}
\int_{(\R_+)_<^{n}\times(\R_+)_<^n}dt_1\cdots dt_{2n}
\prod_{\rho=1}^2\prod_{l\in I_\rho}\exp\Big\{-{1\over 2}(c_l+\theta_\rho)^2 (t_l-t_{l-1})\Big\}\Big\}
p(t_l-t_{l-1}, x_l-x_{l-1})
\end{aligned}
$$
where $p(t,x)$ is the Brownian semi-group defined in (\ref{intro-17}) and the second step follows from
the identity (\ref{intro-16}).

In summary
$$
\begin{aligned}
&\int_{\R_+^{2n}}\exp\Big\{-\sum_{l=1}^{2n}\lambda_l t_l\Big\}F_{\bar{\epsilon}}(dt_1,\cdots, dt_{2n})\cr
&=\Big({1\over 2}\Big)^{2n}
\int_{(\R_+)_<^{n}\times(\R_+)_<^n}dt_1\cdots dt_{2n}\exp\Big\{-{1\over 2}\sum_{\rho=1}^2
\sum_{l\in I_\rho}(c_l+\theta_\rho)^2 (t_l-t_{l-1})\Big\}\\
&\times\int_{(\R^d)^{2n}}d{\bf x}\bigg(\prod_{\rho=1}^2\prod_{l\in I_\rho}p(t_l-t_{l-1}, x_l-x_{l-1})\bigg)
\bigg(\prod_{(j,k)\in{\cal D}}\gamma_{\epsilon_j+\epsilon_k}(x_j-x_k)\bigg).
\end{aligned}
$$
Notice that the function
$$
f(x_1,\cdots, x_{2n})=\prod_{\rho=1}^2\prod_{l\in I_\rho}p(t_l-t_{l-1}, x_l-x_{l-1})
$$
is the density of the random vector $\big(B_1(t_1), \cdots, B_1(t_n), B_2(t_{n+1}),\cdots,B_2(t_{2n})\big)$, 
where $B_1(t)$ and $B_2(t)$ are two independent $d$-dimensional Brownian motions. Let "$\E_0$" denote the expectation of the Brownian motions.
We have
$$
\begin{aligned}
&\int_{(\R^d)^{2n}}d{\bf x}\bigg(\prod_{\rho=1}^2\prod_{l\in I_\rho}p(t_l-t_{l-1}, x_l-x_{l-1})\bigg)
\bigg(\prod_{(j,k)\in{\cal D}}\gamma_{\epsilon_j+\epsilon_k}(x_j-x_k)\bigg)\\
&=\E_0\prod_{(j,k)\in{\cal D}}\gamma_{\epsilon_j+\epsilon_k}\big(B_{v(j)}(t_j)-B_{v(k)}(t_k)\big)
\end{aligned}
$$
where the map $v$: $\{1,\cdots, 2n\}\longrightarrow \{1, 2\}$ is given by $v(I_\rho)=\{\rho\}$ ($\rho=1,2$).
By Fourier transform
$$
\begin{aligned}
&\E_0\prod_{(j,k)\in{\cal D}}\gamma_{\epsilon_j+\epsilon_k}\big(B_{v(j)}(t_j)-B_{v(k)}(t_k)\big)\\
&=\int_{(\R^d)^n}\exp\bigg\{-{1\over 2}\sum_{(j,k)\in{\cal D}}(\epsilon_j+\epsilon_k)\vert\xi_{j,k}\vert^2\bigg\}
\E_0\exp\bigg\{i\sum_{(j,k)\in{\cal D}}\xi_{j,k}\cdot \big(B_{v(j)}(t_j)-B_{v(k)}(t_k)\big)\bigg\}\\
&\times\prod_{(j,k)\in{\cal D}}\mu(d\xi_{j,k})\\
&=\int_{(\R^d)^n}\exp\bigg\{-\sum_{(j,k)\in{\cal D}}(\epsilon_j+\epsilon_k)\vert\xi_{j,k}\vert^2\bigg\}
\exp\bigg\{-{1\over 2}\Var\bigg(\sum_{(j,k)\in{\cal D}}\xi_{j,k}\cdot \big(B_{v(j)}(t_j)-B_{v(k)}(t_k)\big)\bigg)\bigg\}\\
&\times\prod_{(j,k)\in{\cal D}}\mu(d\xi_{j,k}).
\end{aligned}
$$
The right hand side is monotonic in $\epsilon_1,\cdots, \epsilon_{2n}$.
Using monotonic convergence
we conclude that the limit
\begin{align}\label{S-7}
&\lim_{\bar{\epsilon}\to 0^+}
\int_{\R_+^{2n}}\exp\Big\{-\sum_{l=1}^{2n}\lambda_l t_l\Big\}F_{\bar{\epsilon}}(dt_1,\cdots, dt_{2n})\\
&=\Big({1\over 2}\Big)^{2n}
\E_0\int_{(\R_+)_<^{n}\times(\R_+)_<^n}dt_1\cdots dt_{2n}\exp\Big\{-{1\over 2}\sum_{\rho=1}^2
\sum_{l\in I_\rho}(c_l+\theta_\rho)^2 (t_l-t_{l-1})\Big\}\nonumber\\
&\times\prod_{(j,k)\in{\cal D}}\gamma\big(B_{v(j)}(t_j)-B_{v(k)}(t_k)\big)\nonumber
\end{align}
exists as a extended real number. Further,
$$
\begin{aligned}
&\int_{\R_+^{2n}}\exp\Big\{-\sum_{l=1}^{2n}\lambda_l t_l\Big\}F_{\bar{\epsilon}}(dt_1,\cdots, dt_{2n})\cr
&\le \Big({1\over 2}\Big)^{2n}
\E_0\int_{(\R_+)_<^{n}\times(\R_+)_<^n}dt_1\cdots dt_{2n}\exp\Big\{-{1\over 2}\sum_{\rho=1}^2
\sum_{l\in I_\rho}(c_l+\theta_\rho)^2 (t_l-t_{l-1})\Big\}\\
&\times\prod_{(j,k)\in{\cal D}}\gamma\big(B_{v(j)}(t_j)-B_{v(k)}(t_k)\big)
\end{aligned}
$$

Taking $\lambda_1=\cdots =\lambda_{2n}=0$,
$$
\begin{aligned}
&F_{\bar{\epsilon}}(+\infty,\cdots, +\infty)\\
&\le \Big({1\over 2}\Big)^{2n}
  \E_0\int_{(\R_+)_<^{n}\times(\R_+)_<^n}dt_1\cdots dt_{2n}\exp\Big\{-{1\over 2}
  \sum_{\rho=1}^2
\sum_{l\in I_\rho}\theta_\rho^2 (t_l-t_{l-1})\Big\}\nonumber\\
&\times\prod_{(j,k)\in{\cal D}}\gamma\big(B_{v(j)}(t_j)-B_{v(k)}(t_k)\big)\\
  &\le \Big({1\over 2}\Big)^{2n}
    \E_0\int_{(\R_+)_<^{n}\times(\R_+)_<^n}dt_1\cdots dt_{2n}\exp\Big\{-{1\over 2}
    (\theta_1^2 t_n+\theta_2^2 t_{2n})\Big\}\nonumber\\
&\times\prod_{(j,k)\in{\cal D}}\gamma\big(B_{v(j)}(t_j)-B_{v(k)}(t_k)\big).
\end{aligned}
$$
To complete the proof, therefore, all we need is
\begin{align}\label{S-8}
  &\E_0\int_{(\R_+)_<^{n}\times(\R_+)_<^n}dt_1\cdots dt_{2n}\exp\Big\{-{1\over 2}
    (\theta_1^2 t_n+\theta_2^2 t_{2n})\Big\}\\
    &\times\prod_{(j,k)\in{\cal D}}\gamma\big(B_{v(j)}(t_j)-B_{v(k)}(t_k)\big)
    <\infty.\nonumber
  \end{align}
First notice that
\begin{align}\label{S-9} 
  & \int_{(\R_+)_<^n\times(\R_+)_<^n}dt_1\cdots dt_{2n}
    \exp\Big\{-{1\over 2}(\theta^2_1 t_n+\theta^2_2 t_{2n})\Big\}
 \prod_{(j,k)\in{\cal D}}\gamma\big(B_{v(j)}(t_j)-B_{v(k)}(t_k)\big)\\
  &=\theta_1\theta_2\int_0^\infty\!\int_0^\infty ds_1ds_2
    \exp\Big\{-{1\over 2}(\theta^2_1s_1+\theta^2_2 s_2)\Big\}
 \int_{[0, s_1]_<^n\times [0, s_2]_<^n}dt_1\cdots dt_{2n}\nonumber\\
 &\times\prod_{(j,k)\in{\cal D}}\gamma\big(B_{v(j)}(t_j)-B_{v(k)}(t_k)\big).
   \nonumber
\end{align}
Indeed, by Fubini theorem, the right hand side is equal to
$$
\begin{aligned}
  &\int_{(\R_+)_<^n\times(\R_+)_<^n}dt_1\cdots dt_{2n}\exp\Big\{-{1\over 2}(\theta^2_1 t_n+\theta^2_2 t_{2n})\Big\}
 \prod_{(j,k)\in{\cal D}}\gamma\big(B_{v(j)}(t_j)-B_{v(k)}(t_k)\big)\\
 &\times\theta_1\theta_2\int_{t_n}^\infty ds_1\int_{t_{2n}}^\infty ds_2
   \exp\Big\{-{1\over 2}(\theta_1^2 (s_1-t_n)+\theta_2^2(s_2-t_{2n})\Big\}.
   \end{aligned}
 $$
So the claim made in (\ref{S-9}) follows from
$$
\theta_1\theta_2\int_{t_n}^\infty ds_1\int_{t_{2n}}^\infty ds_2
 \exp\Big\{-{1\over 2}(\theta_1^2 (s_1-t_n)+\theta_2^2(s_2-t_{2n})\Big\}=1.
 $$
 In addition,
$$
\begin{aligned}
  &
    \int_{[0, s_1]_<^n\times [0, s_2]_<^n}dt_1\cdots dt_{2n}
    \prod_{(j,k)\in{\cal D}}\gamma\big(B_{v(j)}(t_j)-B_{v(k)}(t_k)\big)\\
 &\le 
   \int_{[0, s_1]^n\times [0, s_2]^n}dt_1\cdots dt_{2n}
   \prod_{(j,k)\in{\cal D}}\gamma\big(B_{v(j)}(t_j)-B_{v(k)}(t_k)\big)\\
  &=\prod_{(j,k)\in{\cal D}}\int_0^{s_{v(j)}}\!\int_0^{s_{v(k)}}
    \gamma\big(B_{v(j)}(t_1)-B_{v(k)}(t_2)\big) dt_1dt_2
    \end{aligned}
 $$
where the last step follows from Fubini's theorem.
For the mutual intersection local times on the right hand side (i.e.,
for the $(j,k)\in {\cal D}$ but $j$ and $k$ coming from different groups)
we treat them by Fourier transform
$$
\begin{aligned}
  &\int_0^{s_1}\!\!\int_0^{s_2}\gamma\big(B_1(t_1)-B_2(t_2)\big)dt_1dt_2\\
&=\int_{\R^d}\mu(d\xi)\int_0^{s_1}\!\!\int_0^{s_2}\exp\Big\{i\xi\cdot (B_1(t_1)-B_2(t_2))\Big\}dt_1dt_2\\
&=\int_{\R^d}\mu(d\xi)\bigg[\int_0^{s_1}e^{i\xi\cdot B_1(t)}dt\bigg]\overline{\bigg[\int_0^{s_2}e^{i\xi\cdot B_2(t)}dt\bigg]}\\
&\le\bigg\{\int_{\R^d}\mu(d\xi)\bigg\vert\int_0^{s_1}e^{i\xi\cdot B_1(t)}dt\bigg\vert^2\bigg\}^{1/2}
\bigg\{\int_{\R^d}\mu(d\xi)\bigg\vert\int_0^{s_2}e^{i\xi\cdot B_2(t)}dt\bigg\vert^2\bigg\}^{1/2}\\
  &=\bigg\{\int_0^{s_1}\!\int_0^{s_1}\gamma\big(B_1(t_1)-B_1(t_2)\big)
    dt_1dt_2\bigg\}^{1/2}
    \bigg\{\int_0^{s_2}\!\int_0^{s_2}
    \gamma\big(B_2(t_1)-B_2(t_2)\big)dt_1dt_2\bigg\}^{1/2}.
\end{aligned}
$$
We therefore have
$$
\begin{aligned}
  &\prod_{(j,k)\in{\cal D}}\int_0^{s_{v(j)}}\!\int_0^{s_{v(k)}}\gamma\big(B_{v(j)}(t_1)-B_{v(k)}(t_2)\big) dt_1dt_2\\
  &\le\bigg[\int_0^{s_1}\!\int_0^{s_1}
    \gamma\big(B_1(t_1)-B_1(t_2)\big)dt_1dt_2\bigg]^n
    \bigg[\int_0^{s_2}\!\int_0^{s_2}\gamma\big(B_2(t_1)-B_2(t_2)\big)
    dt_1dt_2\bigg]^n.
    \end{aligned}
$$
Summarizing the steps since (\ref{S-9}),
$$
\begin{aligned}
& \E_0\int_{(\R_+)_<^n\times(\R_+)_<^n}dt_1\cdots dt_{2n}\exp\Big\{-{1\over 2}(\theta^2_1 t_n+\theta^2_2 t_{2n})\Big\}
 \prod_{(j,k)\in{\cal D}}\gamma\big(B_{v(j)}(t_j)-B_{v(k)}(t_k)\big)\\
&\le \theta_1\theta_2\bigg\{\int_0^\infty
  \exp\Big\{-{\theta^2_1\over 2}s\Big\}\E_0\bigg[\int_0^{s}\!\int_0^{s}\gamma\big(B(t_1)-B(t_2)\big)dt_1dt_2\bigg]^nds\bigg\}\\
 &\times\bigg\{\int_0^\infty \exp\Big\{-{\theta^2_2\over 2}s\Big\}\E_0\bigg[\int_0^{s}\!\int_0^{s}\gamma\big(B(t_1)-B(t_2)\big)dt_1dt_2\bigg]^nds\bigg\}\\
&<\infty
  \end{aligned}
 $$
where the last step follows from (6.1), Lemma 6.1 in \cite{CH}.\qed

\begin{lemma}\label{L-3} For any ${\cal D}\in \Pi_n$, $t_1,t_2>0$ and
  $\eta>0$, the limit
$$
\begin{aligned}
&\lim_{\bar{\epsilon}\to 0^+}\int_{[0, t_1]_<^n\times[0, t_2]_<^n}ds_1\cdots ds_{2n}\int_{(\R^d)^n\times(\R^d)^n}
dx_1\cdots dx_{2n}\nonumber\\
&\times\bigg(\prod_{\rho=1}^2\prod_{l\in I_\rho}G(s_l-s_{l-1}, x_l-x_{l-1})\bigg)
\prod_{(j,k)\in{\cal D}}
\gamma_{\eta}^0(s_j-s_k)\gamma_{\epsilon_j+\epsilon_k}(x_j-x_k)
\end{aligned}
$$
exists and finite.
\end{lemma}

\proof A consequence of Lemma \ref{L-2} is that
$$
\lim_{\bar{\epsilon}\to 0^+}\int_{\R_+^{2n}}\varphi(t_1,\cdots, t_{2n})F_{\bar{\epsilon}}(dt_1,\cdots, dt_{2n})
$$
exists for any bounded, continuous and non-negative
function $\varphi(t_1,\cdots, t_{2n})$ on $\R_+^{2n}$. 
On the other hand,
similar to (\ref{S-9}),
$$
\begin{aligned}
&\int_{\R_+^{2n}}\varphi(t_1,\cdots, t_{2n})F_{\bar{\epsilon}}(dt_1,\cdots, dt_{2n})\\
&=\int_{(\R_+)_<^n\times(\R_+)_<^n}ds_1\cdots ds_{2n}\int_{(\R^d)^{2n}}d{\bf x}
\bigg(\prod_{\rho=1}^2\prod_{l\in I_\rho}e^{-\theta_\rho(s_l-s_{l-1})}G(s_l-s_{l-1}, x_l-x_{l-1})\bigg)\\
&\times\varphi(s_1,\cdots,s_{2n})\prod_{(j,k)\in{\cal D}}\gamma_{\epsilon_j+\epsilon_k}(x_j-x_k)\\
&=(\theta_1\theta_2)\int_{\R_+^2}dt_1dt_2\exp\{-\theta_1t_1-\theta_2 t_2\}
\int_{[0, t_1]_<^n\times[0,t_2]_<^n}ds_1\cdots ds_{2n}\int_{(\R^d)^{2n}}d{\bf x}\\
&\times\bigg(\prod_{\rho=1}^2\prod_{l\in I_\rho}G(s_l-s_{l-1}, x_l-x_{l-1})\bigg)
\varphi(s_1,\cdots,s_{2n})\prod_{(j,k)\in{\cal D}}\gamma_{\epsilon_j+\epsilon_k}(x_j-x_k).
\end{aligned}
$$
We reach the conclusion that the limit
$$
\begin{aligned}
&\lim_{\bar{\epsilon}\to 0^+}\int_{\R_+^2}dt_1dt_2\exp\{-\theta_1t_1-\theta_2 t_2\}
\int_{[0, t_1]_<^n\times[0,t_2]_<^n}ds_1\cdots ds_{2n}\int_{(\R^d)^{2n}}d{\bf x}\\
&\times\bigg(\prod_{\rho=1}^2\prod_{l\in I_\rho}G(s_l-s_{l-1}, x_l-x_{l-1})\bigg)
\varphi(s_1,\cdots,s_{2n})\prod_{(j,k)\in{\cal D}}\gamma_{\epsilon_j+\epsilon_k}(x_j-x_k)
\end{aligned}
$$
exists for any $\theta_1,\theta_2>0$. Set
$$
\begin{aligned}
&{\cal G}^{\cal D}_{\bar{\epsilon}}
(t_1, t_2)=\int_{[0, t_1]_<^n\times[0,t_2]_<^n}ds_1\cdots ds_{2n}\int_{(\R^d)^{2n}}d{\bf x}\\
&\times\bigg(\prod_{\rho=1}^2\prod_{l\in I_\rho}G(s_l-s_{l-1}, x_l-x_{l-1})\bigg)
\varphi(s_1,\cdots,s_{2n})\prod_{(j,k)\in{\cal D}}\gamma_{\epsilon_j+\epsilon_k}(x_j-x_k).
\end{aligned}
$$
By continuity theorem of Laplace transform again, ${\cal G}^{\cal D}_{\bar{\epsilon}}(t_1, t_2)$ weakly converges: There is
a non-decreasing and right continuous function ${\cal G}(t_1, t_2)$ such that
\begin{align}\label{S-10}
\lim_{\bar{\epsilon}\to 0^+}{\cal G}^{\cal D}_{\bar{\epsilon}}(t_1, t_2)={\cal G}^{\cal D}(t_1, t_2)
\end{align}
at every continuous point $(t_1, t_2)$ of ${\cal G}^{\cal D}(\cdot, \cdot)$.
(Actually, Theorem 5.22, \cite{K} is stated for
probability measures on $(\R^+)^d$. The case of general measures on $(\R^+)^d$
can be derived as in the proof of Theorem 2a, Section 1, Chapter 2,
\cite{Feller} Although this theorem only considers
measures on $\R_+$ its extension to $\R_+^2$ is routine).

We now claim that ${\cal G}^{\cal D}(\cdot, \cdot)$
is continuous on $\R_+\times\R_+$. Consequently, this is to say that
(\ref{S-10}) holds on $\R_+\times\R_+$ and therefore Lemma \ref{L-3} holds.
To prove it, all we need is to show that
$$
\lim_{\eta_1,\eta_2\to 0^+}
\sup_{\bar{\epsilon}}\Big\{{\cal G}^{\cal D}_{\bar{\epsilon}}(t_1, t_2)-{\cal G}^{\cal D}_{\bar{\epsilon}}(t_1-\eta_1, t_2-\eta_2)\Big\}
=0
$$
for every $(t_1,t_2)\in\R_+\times\R_+$, Or,
$$
\begin{aligned}
&\lim_{\eta_1,\eta_2\to 0^+}
\sup_{\bar{\epsilon}}\int_{([0, t_1]_<^n\times[0,t_2]_<^n)\setminus ([0, t_1-\eta_1]_<^n\times[0,t_2-\eta_2]_<^n)}
ds_1\cdots ds_{2n}\int_{(\R^d)^{2n}}d{\bf x}\\
&\times\bigg(\prod_{\rho=1}^2\prod_{l\in I_\rho}G(s_l-s_{l-1}, x_l-x_{l-1})\bigg)
\varphi(s_1,\cdots,s_{2n})\prod_{(j,k)\in{\cal D}}\gamma_{\epsilon_j+\epsilon_k}(x_j-x_k)=0.
\end{aligned}
$$
Indeed,
$$
\begin{aligned}
&\int_{([0, t_1]_<^n\times[0,t_2]_<^n)\setminus ([0, t_1-\eta_1]_<^n\times[0,t_2-\eta_2]_<^n)}
ds_1\cdots ds_{2n}\int_{(\R^d)^{2n}}d{\bf x}\\
&\times\bigg(\prod_{\rho=1}^2\prod_{l\in I_\rho}G(s_l-s_{l-1}, x_l-x_{l-1})\bigg)
\varphi(s_1,\cdots,s_{2n})\prod_{(j,k)\in{\cal D}}\gamma_{\epsilon_j+\epsilon_k}(x_j-x_k)\\
&\le\sup_{t_1,\cdots, t_{2n}}\varphi(t_1,\cdots, t_{2n})\int_{([0, t_1]_<^n\times[0,t_2]_<^n)\setminus ([0, t_1-\eta_1]_<^n\times[0,t_2-\eta_2]_<^n)}
ds_1\cdots ds_{2n}\int_{(\R^d)^{2n}}d{\bf x}\\
&\times\bigg(\prod_{\rho=1}^2\prod_{l\in I_\rho}G(s_l-s_{l-1}, x_l-x_{l-1})\bigg)
\prod_{(j,k)\in{\cal D}}\gamma_{\epsilon_j+\epsilon_k}(x_j-x_k).
\end{aligned}
$$
Thus, our claim follows from the fact (established in
the proof of Lemma 3.6, \cite{CH}) that
$$
\begin{aligned}
&\lim_{\eta_1,\eta_2\to 0^+}\sup_{\bar{\epsilon}}\int_{([0, t_1]_<^n\times[0,t_2]_<^n)\setminus ([0, t_1-\eta_1]_<^n\times[0,t_2-\eta_2]_<^n)}
ds_1\cdots ds_{2n}\int_{(\R^d)^{2n}}d{\bf x}\\
&\times\bigg(\prod_{\rho=1}^2\prod_{l\in I_\rho}G(s_l-s_{l-1}, x_l-x_{l-1})\bigg)
\prod_{(j,k)\in{\cal D}}\gamma_{\epsilon_j+\epsilon_k}(x_j-x_k)=0
\end{aligned}
$$

Thus, (\ref{S-10}) holds for every $(t_1, t_2)\in\R_+\times\R_+$. Taking
$$
\varphi(t_1,\cdots, t_{2n})=\prod_{(j,k)\in{\cal D}}\gamma_\eta^0(t_j-t_k)
$$
in (\ref{S-10}) completes the proof. \qed

\begin{lemma}\label{L-4} For any ${\cal D}\in \Pi_n$, $t_1,t_2>0$, the limit
$$
\begin{aligned}
&\lim_{\bar{\epsilon}, \bar{\delta}\to 0^+}\int_{[0, t_1]_<^n\times[0, t_2]_<^n}ds_1\cdots ds_{2n}\int_{(\R^d)^n\times(\R^d)^n}
dx_1\cdots dx_{2n}\nonumber\\
&\times\bigg(\prod_{\rho=1}^2\prod_{l\in I_\rho}G(s_l-s_{l-1}, x_l-x_{l-1})\bigg)
\prod_{(j,k)\in{\cal D}}
\gamma_{\delta_j\vee\delta_k}^0(s_j-s_k)\gamma_{\epsilon_j+\epsilon_k}(x_j-x_k)
\end{aligned}
$$
exists and finite.
\end{lemma}

\begin{remark} In view (\ref{S-5}),  Lemma \ref{L-4} leads to the existence of the limit in (\ref{S-4}),
and therefore (Lemma \ref{L-0}) to the existence of the ${\cal L}^2$-limit in (\ref{S-2}). According to the definition given in
(\ref{M-15}), Lemma \ref{L-4} justifies the use of the notation
$$
\begin{aligned}
&\int_{[0, t_1]_<^n\times[0, t_2]_<^n}ds_1\cdots ds_{2n}\int_{(\R^d)^n\times(\R^d)^n}
dx_1\cdots dx_{2n}\nonumber\\
&\times\bigg(\prod_{\rho=1}^2\prod_{l\in I_\rho}G(s_l-s_{l-1}, x_l-x_{l-1})\bigg)
\prod_{(j,k)\in{\cal D}}\vert s_j-s_k\vert^{-\alpha_0}\gamma(x_j-x_k).
\end{aligned}
$$
Therefore, taking limit in (\ref{S-5}) leads to (\ref{S-3}). That completes the proof of Theorem \ref{th-4}
\end{remark}

\proof Let
$$
\begin{aligned}
&{\cal G}_{\bar{\epsilon}, \bar{\delta}}(t_1,t_2)=\int_{[0, t_1]_<^n\times[0, t_2]_<^n}ds_1\cdots ds_{2n}\int_{(\R^d)^n\times(\R^d)^n}
dx_1\cdots dx_{2n}\nonumber\\
&\times\bigg(\prod_{\rho=1}^2\prod_{l\in I_\rho}G(s_l-s_{l-1}, x_l-x_{l-1})\bigg)
\prod_{(j,k)\in{\cal D}}
\gamma_{\delta_j\vee\delta_k}^0(s_j-s_k)\gamma_{\epsilon_j+\epsilon_k}(x_j-x_k).
\end{aligned}
$$
For each $\eta>0$
$$
{\cal G}_{\bar{\epsilon}, \eta}(t_1,t_2)={\cal G}_{\bar{\epsilon}, \bar{\delta}}(t_1,t_2)
\Big\vert_{\delta_1=\cdots =\delta_{2n}=\eta}
$$
and
$$
\begin{aligned}
&{\cal G}_{\bar{\epsilon}, 0}(t_1,t_2)=\int_{[0, t_1]_<^n\times[0, t_2]_<^n}ds_1\cdots ds_{2n}\int_{(\R^d)^n\times(\R^d)^n}
dx_1\cdots dx_{2n}\nonumber\\
&\times\bigg(\prod_{\rho=1}^2\prod_{l\in I_\rho}G(s_l-s_{l-1}, x_l-x_{l-1})\bigg)
\prod_{(j,k)\in{\cal D}}
\vert s_j-s_k)\vert^{-\alpha_0}\gamma_{\epsilon_j+\epsilon_k}(x_j-x_k).
\end{aligned}
$$
By the monotonicity of ${\cal G}_{\bar{\epsilon}, \bar{\delta}}(t_1,t_2)$ in $\delta_1.\cdots\delta_{2n}$ and by
(\ref{intro-18}),
$$
{\cal G}_{\bar{\epsilon}, \eta}(t_1,t_2)\le {\cal G}_{\bar{\epsilon}, \bar{\delta}}(t_1,t_2)\le {\cal G}_{\bar{\epsilon}, 0}(t_1,t_2)
$$
whenever $\delta_1,\cdots,\delta_{2n}<\eta$. Therefore,
$$
\lim_{\bar{\epsilon}\to 0^+}{\cal G}_{\bar{\epsilon}, \eta}(t_1,t_2)\le \liminf_{\bar{\epsilon},\bar{\delta}\to 0^+}
{\cal G}_{\bar{\epsilon}, \bar{\delta}}(t_1,t_2)\le \liminf_{\bar{\epsilon},\bar{\delta}\to 0^+}
{\cal G}_{\bar{\epsilon}, \bar{\delta}}(t_1, t_2)\le\limsup_{\bar{\epsilon}\to 0^+}{\cal G}_{\bar{\epsilon}, 0}(t_1,t_2)
$$
where the limit on the left end is guarantteed by Lemma \ref{L-3}. To complete the proof, all we need is that
\begin{align}\label{S-11}
\lim_{\eta\to 0^+}\Big\{{\cal G}_{\bar{\epsilon}, 0}(t_1,t_2)-{\cal G}_{\bar{\epsilon}, \eta}(t_1,t_2)\Big\}=0
\hskip.1in\hbox{uniformly over $\bar{\epsilon}$}.
\end{align}
Indeed, (\ref{S-11}) implies that for any $\eta_1>0$,
$$
\limsup_{\bar{\epsilon}\to 0^+}{\cal G}_{\bar{\epsilon}, 0}(t_1,t_2)\le
\eta_1+ \lim_{\bar{\epsilon}\to 0^+}{\cal G}_{\bar{\epsilon}, \eta}(t_1,t_2)
$$
as $\eta$ is sufficiently small. Since $\displaystyle\lim_{\bar{\epsilon}\to 0^+}{\cal G}_{\bar{\epsilon}, \eta}(t_1,t_2)<\infty$
according to Lemma \ref{S-3}, in particular, 
$\displaystyle\limsup_{\bar{\epsilon}\to 0^+}{\cal G}_{\bar{\epsilon}, 0}(t_1,t_2)<\infty$.
On the other hand, by the fact that 
$\displaystyle\lim_{\bar{\epsilon}\to 0^+}{\cal G}_{\bar{\epsilon}, \eta}(t_1,t_2)$ is non-increasing in $\eta$, the limit
$$
\lim_{\eta\to 0^+}\lim_{\bar{\epsilon}\to 0^+}{\cal G}_{\bar{\epsilon}, \eta}(t_1,t_2)
$$ 
exists and is finite (as it is bounded by $\displaystyle\limsup_{\bar{\epsilon}\to 0^+}{\cal G}_{\bar{\epsilon}, 0}(t_1,t_2)$).
So we have
$$
\lim_{\eta\to 0^+}\lim_{\bar{\epsilon}\to 0^+}{\cal G}_{\bar{\epsilon}, \eta}(t_1,t_2)\le 
\liminf_{\bar{\epsilon},\bar{\delta}\to 0^+}
{\cal G}_{\bar{\epsilon}, \bar{\delta}}(t_1,t_2)\le \liminf_{\bar{\epsilon},\bar{\delta}\to 0^+}
{\cal G}_{\bar{\epsilon}, \bar{\delta}}(t_1, t_2)\le\eta_1+\lim_{\eta\to 0^+}
\lim_{\bar{\epsilon}\to 0^+}{\cal G}_{\bar{\epsilon}, \eta}(t_1,t_2)
$$
Letting $\eta_1\to 0^+$ leads to the proof of Lemma \ref{S-4}. 

It remains to prove (\ref{S-11}). For being consistent to the notation $\gamma_\delta^0(\cdot)$ for time-covariance,
by  (\ref{intro-18}) we use $\gamma_0^0(\cdot)$ instead of $\vert\cdot\vert^{-\alpha_0}$.
Notice
$$
\begin{aligned}
&\prod_{(j,k)\in{\cal D}}\gamma_{0}^0(s_j-s_k)-\prod_{(j,k)\in{\cal D}}\gamma_{\eta}^0(s_j-s_k)\\
&\le\sum_{(j_1, k_1)\in {\cal D}}(\gamma_0^0-\gamma_{\eta}^0)(s_{j_1}-s_{k_1})
\prod_{(j,k)\in{\cal D}\setminus\{(j_1, k_1)\}}\gamma_0^0(s_j-s_k).
\end{aligned}
$$
All we need is to show that for every $(j_1, k_1)\in{\cal D}$
$$
\begin{aligned}&\int_{[0,t_1]_<^{n}\times [0,t_2]_<^n} ds_1\cdots ds_{2n}
\int_{(\R^d)^{2n}}d{\bf x}
\bigg(\prod_{\rho=1}^2\prod_{l\in I_\rho}
G(s_l-s_{l-1}, x_l-x_{l-1})\bigg)\\
&\times\bigg(\prod_{(j,k)\in{\cal D}}
\gamma_{\epsilon_j+\epsilon_k}(x_j-x_k)\bigg)\bigg((\gamma_0^0-\gamma_{\eta}^0)(s_{j_1}-s_{k_1})
\prod_{(j,k)\in{\cal D}\setminus\{(j_1, k_1)\}}\gamma_0^0(s_j-s_k)\bigg)
\end{aligned}
$$
converges to 0 uniformly over $\bar{\epsilon}$.

We are in the position of using Theorem \ref{th-3}. Set
$$
\gamma^0_{j_1,k_1}(u)=(\gamma_0^0-\gamma_{\eta}^0)(\cdot)
=\Gamma(\alpha_0)^{-1}\int_{\eta^{-1}}^\infty e^{-\lambda\vert u\vert}
{d\lambda\over \lambda^{1-\alpha_0}}
$$
and
$\gamma_{j,k}^0(\cdot)=\gamma_0^0(u)$ for $(j, k)\not =(k_1,j_1)$. The above integral is
written as
$$
\begin{aligned}
&\int_{[0,t_1]_<^{n}\times [0,t_2]_<^n} ds_1\cdots ds_{2n}
\int_{(\R^d)^{2n}}d{\bf x}
\bigg(\prod_{\rho=1}^2\prod_{l\in I_\rho}
G(s_l-s_{l-1}, x_l-x_{l-1})\bigg)\cr
&\times\bigg(\prod_{(j,k)\in{\cal D}}\gamma_{j,k}^0(s_j-s_k)\gamma_{\epsilon_j+\epsilon_k}(x_j-x_k)\bigg)\\
&\le e^{t_1+t_2}\int_{(\R_+)^2}\exp\{-\tilde{t}_1-\tilde{t}_2\}\int_{[0,\tilde{t}_1]<^{n}\times [0,\tilde{t}_2]_<^n} ds_1\cdots ds_{2n}
\int_{(\R^d)^{2n}}d{\bf x}\\
&\times\bigg(\prod_{\rho=1}^2\prod_{l\in I_\rho}
G(s_l-s_{l-1}, x_l-x_{l-1})\bigg)
\bigg(\prod_{(j,k)\in{\cal D}}\gamma_{j,k}^0(s_j-s_k)\gamma_{\epsilon_j+\epsilon_k}(x_j-x_k)\bigg)\\
&=e^{t_1+t_2}\int_{(\R_+)_<^{n}\times (\R_+)_<^n} ds_1\cdots ds_{2n}
\int_{(\R^d)^{2n}}d{\bf x}\cr
&\times\bigg(\prod_{\rho=1}^2\prod_{l\in I_\rho}
e^{-(s_l-s_{l-1})}G(s_l-s_{l-1}, x_l-x_{l-1})\bigg)
\bigg(\prod_{(j,k)\in{\cal D}}\gamma_{j,k}^0(s_j-s_k)\gamma_{\epsilon_j+\epsilon_k}(x_j-x_k)\bigg).
\end{aligned}
$$
Applying Theorem \ref{th-3} to $G_l(t,x)=e^{-t}G(t,x)$ ($1\le l\le 2n$), the integral on the right hand side yields the bound
$$
\bigg(\prod_{l\in Q_0}\|G_l\|^{(0)}\bigg)\bigg(\prod_{l\in Q_1}2\|G_l\|_{(\cdot, \cdot)}^{(1)}\bigg)
\bigg(\prod_{l\in Q_2}\|G_l\|_{(\cdot, \cdot)}^{(2)}\bigg)
$$
where $Q_0$, $Q_1$, $Q_2$ form a partition of $\{1,\cdots, 2n\}$ with $\#(Q_0)=\#(Q_1)$
and $\#(Q_2)$ is even. What important to our course is that $(j_1,k_1)$ either appears
in $Q_1$-product once or in $Q_2$ twice. By (\ref{M-11-1})
$$
\|G_l\|^{(0)}=\int_0^\infty \int_{\R^d} e^{- t}G(t,x)dxdt
=\int_0^\infty t e^{- t}dt=1\hskip.2in l=1,2,\cdots 2n.
$$
The rest of the argument is to
check that $\|G_l\|_{j,k}^{(i)}$ ($i=1, 2$) are bounded uniformly over $\bar{\epsilon}$ for $(j,k)\not =(j_1,k_1)$
and $\|G_l\|_{j_1,k_1}^{(i)}\to 0^+$ ($i=1, 2$) uniformly over $\bar{\epsilon}$ as $\eta\to 0^+$.

For $(j,k)\not=(j_1,k_1)$,
$$
\begin{aligned}
\|G_l\|_{j,k}^{(1)}&=\Gamma(\alpha_0)^{-1}\int_{\R_+\times\R^d}\bigg\vert\int_0^\infty\!\!\int_{\R^d}
e^{-(\lambda+1) t+i\xi\cdot x}G(t,x)dxdt\bigg\vert {d\lambda\over \lambda^{1-\alpha_0}}
\exp\Big\{-(\epsilon_j+\epsilon_k)\vert\xi\vert^2\Big\}\mu(d\xi)\\
&\le\Gamma(\alpha_0)^{-1}\int_{\R_+\times\R^d}\bigg\vert\int_0^\infty\!\!\int_{\R^d}
e^{-(\lambda+1)t+i\xi\cdot x}G(t,x)dxdt\bigg\vert {d\lambda\over \lambda^{1-\alpha_0}}
\mu(d\xi)
\end{aligned}
$$
and
\begin{align}\label{S-12}
\|G_l\|_{j,k}^{(2)}&=\bigg(\int_{(\R_+\times\R^d)^2}\vert s-t\vert^{-\alpha_0}
\gamma_{\epsilon_j+\epsilon_k}(x-y)
e^{-s}G(s,x) e^{-t}G(t,x)dsdxdtdy\bigg)^{1/2}\\
&\le \bigg(\int_{(\R_+\times\R^d)^2}\vert s-t\vert^{-\alpha_0}
\gamma(x-y)
e^{-s}G(s,x) e^{-t}G(t,x)dsdxdtdy\bigg)^{1/2}\nonumber
\end{align}
where the inequality follows from the procedure of Fourier transform
$$
\begin{aligned}
&\int_{(\R_+\times\R^d)^2}\vert s-t\vert^{-\alpha_0}
\gamma_{\epsilon_j+\epsilon_k}(x-y)
e^{-s}G(s,x) e^{-t}G(t,x)dsdxdtdy\\
&=\int_{\R^{d+1}}\bigg\vert\int_{\R_+\times\R^d}\exp\big\{-t+i\lambda t+i\xi\cdot x\big\}G(t,x)dtdx\bigg\vert^2
\mu^0(d\lambda)\exp\{-(\epsilon_j+\epsilon_k)\vert\xi\vert^2\}\mu(d\xi)\\
&\le \int_{\R^{d+1}}\bigg\vert\int_{\R_+\times\R^d}\exp\big\{-t+i\lambda t+i\xi\cdot x\big\}G(t,x)dtdx\bigg\vert^2
\mu^0(d\lambda)\mu(d\xi)\\
&=\int_{(\R_+\times\R^d)^2}\vert s-t\vert^{-\alpha_0}
\gamma(x-y)
e^{-s}G(s,x) e^{-t}G(t,x)dsdxdtdy
\end{aligned}
$$
where $\mu^0(d\lambda)$ is the spectral measure of $\vert\cdot\vert^{-\alpha_0}$.

By  Lemma \ref{A} (with $\theta=1$), all above bounds are finite.

As for $(j,k)=(j_1,k_1)$,
$$
\|G_l\|_{j_1,k_1}^{(1)}\le\Gamma(\alpha_0)^{-1}\int_{\eta^{-1}}^\infty d\lambda\int_{\R^d}
{1\over (\lambda+1)^2+\vert\xi\vert^2}{d\lambda\over \lambda^{1-\alpha_0}}\mu(d\xi)
\to 0\hskip.2in (\eta\to 0^+),
$$
$$
\begin{aligned}
\|G_l\|_{j_1,k_1}^{(2)}&=
\bigg\{\int_{(\R_+\times\R^d)^2}\Big(\gamma_0^0-\gamma_\eta^0\Big)(s-t)
\gamma_{\epsilon_j+\epsilon_k}(x-y)e^{-t}G(t,x)e^{-s}G(s,y)dsdxdtdy\bigg\}^{1/2}\\
&\le\bigg\{\int_{\R_+\times\R_+}dsdt e^{-(s+t)}\Big(\gamma_0^0-\gamma_\eta^0\Big)(s-t)
\int_{\R^d\times\R^d}\gamma(x-y)G(t,x)G(s,y)dxdy\bigg\}^{1/2}
\end{aligned}
$$
where the inequality follows for the reason similar to the one for (\ref{S-12}).
By  (\ref{A-2}), Lemma \ref{A} (with $\theta=1$) and dominating control theorem, the right hand side
converges as $\eta\to 0^+$. \qed

\section{Proof of Theorem \ref{th-1}}\label{E}

For the part (1) of Theorem \ref{th-1}, we prove that the Stratonovich expansion (\ref{M-3}) solves
(\ref{M-1}). According to Definition \ref{d.mild_solution}, all we need is to show

\begin{enumerate}
	\item[(i)] The random  series in (\ref{M-3}) converges in ${\cal L}^2(\Omega, {\cal A},\P)$.
			\item[(ii)] The random field $\Psi(s,y)=G(t-s, x-y)u(s,y)1_{[0, t]}(s)$ is Stratonovich
			integrable and satisfies (\ref{M-1}).   
\end{enumerate}

Our way for (i) is to show
\begin{align}\label{E-1}
\sum_{n=0}^\infty \Big\{\E \big[S_n(g_n(\cdot, t, x)\big)\big]^2\Big\}^{1/2}<\infty \hskip.2in \forall t>0.
\end{align}

We start at (\ref{S-5}). Taking
$$
\epsilon=(\epsilon_1,\cdots, \epsilon_n)=(\epsilon_{n+1},\cdots \epsilon_{2n})
=\tilde{\epsilon}\hskip.1in\hbox{and}\hskip.1in
\delta=(\delta_1,\cdots, \delta_n)=(\delta_{n+1},\cdots \delta_{2n})
=\tilde{\delta}
$$
we have
\begin{align}\label{E-2}
&\int_0^\infty\!\int_0^\infty d\bar{t}_1d\bar{t}_2
\exp\Big\{-{n\over t}(\bar{t}_1+\bar{t}_2)\Big\}\E \big[S_{n,\epsilon,\delta}\big(g_n(\cdot, \bar{t}_1, 0)\big)
S_{n,\epsilon,\delta}\big(g_n(\cdot, \bar{t}_2, 0)\big)\big]\\
&=\sum_{{\cal D}\in\Pi_n}\int_0^\infty\!\int_0^\infty d\bar{t}_1d\bar{t}_2
\exp\Big\{-{n\over t}(\bar{t}_1+\bar{t}_2)\Big\}\int_{[0, \bar{t}_1]_<^n\times[0, \bar{t}_2]_<^n}ds_1\cdots ds_{2n}\int_{(\R^d)^n\times(\R^d)^n}
dx_1\cdots dx_{2n}\nonumber\\
&\times\bigg(\prod_{\rho=1}^2\prod_{l\in I_\rho}G(s_l-s_{l-1}, x_l-x_{l-1})\bigg)
\prod_{(j,k)\in{\cal D}}
\gamma_{\delta_j\vee\delta_k}^0(s_j-s_k)\gamma_{\epsilon_j+\epsilon_k}(x_j-x_k)\nonumber
\end{align}
For each ${\cal D}\in\Pi_n$,
$$
\begin{aligned}
&\int_0^\infty\!\int_0^\infty d\bar{t}_1d\bar{t}_2
\exp\Big\{-{n\over t}(\bar{t}_1+\bar{t}_2)\Big\}\int_{[0, \bar{t}_1]_<^n\times[0, \bar{t}_2]_<^n}ds_1\cdots ds_{2n}\int_{(\R^d)^n\times(\R^d)^n}
dx_1\cdots dx_{2n}\\
&\times\bigg(\prod_{\rho=1}^2\prod_{l\in I_\rho}G(s_l-s_{l-1}, x_l-x_{l-1})\bigg)
\prod_{(j,k)\in{\cal D}}
\gamma_{\delta_j\vee\delta_k}^0(s_j-s_k)\gamma_{\epsilon_j+\epsilon_k}(x_j-x_k)\\
&=\Big({t\over n}\Big)^2\int_{(\R_+)_<^n\times(\R_+)_<^n}ds_1\cdots ds_{2n}\int_{(\R^d)^n\times(\R^d)^n}
dx_1\cdots dx_{2n}\\
&\times\bigg(\prod_{\rho=1}^2\prod_{l\in I_\rho}e^{-nt^{-1}(s_l-s_{l-1})}G(s_l-s_{l-1}, x_l-x_{l-1})\bigg)
\prod_{(j,k)\in{\cal D}}
\gamma_{\delta_j\vee\delta_k}^0(s_j-s_k)\gamma_{\epsilon_j+\epsilon_k}(x_j-x_k).
\end{aligned}
$$
Applying Theorem \ref{th-3} to the right hand side we have the bound
$$
\Big({t\over n}\Big)^2\bigg(\prod_{l\in Q_0}\|G_l\|^{(0)}\bigg)\bigg(\prod_{l\in Q_1}2\|G_l\|_{(\cdot, \cdot)}^{(1)}\bigg)
\bigg(\prod_{l\in Q_2}\|G_l\|_{(\cdot, \cdot)}^{(2)}\bigg)
$$
for $G_l(\bar{t},x)=e^{-nt^{-1}\bar{t}}G(\bar{t}, x)$ ($l=1,\cdots, 2n$). Further
$$
\|G_l\|^{(0)}=\int_0^\infty \int_{\R^d}e^{-nt^{-1}\bar{t}}G(\bar{t}, x)dxd\bar{t}
=\int_0^\infty e^{-nt^{-1}\bar{t}}\bar{t}d\bar{t}=\Big({t\over n}\Big)^{2n}
$$
where the first equality follows from (\ref{M-11-1}). By (i), Lemma \ref{A-1}
$$
\begin{aligned}
\|G_l\|_{j,k}^{(1)}&=\int_{\R_+\times\R^d}\bigg\vert\int_0^\infty\!\!\int_{\R^d}
e^{-(nt^{-1}+\lambda)\bar{t}+i\xi\cdot x}G(\bar{t},x)dxd\bar{t}\bigg\vert {d\lambda\over \lambda^{1-\alpha_0}}
\exp\{-(\epsilon_j+\epsilon_k)\vert\xi\vert^2\}\mu(d\xi)\\
&\le\int_{\R_+\times\R^d}\bigg\vert\int_0^\infty\!\!\int_{\R^d}
e^{-(nt^{-1}+\lambda)\bar{t}+i\xi\cdot x}G(\bar{t},x)dxd\bar{t}\bigg\vert {d\lambda\over \lambda^{1-\alpha_0}}
\mu(d\xi)\le C<\infty.
\end{aligned}
$$
By (ii), Lemma \ref{A-1} with $\theta=nt^{-1}$,
$$
\begin{aligned}
\|G_l\|_{j,k}^{(2)}&=\bigg(\int_{(\R_+\times\R^d)^2}\vert s-t\vert^{-\alpha_0}
\gamma_{\epsilon_j+\epsilon_k}(x-y)
e^{-nt^{-1}\bar{s}}G(s,x) e^{-nt^{-1}\bar{t}}G(\bar{t},x)dsdxdtdy\bigg)^{1/2}\\
&\le \bigg(\int_{(\R_+\times\R^d)^2}\vert s-t\vert^{-\alpha_0}
\gamma(x-y)
e^{-nt^{-1}\bar{s}}G(s,x) e^{-nt^{-1}\bar{t}}G(\bar{t},x)dsdxdtdy\bigg)^{1/2}
\le C{t\over n}
\end{aligned}
$$
where the second step follows for the same reason as the one in (\ref{S-12}). 
Therefore, the bound can be rewritten as
$$
\Big({t\over n}\Big)^2\Big({t\over n}\Big)^{2\#(Q_0)}(2C)^{\#(Q_1}\Big({t\over n}\Big)^{\#(Q_2)}
\le C^n\Big({t\over n}\Big)^2\Big({t\over n}\Big)^{2(\#(Q_0)+2^{-1}\#(Q_2))}.
$$
According to Theorem \ref{th-3}, $\#(Q_0)+\#(Q_1)+\#(Q_2)=2n$ and $\#(Q_0)=\#(Q_1)$.
Therefore, $\#(Q_0)+2^{-1}\#(Q_2)=n$. In summary,
$$
\begin{aligned}
&\int_0^\infty\!\int_0^\infty d\bar{t}_1d\bar{t}_2
\exp\Big\{-{n\over t}(\bar{t}_1+\bar{t}_2)\Big\}\int_{[0, \bar{t}_1]_<^n\times[0, \bar{t}_2]_<^n}ds_1\cdots ds_{2n}\int_{(\R^d)^n\times(\R^d)^n}
dx_1\cdots dx_{2n}\\
&\times\bigg(\prod_{\rho=1}^2\prod_{l\in I_\rho}G(s_l-s_{l-1}, x_l-x_{l-1})\bigg)
\prod_{(j,k)\in{\cal D}}
\gamma_{\delta_j\vee\delta_k}^0(s_j-s_k)\gamma_{\epsilon_j+\epsilon_k}(x_j-x_k)\\
&\le C^n\Big({t\over n}\Big)^{2n+2}
\end{aligned}
$$
where the constant $C>0$ is independent of $\epsilon$, $n$, $t$ and ${\cal D}$. By (\ref{E-2}), therefore,
$$
\int_0^\infty d\bar{t}\exp\Big\{-{n\over t}\bar{t}\Big\}\E \big[S_{n,\epsilon,\delta}\big(g_n(\cdot, \bar{t}, 0)\big)\big]^2
\le \#(\Pi_n)C^n\Big({t\over n}\Big)^{2n+2}=C^n{(2n)!\over 2^n n!}
\Big({t\over n}\Big)^{2n+2}.
$$
On the other hand,
$$
\begin{aligned}
&\int_0^\infty\!\int_0^\infty d\bar{t}_1d\bar{t}_2
\exp\Big\{-{n\over t}(\bar{t}_1+\bar{t}_2)\Big\}\E \big[S_{n,\epsilon,\delta}\big(g_n(\cdot, \bar{t}_1, 0)\big)
S_{n,\epsilon,\delta}\big(g_n(\cdot, \bar{t}_2, 0)\big)\big]\\
&\ge \E \big[S_{n,\epsilon,\delta}\big(g_n(\cdot, 0)\big)\big]^2\int_t^\infty \!\int_0^\infty d\bar{t}_1d\bar{t}_2
d\bar{t}_1d\bar{t}_2
\exp\Big\{-{n\over t}(\bar{t}_1+\bar{t}_2)\Big\}\\
&=\Big({t\over n}\Big)^2e^{-2n}\E \big[S_{n,\epsilon,\delta}\big(g_n(\cdot, t, 0)\big)\big]^2.
\end{aligned}
$$
In summary, we have the bound
\begin{align}\label{E-3}
\E \big[S_{n,\epsilon,\delta}\big(g_n(\cdot, t, 0)\big)\big]^2\le C^n{(2n)!\over 2^n n!}
\Big({t\over n}\Big)^{2n}\le C^n{t^{2n}\over n!}
\end{align}
with the constant $C>0$ independent of $\epsilon=(\epsilon_1,\cdots, \epsilon_n)$, $\delta=(\delta_1,\cdots,\delta_n)$,
$n$ and $t$. Letting $\epsilon,\delta\to 0^+$. By (\ref{S-2}), therefore,
\begin{align}\label{E-4}
\E \big[S_n\big(g_n(\cdot, t, 0)\big)\big]^2\le C^n{t^{2n}\over n!}.
\end{align}
In particular, we have (\ref{E-1}).

To confirm (ii). Let $\epsilon_1, \delta_1>0$. In view of (\ref{M-3}), justified by (\ref{E-4})
\begin{align}\label{E-5}
&1+\int_0^t\!\int_{\R^d}G(t-s, y-x)u(s,y)\dot{W}_{\epsilon_1,\delta_1}(s,y)dyds\\
&=1+\sum_{n=1}^\infty\int_0^t\!\int_{\R^d}G(t-s, y-x)S_{n-1}\big(g_{n-1}(\cdot, s, y)\big)\dot{W}_{\epsilon_1,\delta_1}(s,y)dyds.\nonumber
\end{align}
justified by (\ref{E-4}). By taking the limit properly in (\ref{E-3}) 
$$
\E\bigg[\int_0^t\!\int_{\R^d}G(t-s, y-x)S_{n-1}\big(g_{n-1}(\cdot, s, y)\big)\dot{W}_{\epsilon_1,\delta_1}(s,y)dyds
\bigg]^2
\le {C^{n}\over n!}t^{2n}\hskip.2in n=0,1,\cdots.
$$
In view of the definition in (\ref{e.def_integral}),
by dominated convergence theorem, all we need is
\begin{align}\label{E-6}
\lim_{\epsilon_1,\delta_1\to 0^+}\int_0^t\!\int_{\R^d}G(t-s, y-x)S_{n-1}\big(g_{n-1}(\cdot, s, y)\big)
\dot{W}_{\epsilon_1,\delta_1}(s,y)dyds=S_n\big(g_n(\cdot, t, x)\big)
\end{align}
in ${\cal L}^2(\Omega, {\cal A}, \P)$ for each $n\ge 1$. Indeed, let
$\epsilon'=(\epsilon_2,\cdots, \epsilon_n)$ and $\delta'=(\delta_2,\cdots,\delta_n)$.
By the definition of $g_n(\cdot, t,x)$ given in (\ref{M-14}),
$$
\begin{aligned}
&\int_0^t\!\int_{\R^d}G(t-s, y-x)S_{n-1, \epsilon',\delta'}\big(g_{n-1}(\cdot, s, y)\big)
\dot{W}_{\epsilon_1,\delta_1}(s,y)dyds\\
&=\int_{(\R^+\times\R^d)^n}g_n(s_1,\cdots, s_n, x_1,\cdots, x_n, t, x)
\Big(\prod_{k=1}^nW_{\epsilon_k, \delta_k}(s_k,x_k)\Big)ds_1\cdots ds_ndx_1\cdots dx_n\\
&=S_{n, \epsilon, \delta}\big(g_n(\cdot, t, x)\big).
\end{aligned}
$$
Therefore, by (\ref{S-2}) in Theorem \ref{th-4},
$$
\begin{aligned}
&\lim_{\epsilon_1,\delta_1\to 0^+}\int_0^t\!\int_{\R^d}G(t-s, y-x)S_{n-1}\big(g_{n-1}(\cdot, s, y)\big)
\dot{W}_{\epsilon_1,\delta_1}(s,y)dyds\\
&=\lim_{\epsilon_1,\delta_1\to 0^+}\lim_{\epsilon',\delta'\to 0^+}
\int_0^t\!\int_{\R^d}G(t-s, y-x)S_{n-1}\big(g_{n-1, \epsilon',\delta'}(\cdot, s, y)\big)
\dot{W}_{\epsilon_1,\delta_1}(s,y)dyds\\
&=\lim_{\epsilon,\delta\to 0^+}S_{n, \epsilon, \delta}\big(g_n(\cdot, t, x)\big)=S_n\big(g_n(\cdot, t, x)\big)
\hskip.2in \hbox{in ${\cal L}^2(\Omega, {\cal A}, \P)$}.
\end{aligned}
$$
So we have proved (\ref{E-6}), and therefore Part (1) of Theorem \ref{th-1}.

\medskip

To prove Part (2) of Theorem \ref{th-1}, all we need is to show that the condition (\ref{intro-5}) is necessary for
	$$
	\E S_2\big(g_2(\cdot, t, 0)\big)<\infty
	$$
	with any $t>0$. Indeed, by (\ref{M-10-20})
	$$
	\begin{aligned} 
		&\E S_2\big(g_2(\cdot, t, 0)\big)\\
		&=\int_{[0,t]_<^2}ds_1ds_2\int_{(\R^d)^2}(s_2-s_1)^{-\alpha_0}\gamma(x_2-x_1)
		G(s_1, x_1)G(s_2-s_1, x_2-x_1)dx_1dx_2\\
		&=\int_{[0,t]_<^2}ds_1ds_2(s_2-s_1)^{-\alpha_0}
		\int_{(\R^d)^2}\gamma(x_2-x_1)G(s_1,x_1)G(s_2-s_1, x_2-x_1)dx_1dx_2\\
		& =\int_{[0,t]_<^2}(s_2-s_1)^{-\alpha_0}\bigg(\int_{\R^d}G(s_1,x)dx\bigg)\bigg(\int_{\R^d}\gamma(x)G(s_2-s_1, x)dx\bigg)ds_1ds_2\\
		&=\int_{[0,t]_<^2} s_1(s_2-s_1)^{-\alpha_0}\bigg[\int_{\R^d}{\sin(\vert\xi\vert (s_2-s_1))\over\vert\xi\vert}\mu(d\xi)\bigg]ds_1ds_2\\
		&=\int_{\R^d}{\mu(d\xi)\over\vert\xi\vert}\int_0^ts_1ds_1\int_0^{t-s_1}{\sin(s_2\vert\xi\vert\over s_2^{\alpha_0}}ds_2
		=\int_{\R^d}{\mu(d\xi)\over\vert\xi\vert^{2-\alpha_0}}\int_0^t s_1ds_1\int_0^{(t-s_1)\vert\xi\vert}
		{\sin s_2\over s_2^{\alpha_0}}ds_2.
\end{aligned}
	$$
Notice that
\begin{align}\label{E-7}
\int_0^a{\sin s_2\over s_2^{\alpha_0}}ds_2> 0
\end{align}
for any $0<a\le \pi$. When $a>\pi$,
$$
\int_0^a{\sin s_2\over s_2^{\alpha_0}}ds_2\ge \int_0^\pi{\sin s_2\over s_2^{\alpha_0}}ds_2+
\int_\pi^{2\pi}{\sin s_2\over s_2^{\alpha_0}}ds_2\equiv\delta>0.
$$
In particular, (\ref{E-7}) holds for any $a>0$.
 Therefore,
	$$
	\begin{aligned} 
		&\E S_2\big(g_2(\cdot, t, 0)\big)\\
&\ge \int_{\{\vert\xi\vert\ge 2\pi t^{-1}\}}{\mu(d\xi)\over\vert\xi\vert^{2-\alpha_0}}\int_0^{t/2}s_1
\bigg(\int_0^{(t-s_1)\vert\xi\vert}
		{\sin s_2\over s_2^{\alpha_0}}ds_2\bigg)ds_1\\
		&\ge\delta\int_{\{\vert\xi\vert\ge 2\pi t^{-1}\}}{\mu(d\xi)\over\vert\xi\vert^{2-\alpha_0}}\int_0^{t/2}s_1ds_1
=\delta {t^2\over 8}\int_{\{\vert\xi\vert\ge 2\pi t^{-1}\}}{\mu(d\xi)\over\vert\xi\vert^{2-\alpha_0}}.
	\end{aligned}
	$$	
	Clearly, the finiteness on the left hand side leads to  the condition (\ref{intro-5}).
\qed

\begin{remark}\label{R-3} From (\ref{M-3}) and (\ref{S-4}), we have the bound
\begin{align}\label{E-8}
\E u^2(t,x)\le e^{Ct^2}
\end{align}
for large $t$.
\end{remark}

\section{A Stratonovich moment representation}\label{B}

Let $\beta(t)$ and $B(t)$ be 1-dimensional and $d$-dimensional Brownian motions, respectively. In the rest of the 
paper, we assume the independence among $\beta(t)$, $B(t)$ and $\dot{W}(t,x)$ and use the notations $\E_0$ and
$\P_0$ for the expectation and probability with respect to $\beta(t)$ and $B(t)$ when $B(0)=0$ and $\beta(0)=0$ (which is the case for most of the time).

\begin{theorem}\label{th-5} Under the assumption (\ref{intro-5}),
\begin{align}\label{B-1}
&\int_0^\infty e^{-\theta t} \E S_{2n}\big(g_{2n}(\cdot, t,0)\big)dt\\
&={\theta\over 2}\Big({1\over 2}\Big)^{3n}{1\over n!}\int_0^\infty dt\exp\Big\{-{\theta^2\over 2}t\Big\}\cr
&\times\E_0\bigg[\int_0^t\!\int_0^t\Big(\theta\vert s-r\vert+i \big(\beta(s)-\beta(r)\big)\Big)^{-\alpha_0}
\gamma\big(B(s)-B(r)\big)dsdr
\bigg]^n\nonumber
\end{align}
for any $\theta>0$ and $n\ge 1$. In addition,
\begin{align}\label{B-2}
&\int_0^t\!\int_0^t\Big(\theta\vert s-r\vert+i \big(\beta(s)-\beta(r)\big)\Big)^{-\alpha_0}
\gamma\big(B(s)-B(r)\big)dsdr\\
&=\Gamma(\alpha_0)^{-1}
\E^\kappa\int_{\R_+\times\R^d}\bigg\vert \int_0^t\exp\Big\{i\lambda\big(\theta\kappa(s)+\beta(s)\big)+i\xi\cdot B(s)\Big\}ds\bigg\vert^2{d\lambda\over\lambda^{1-\alpha_0}}\mu(d\xi)\ge 0\nonumber
\end{align}
where $\kappa(t)$ is a standard Cauchy process independent of the Brownian motions.
\end{theorem}

\begin{remark}\label{B-0} The appearance of the Brownian motion $\beta(t)$ is responsable
  for the difference in local behaviors
  between hyperbolic and parabolic equations
  when it comes to the setting of time-dependent Guassian field.
  By the well-known Feynman-Kac formula,
  $\E S_{2n}\big(g_{2n}(\cdot, t,0)\big)$  in the parabolic case
  is given as a constant multiple of
  the form
  $$
  {C^n\over n!}\E_0\bigg[\int_0^t\!\int_0^t\vert s-r\vert^{-\alpha_0}
\gamma\big(B(s)-B(r)\big)dsdr
\bigg]^n
$$
where the time-singularity contributed by the Gaussian field is measured
by $\vert s-r\vert^{-\alpha_0}$ for closed $r$ and $s$. In the hyperbolic
system, the time-singularity brought by the Gaussian field is measured
by
$$
\Big\vert\vert s-r\vert+i \big(\beta(s)-\beta(r)\big)\vert^{-\alpha_0}
\approx\Big\vert(s-r)^2+\vert s-r\vert\Big\vert^{-\alpha_0/2}
\approx \vert s-r\vert^{-\alpha_0/2}
$$
for closed $r$ and $s$. It explains, for example, how the gap between
conditions (\ref{intro-5}) and (\ref{intro-7}) of existence is created.

\end{remark}

\proof Since $\gamma(\cdot)$ may exist  as generalized function, we may encounter
some legality issue. Thanks to Theorem \ref{th-4}, we are allowed to proceed with
a point-wise defined $\gamma(\cdot)$, for otherwise we use $\gamma_\epsilon(\cdot)$
instead.

Recall (\ref{M-10-20}) that
\begin{align}\label{B-3}
\E S_{2n}\big(g_{2n}(\cdot, t,0)\big)&=\sum_{{\cal D}\in\Pi_n}\int_{[0,t]_<^{2n}}ds_1\cdots ds_{2n}
\int_{(\R^d)^{2n}}dx_1\cdots dx_{2n}\\
&\times\bigg(\prod_{l=1}^{2n}G(s_l-s_{l-1},x_l-x_{l-1})\bigg)
\prod_{(j,k)\in{\cal D}}\vert s_j-s_k\vert^{-\alpha_0}\gamma(x_j-x_k).\nonumber
\end{align}
Let ${\cal D}\in \Pi_n$ be fixed. For any $(j,k)\in {\cal D}$, we introduce the rule that $j<k$.
Consequently, by (\ref{intro-18})
$$
\prod_{(j,k)\in{\cal D}}\vert s_j-s_k\vert^{-\alpha_0}
=\big(\Gamma(\alpha_0)\big)^{-n}\int_{\R_+^n}\bigg(\prod_{(j, k)\in {\cal D}}{d\lambda_{j,k}\over \lambda_{j,k}^{1-\alpha_0}}\bigg)
\exp\bigg\{-\sum_{(j,k)\in{\cal D}}\lambda_{j,k}(s_k-s_j)\bigg\}
$$
as  $(s_1,\cdots, s_{2n})\in [0,t]_<^{2n}$. Write
$$
\sum_{(j,k)\in{\cal D}}\lambda_{j,k}(s_k-s_j)
=\sum_{l=1}^{2n}q_ls_l=\sum_{l=1}^{2n}\bigg(\sum_{i=l}^{2n}q_i\bigg)(s_l-s_{l-1})
=\sum_{l=1}^{2n}c_l(s_l-s_{l-1})
$$
where $q_l$ is equal to $\lambda_{j,k}$ or $-\lambda_{j,k}$ for some $(j,k)\in{\cal D}$. Since $j<k$,
$$
c_l=\sum_{i=l}^{2n}q_i=\sum_{j<l\le k}\lambda_{j,k}\ge 0\hskip.2in 1\le l\le 2n.
$$
We have
$$
\begin{aligned}
&\int_{[0,t]_<^{2n}}ds_1\cdots ds_{2n}\bigg(\prod_{l=1}^{2n}G(s_l-s_{l-1},x_l-x_{l-1})\bigg)
\prod_{(j,k)\in{\cal D}}\vert s_j-s_k\vert^{-\alpha_0}\\
&=\big(\Gamma(\alpha_0)\big)^{-n}\int_{\R_+^n}\prod_{(j,k)\in{\cal D}}{d\lambda_{j,k}\over\lambda_{j,k}^{1-\alpha_0}}
\int_{[0,t]_<^{2n}}ds_1\cdots ds_{2n}\prod_{l=1}^{2n}e^{-c_l(s_l-s_{l-1})}G(s_l-s_{l-1},x_l-x_{l-1}).
\end{aligned}
$$
Using the  identity (Lemma 2.2.7, p.39, \cite{Chen-1})
\begin{align}\label{B-4}
\theta\int_0^\infty e^{-\theta t}\int_{[0, t]_<^{2n}}\bigg(\prod_{l=1}^{2n}\varphi_l(s_l-s_{l-1})\bigg)
ds_1\cdots ds_{2n}=\prod_{l=1}^{2n}\int_0^\infty \varphi_l(t)dt
\end{align}
we have
$$
\begin{aligned}
&\int_0^\infty dte^{-\theta t}\int_{[0,t]_<^{2n}}ds_1\cdots ds_{2n}\bigg(\prod_{l=1}^{2n}G(s_l-s_{l-1},x_l-x_{l-1})\bigg)\prod_{(j,k)\in{\cal D}}\vert s_j-s_k\vert^{-\alpha_0}\\
&=\big(\Gamma(\alpha_0)\big)^{-n}
\theta^{-1}\int_{\R_+^n}\bigg(\prod_{(j,k)\in{\cal D}}{d\lambda_{j,k}\over \lambda_{j,k}^{1-\alpha_0}}\bigg)
\prod_{l=1}^{2n}\int_0^\infty e^{-\theta t}e^{-c_lt}G(t,x_l-x_{l-1})dt.
\end{aligned}
$$
Noticing $c_l\ge 0$, by (\ref{intro-16})
$$
\begin{aligned}
&\int_0^\infty e^{-\theta t}e^{-c_lt}G(t,x_l-x_{l-1})dt=\int_0^\infty e^{-(\theta +c_l)t}G(t,x_l-x_{l-1})dt\\
&={1\over 2}\int_0^\infty \exp\Big\{-{1\over 2}(\theta+c_l)^2 t\Big\}p(t,x_l-x_{l-1})dt\\
&={1\over 2}\int_0^\infty \exp\Big\{-{\theta^2\over 2}t\Big\}\exp\Big\{-\theta c_lt-{c_l^2\over 2}t\Big\}p(t,x_l-x_{l-1})dt
\end{aligned}
$$
where $p(t,x)$ is the Brownian semi-group introduced in (\ref{intro-17}). Thus
$$
\begin{aligned}
&\int_0^\infty dte^{-\theta t}\int_{[0,t]_<^{2n}}ds_1\cdots ds_{2n}\bigg(\prod_{l=1}^{2n}G(s_l-s_{l-1},x_l-x_{l-1})\bigg)\prod_{(j,k)\in{\cal D}}\vert s_j-s_k\vert^{-\alpha_0}\\
&=\big(\Gamma(\alpha_0)\big)^{-n}\Big({1\over 2}\Big)^{2n}
{1\over \theta}\int_{\R_+^n}\bigg(\prod_{(j,k)\in{\cal D}}{d\lambda_{j,k}\over \lambda_{j,k}^{1-\alpha_0}}\bigg)\\
&\times\prod_{l=1}^{2n}\int_0^\infty\exp\Big\{-{\theta^2\over 2}t\Big\}\exp\Big\{-\theta c_lt-{c_l^2\over 2}t\Big\}p(t,x_l-x_{l-1})dt\\
&=\big(\Gamma(\alpha_0)\big)^{-n}\Big({1\over 2}\Big)^{2n}
{\theta\over 2}\int_0^\infty dt\exp\Big\{-{\theta^2\over 2}t\Big\}\int_{\R_+^n}\bigg(\prod_{(j,k)\in{\cal D}}{d\lambda_{j,k}\over \lambda_{j,k}^{1-\alpha_0}}\bigg)\\
&\times\int_{[0,t]_<^{2n}}ds_1\cdots ds_{2n}
\prod_{l=1}^{2n}\exp\Big\{-\theta c_l(s_l-s_{l-1})-{c_l^2\over 2}(s_l-s_{l-1})\Big\}p(s_l-s_{l-1}, x_l-x_{l-1})
\end{aligned}
$$
where the last step follows from (\ref{B-4}).

Multiplying the factor
$$
\prod_{(j,k)\in{\cal D}}\gamma(x_j-x_k)
$$
and integrating over $(x_1,\cdots, x_{2n})$ on the both sides,
$$
\begin{aligned}
&\int_0^\infty dte^{-\theta t}\int_{(\R^d)^{2n}}dx_1\cdots dx_{2n}\int_{[0,t]_<^{2n}}ds_1\cdots ds_{2n}\bigg(\prod_{l=1}^{2n}G(s_l-s_{l-1},x_l-x_{l-1})\bigg)\\
&\times\prod_{(j,k)\in{\cal D}}\vert s_j-s_k\vert^{-\alpha_0}\gamma(x_j-x_k)\\
&=\big(\Gamma(\alpha_0)\big)^{-n}\Big({1\over 2}\Big)^{2n}
{\theta\over 2}\int_0^\infty dt\exp\Big\{-{\theta^2\over 2}t\Big\}
\int_{[0,t]_<^{2n}}ds_1\cdots ds_{2n}\\
&\times
\int_{\R_+^n}\bigg(\prod_{(j,k)\in{\cal D}}{d\lambda_{j,k}\over \lambda_{j,k}^{1-\alpha_0}}\bigg)
\bigg(\prod_{l=1}^{2n}\exp\Big\{-\theta c_l(s_l-s_{l-1})-{c_l^2\over 2}(s_l-s_{l-1})\Big\}\bigg)\\
&\times\int_{(\R^d)^{2n}}\bigg(\prod_{l=1}^{2n}p(s_l-s_{l-1}, x_l-x_{l-1})\bigg)
\bigg(\prod_{(j,k)\in{\cal D}}\gamma(x_j-x_k)\bigg)dx_1\cdots x_{2n}.
\end{aligned}
$$
By the fact that
$$
f(x_1,\cdots, x_{2n})=\prod_{l=1}^{2n}p(s_l-s_{l-1}, x_l-x_{l-1})
$$
is the joint density of the random vector $\big(B(s_1),\cdots, B(s_{2n})\big)$,
$$
\begin{aligned}
&\int_{(\R^d)^{2n}}\bigg(\prod_{l=1}^{2n}p(s_l-s_{l-1}, x_l-x_{l-1})\bigg)
\bigg(\prod_{(j,k)\in{\cal D}}\gamma(x_j-x_k)\bigg)dx_1\cdots x_{2n}\\
&=\E_0\prod_{(j,k)\in{\cal D}}\gamma\big(B(s_j)-B(s_k)\big).
\end{aligned}
$$
In summary,
$$
\begin{aligned}
&\int_0^\infty dte^{-\theta t}\int_{(\R^d)^{2n}}dx_1\cdots dx_{2n}\int_{[0,t]_<^{2n}}ds_1\cdots ds_{2n}\bigg(\prod_{l=1}^{2n}G(s_l-s_{l-1},x_l-x_{l-1})\bigg)\cr
&\times\prod_{(j,k)\in{\cal D}}\vert s_j-s_k\vert^{-\alpha_0}\gamma(x_j-x_k)\\
&=\big(\Gamma(\alpha_0)\big)^{-n}\Big({1\over 2}\Big)^{2n}
{\theta\over 2}\!\int_0^\infty dt\exp\Big\{-{\theta^2\over 2}t\Big\}\\
&\times
\E_0\int_{[0,t]_<^{2n}}\!ds_1\cdots ds_{2n}\bigg(\!\prod_{(j,k)\in{\cal D}}\gamma\big(B(s_j)-B(s_k)\big)
\bigg)\\
&\times
\int_{\R_+^n}\bigg(\prod_{(j,k)\in{\cal D}}{d\lambda_{j,k}\over \lambda_{j,k}^{1-\alpha_0}}\bigg)
\exp\bigg\{-\theta\sum_{l=1}^{2n}c_l(s_l-s_{l-1})-{1\over 2}\sum_{l=1}^{2n}c^2_l(s_l-s_{l-1})\bigg\}.
\end{aligned}
$$
Notice that
$$
\begin{aligned}
&\exp\bigg\{-\theta\sum_{l=1}^{2n}c_l(s_l-s_{l-1})-{1\over 2}\sum_{l=1}^{2n}c^2_l(s_l-s_{l-1})\bigg\}\\
&=\E_0\exp\bigg\{-\theta\sum_{l=1}^{2n}c_l(s_l-s_{l-1})
-i\sum_{l=1}^{2n}c_l\big(\beta(s_l)-\beta(s_{l-1})\big)\bigg\}.
\end{aligned}
$$
Recall that
$$
\sum_{l=1}^{2n}c_l(s_l-s_{l-1})=\sum_{(j,k)\in{\cal D}}\lambda_{j,k}(s_k-s_j).
$$
The same algebra leads to
$$
\sum_{l=1}^{2n}c_l\big(\beta(s_l)-\beta(s_{l-1})\big)
=\sum_{(j,k)\in{\cal D}}\lambda_{j,k}\big(\beta(s_k)-\beta(s_j)\big).
$$
By Fubini's theorem
$$
\begin{aligned}
&\int_{\R_+^n}\bigg(\prod_{(j,k)\in{\cal D}}{d\lambda_{j,k}\over \lambda_{j,k}^{1-\alpha_0}}\bigg)
\exp\bigg\{-\theta\sum_{l=1}^{2n}c_l(s_l-s_{l-1})-{1\over 2}\sum_{l=1}^{2n}c^2_l(s_l-s_{l-1})\bigg\}\\
&=\E_0\prod_{(j,k)\in{\cal D}}\int_{\R_+}\exp\Big\{-\theta\lambda(s_k-s_j)-i\lambda\big(\beta(s_k)-\beta(s_j)\big)\Big\}{d\lambda\over\lambda^{1-\alpha_0}}\\
&=\Gamma(\alpha_0)^n
\E_0\prod_{(j,k)\in{\cal D}}\Big(\theta(s_k-s_j)+i\big( \beta(s_k)-\beta(s_j)\big)\Big)^{-\alpha_0}
\end{aligned}
$$
where the last step follows from the identity (p. 183, \cite{LSL})
\begin{align}\label{B-5}
(u+iv)^{-\alpha_0}=\Gamma(\alpha_0)^{-1}\int_0^\infty e^{-\lambda (u+iv)}{d\lambda\over\lambda^{1-\alpha_0}}
\hskip.2in (u,v)\in\R_+\times\R
\end{align}
which appears to be a complex extension of (\ref{intro-18}).

Summarizing our steps,
$$
\begin{aligned}
&\int_0^\infty dte^{-\theta t}\int_{(\R^d)^{2n}}dx_1\cdots dx_{2n}
\int_{[0,t]_<^{2n}}ds_1\cdots ds_{2n}\bigg(\prod_{l=1}^{2n}G(s_l-s_{l-1},x_l-x_{l-1})\bigg)\\
&\times\prod_{(j,k)\in{\cal D}}\vert s_j-s_k\vert^{-\alpha_0}\gamma(x_j-x_k)\cr
&=\Big({1\over 2}\Big)^{2n}
{\theta\over 2}\!\int_0^\infty dt\exp\Big\{-{\theta^2\over 2}t\Big\}\\
&\times
\E_0\int_{[0,t]_<^{2n}}\!ds_1\cdots ds_{2n}\bigg(\!\prod_{(j,k)\in{\cal D}}
\Big(\theta(s_k-s_j)+i\big( \beta(s_k)-\beta(s_j)\big)\Big)^{-\alpha_0}\gamma\big(B(s_k)-B(s_j)\big)
\bigg).
\end{aligned}
$$
Summing up over ${\cal D}\in \Pi_n$,  by (\ref{B-3}),
$$
\begin{aligned}
&\int_0^\infty e^{-\theta t}\E S_{2n}\big(g_{2n}(\cdot, t, 0)\big)dt=\Big({1\over 2}\Big)^{2n}
{\theta\over 2}\!\int_0^\infty dt\exp\Big\{-{\theta^2\over 2}t\Big\}\\
&\times
\sum_{{\cal D}\in\Pi_n}\E_0\int_{[0,t]_<^{2n}}\!ds_1\cdots ds_{2n}\bigg(\!\prod_{(j,k)\in{\cal D}}
\Big(\theta(s_k-s_j)+i\big( \beta(s_k)-\beta(s_j)\big)\Big)^{-\alpha_0}\gamma\big(B(s_k)-B(s_j)\big)
\bigg).
\end{aligned}
$$
Write $(s_k-s_j)=\vert s_k-s_j\vert$ for allowing the following permutation invariance:
$$
\begin{aligned}
&\sum_{{\cal D}\in\Pi_n}\prod_{(j,k)\in{\cal D}}
\Big(\theta\vert s_{\sigma(k)}-s_{\sigma(j)}\vert+i\big( \beta(s_{\sigma(k)})-\beta(s_{\sigma(j)})\big)\Big)^{-\alpha_0}
\gamma\big(B(s_{\sigma(k)})-B(s_{\sigma(j)})\big)\\
&=\sum_{{\cal D}\in\Pi_n}\prod_{(j,k)\in{\cal D}}
\Big(\theta\vert s_k-s_j\vert+i\big( \beta(s_k)-\beta(s_j)\big)\Big)^{-\alpha_0}
\gamma\big(B(s_k)-B(s_j)\big)
\end{aligned}
$$
for any permutation $\sigma$ on $\{1,\cdots, 2n\}$. Consequently,
$$
\begin{aligned}
&\sum_{{\cal D}\in\Pi_n}\int_{[0, t]_<^{2n}}ds_1\cdots ds_{2n}\prod_{(j,k)\in{\cal D}}
\Big(\theta\vert s_k-s_j\vert+i\big( \beta(s_k)-\beta(s_j)\big)\Big)^{-\alpha_0}
\gamma\big(B(s_k)-B(s_j)\big)\\
&={1\over (2n)!}\sum_{{\cal D}\in\Pi_n}\int_{[0, t]^{2n}}ds_1\cdots ds_{2n}\prod_{(j,k)\in{\cal D}}
\Big(\theta\vert s_k-s_j\vert+i\big( \beta(s_k)-\beta(s_j)\big)\Big)^{-\alpha_0}
\gamma\big(B(s_k)-B(s_j)\big)\\
&={1\over (2n)!}\sum_{{\cal D}\in\Pi_n}\prod_{(j,k)\in{\cal D}}
\int_0^t\!\int_0^t\Big(\theta\vert s-r\vert+i\big( \beta(s)-\beta(r)\big)\Big)^{-\alpha_0}\gamma\big(B(s)-B(r)\big)dsdr\\
&={1\over (2n)!}\#(\Pi_n)
\bigg[\int_0^t\!\int_0^t\Big(\theta\vert s-r\vert+i\big( \beta(s)-\beta(r)\big)\Big)^{-\alpha_0}
\gamma\big(B(s)-B(r)\big)dsdr
\bigg]^n
\end{aligned}
$$
Therefore, (\ref{B-1}) follows from the fact that
$$
\#(\Pi_n)={(2n)!\over 2^n n!}.
$$
Finally, by (\ref{B-5})
$$
\begin{aligned}
&\Big(\theta\vert s-r\vert+i\big( \beta(s)-\beta(r)\big)\Big)^{-\alpha_0}\\
&=\Gamma(\alpha_0)
\int_{\R_+}\exp\Big\{-\theta\lambda\vert s-r\vert -i\lambda\big(\beta(s)-\beta(r)\big)\Big\}{d\lambda\over\lambda^{1-\alpha_0}}\\
&=\Gamma(\alpha_0)\E^\kappa\int_{\R_+}\exp\Big\{i\theta\lambda\big(\kappa(s)-\kappa(r)\big) -i\lambda\big(\beta(s)-\beta(r)\big)\Big\}{d\lambda\over\lambda^{1-\alpha_0}}.
\end{aligned}
$$
Therefore, (\ref{B-2}) follows from a standard use of the Fourier transform (\ref{intro-3}) of $\gamma(\cdot)$.
 \qed

\begin{remark} The monotonic order $s_1\le s_2\le\cdots\le s_{2n}$
  in the expression of $\E S_{2n}\big(g_{2n}(\cdot, t, 0)\big)$ (i.e., (\ref{B-3}))
  is a key factor
  that the proof of (\ref{B-2}) can get through. That is the major
  reason why the current idea can not work
for $\E u^p(t,x)$ for $p>1$.
\end{remark}

\section{The time-randomized intersection local time}\label{T}

We assume the assumption in Theorem \ref{th-2}, i.e, (\ref{intro-8})
with $\alpha_0+\alpha<2$ (along with other conditions required for
$\gamma(\cdot)$ to be non-negative definite).
Motivated by Theorem \ref{th-5} and by the
relation (take also (\ref{B-2}) in account)
\begin{align}\label{T-1}
0\le &\int_0^t\!\int_0^t\Big(\theta\vert s-r\vert+i \big(\beta(s)-\beta(r)\big)\Big)^{-\alpha_0}
\gamma\big(B(s)-B(r)\big)dsdr\\
&\le \int_0^t\!\int_0^t\Big\vert\theta\vert s-r\vert+i \big(\beta(s)-\beta(r)\big)\Big\vert^{-\alpha_0}
\gamma\big(B(s)-B(r)\big)dsdr\nonumber\\
&=\int_0^t\!\int_0^t\Big\vert\theta(s-r)+i \big(\beta(s)-\beta(r)\big)\Big\vert^{-\alpha_0}
\gamma\big(B(s)-B(r)\big)dsdr\nonumber
\end{align}
the maim  goal in this section is to
establish the precise large $t$ asymptotics for  the Hamiltonian on the right hand side.
The fact that the two components $(s-r)$ and
$\beta(s)-\beta(r)$ have different scaling rates destroys the homogeneity of the  Hamiltonian.
It also suggests that the contributions from $(s-r)$ and
$\beta(s)-\beta(r)$ are not equal.  Very likely, one of them  completely dominates the game.
The puzzle we face is to tell which one 
of $s-r$ and $\beta(s)-\beta(r)$ is the major player. To put all possible cards on the table we start with some
existing results in literature. First (Theorem 1.1, \cite{Chen-3}), under the Dalang's condition (\ref{intro-14})
(or (\ref{intro-8}) with $\alpha<2$),
\begin{align}\label{T-2}
\lim_{t\to\infty}{1\over t}\log\E_0\exp\bigg\{{b\over t}\int_0^t\!\int_0^t\gamma\big(B(s)-B(r)\big)dsdr\bigg\}
=b^{2\over 2-\alpha}{\cal H}\hskip.2in b>0
\end{align}
where
$$
{\cal H}=\sup_{g\in{\cal F}_d}\bigg\{\int_{\R^d\times\\R^d}\gamma(x-y)g^2(x)g^2(y)dxdy-{1\over 2}\int_{\R^d}
\vert\nabla g(x)\vert^2dx\bigg\}
$$
and ${\cal F}_d=\{g\in W^{1,2}(\R^d); \hskip.1in \|g\|_2=1\}$.

Applying this result to the augmented Brownian motion $\widetilde{B}(t)=\big(\beta(t), B(t)\Big)$
and to the augmented space covariance $\tilde{\gamma}(u,x)=\vert u\vert^{-\alpha_0}\gamma(x)$
(so (\ref{intro-8}) holds with $\tilde{\alpha}=\alpha_0+\alpha<2$),
\begin{align}\label{T-3}
\lim_{t\to\infty}{1\over t}\log\E_0\exp\bigg\{{b\over t}\int_0^t\!\int_0^t\vert \beta(s)-\beta(r)\vert^{-\alpha_0}
\gamma\big(B(s)-B(r)\big)dsdr\bigg\}
=b^{2\over 2-\alpha-\alpha_0}\widetilde{\cal H}
\end{align}
for every $b>0$. This outlines a possible scenario of "$\big(\beta(s)-\beta(r)\big)$-domination".

The  scenario of "$(s-r)$-domination" follows from the pattern (Theorem 1.1, \cite{CHSX} or, (4.10), \cite{Chen-5}) that 
\begin{align}\label{T-4}
\lim_{t\to\infty}{1\over t}\log\E_0\exp\bigg\{{b\over t^{1-\alpha_0}}\int_0^t\!\int_0^t\vert s-r\vert^{-\alpha_0}
\gamma\big(B(s)-B(r)\big)dsdr\bigg\}
=b^{2\over 2-\alpha}{\cal E}_0\hskip.2in b>0
\end{align}
where
\begin{align}\label{T-5}
{\cal E}_0&=\sup_{g\in{\cal A}_d}\bigg\{\int_0^1\!\!\int_0^1\int_{\R^d\times\R^d}{\gamma(x-y)\over \vert s-r\vert^{\alpha_0}}g^2(s, x)g^2(r, y)dxdydsdr\\
&-{1\over 2}\int_0^1\int_{\R^d}\vert \nabla_xg(s,x)\vert^2dxds\bigg\}\nonumber
\end{align}
and ${\cal A}_d=\{g(s, \cdot);\hskip.1in g(s,\cdot)\in {\cal F}_d\hskip.1in\hbox{for each $0\le s\le 1$}\}$

The result (\ref{T-4}) can not directly apply to our setting as it requires the more restrictive
assumption  "$2\alpha_0+\alpha<2$", for otherwise
the left hand side of (\ref{T-4}) blows up even before the limit is taken.
On the other hand,
${\cal E}_0<\infty$ (Lemma 5.2, \cite{Chen-5}) under $\alpha<2$ (and therefore under $\alpha_0+\alpha<2$).
Here we point out that the proof of (\ref{T-4}) (\cite{CHSX}, \cite{Chen-5}) is based on the relation
$$
{1\over t^{1-\alpha_0}}\int_0^t\!\int_0^t\vert s-r\vert^{-\alpha_0}
\gamma\big(B(s)-B(r)\big)dsdr
={1\over t}\int_0^t\!\int_0^t\Big\vert {s-r\over t}\Big\vert^{-\alpha_0}
\gamma\big(B(s)-B(r)\big)dsdr
$$
and on an argument similar to the one used for (\ref{T-2}). By a parallel ( and easier) modification of this idea
we have\footnote{Indeed, we refer an interested reader to the argument
  used Section 4.2, \cite{Chen-5} or, to the proof of (6.16) in \cite{Chen-4}
  for the proof of the upper bound; and to Section 4 and 5 in \cite{CHSX}
  for the proof of lower and upper bounds, respectively.}
that under $\alpha<2$,
\begin{align}\label{T-6}
  \lim_{t\to\infty}{1\over t}\log\E_0\exp\bigg\{{b\over t}
  \int_0^t\!\int_0^t\gamma_\delta^0\Big({s-r\over t}\Big)
\gamma\big(B(s)-B(r)\big)dsdr\bigg\}
=b^{2\over 2-\alpha}{\cal E}_\delta\hskip.2in b>0
\end{align}
for every $\delta>0$, where $\gamma_\delta^0(\cdot)$ is introduced in (\ref{M-17}) and
$$
\begin{aligned}
  {\cal E}_\delta&=\sup_{g\in{\cal A}_d}\bigg\{\int_0^1\!\!\int_0^1
              \int_{\R^d\times\R^d}\gamma_\delta^0(s-r)
\gamma(x-y)g^2(s, x)g^2(r, y)dxdydsdr\\
&-{1\over 2}\int_0^1\int_{\R^d}\vert\nabla g(x)\vert^2dxds\bigg\}.
\end{aligned}
$$

\begin{theorem}\label{th-6} Under $\alpha_0+\alpha<2$,
\begin{align}\label{T-7}
&\lim_{t\to\infty}{1\over t}
\log\E_0\exp\bigg\{{b\over t^{1-\alpha_0}}\int_0^t\!\int_0^t
\Big\vert\theta(s-r)+i\eta\big(\beta(s)-\beta(r)\big)\Big\vert^{-\alpha_0}\gamma\big(B(s)-B(r)\big)drds\bigg\}\nonumber\\
&= b^{2\over 2-\alpha}\theta^{-{2\alpha_0\over 2-\alpha}}{\cal E}_0
\end{align}
for every $b,\theta,\eta >0$.
\end{theorem}

\begin{remark}\label{T-0}
\begin{enumerate}
\item[(1)] Theorem \ref{th-6} clearly
 highlights the pattern of (\ref{T-4}) rather than (\ref{T-3}).
  It is sharply contrary to the local behavior described in Remark \ref{B-0}.
\item[(2)] A challenge for the proof: Playing the game of (\ref{T-4})
  without the ticket ``$2\alpha_0+\alpha<2$'' for that game.
\end{enumerate}
\end{remark}

{\it Proof of the lower bound}. Given $\delta>0$,
$$
\begin{aligned}
&{1\over t^{1-\alpha_0}}
\Big\vert\theta(s-r)+i\eta\big(\beta(s)-\beta(r)\big)\Big\vert^{-\alpha_0}
={\theta^{-\alpha_0}\over t}
\Big\vert{(s-r)+i\theta^{-1}\eta\big(\beta(s)-\beta(r)\big)\over t}\Big\vert^{-\alpha_0}\\
&\ge{\theta^{-\alpha_0}\over t}\gamma_\delta^0\Big({(s-r)+i\eta\theta^{-1}\big(\beta(s)-\beta(r)\big)\over t}\Big)
\ge{\theta^{-\alpha_0}\over t}\gamma_\delta^0
\Big({\vert s-r\vert+\eta\theta^{-1}\big\vert \beta(s)-\beta(r)\big\vert\over t}\Big).
\end{aligned}
$$
Given $a>0$, on $\displaystyle\{\max_{s\le t}\vert\beta(s)\vert\le a\}$,
$$
\begin{aligned}
&\gamma_\delta^0\Big({\vert s-r\vert+\eta\theta^{-1}\big\vert \beta(s)-\beta(r)\big\vert\over t}\Big)
=\Gamma(\alpha_0)^{-1}\int_0^{\delta^{-1}}\exp\Big\{-\lambda{\vert s-r\vert+\eta\theta^{-1}\vert\beta(s)-\beta(r)\vert\over t}\Big\}
{d\lambda\over \lambda^{1-\alpha_0}}\\
&\ge\exp\Big\{-{2\theta^{-1}\eta a\over \delta t}\Big\}\Gamma(\alpha_0)^{-1}\int_0^{\delta^{-1}}
\exp\Big\{-\lambda{\vert s-r\vert\over t}\Big\}{d\lambda\over \lambda^{1-\alpha_0}}
=\exp\Big\{-{2\theta^{-1}\eta a\over \delta t}\Big\}\gamma^0_\delta\Big({s-r\over t}\Big).
\end{aligned}
$$
So we have
$$
\begin{aligned}
&\E_0\exp\bigg\{{b\over t^{1-\alpha_0}}\int_0^t\!\int_0^t
\Big\vert\theta(s-r)+i\eta\big(\beta(s)-\beta(r)\big)\Big\vert^{-\alpha_0}\gamma\big(B(s)-B(r)\big)drds\bigg\}\\
&\ge \E_0\exp\bigg\{{b\theta^{-\alpha_0}\over t}
\exp\Big\{-{2\theta^{-1}\eta a\over \delta t}\Big\}\int_0^t\!\int_0^t\gamma^0_\delta\Big({s-r\over t}\Big)\gamma\big(B(s)-B(r)\big)drds\bigg\}1_{\{\max_{s\le t}\vert\beta(s)\vert\le a\}}\\
&=\E_0\exp\bigg\{{b\theta^{-\alpha_0}\over t}\exp\Big\{-{2\eta a\over \delta t}\Big\}\int_0^t\!\int_0^t\gamma^0_\delta\Big({s-r\over t}\Big)\gamma\big(B(s)-B(r)\big)drds\bigg\}
\P_0\Big\{\max_{s\le t}\vert\beta(s)\vert\le a\Big\}
\end{aligned}
$$
where the last step follows from the independence between $\beta(\cdot)$ and $B(\cdot)$. 

From (\ref{T-6}) (with $b$ being replaced by $b\theta^{-\alpha_0}$)
$$
\begin{aligned}
&\liminf_{t\to\infty}{1\over t}\E_0\exp\bigg\{{b\theta^{-\alpha_0}\over t}\exp\Big\{-{2\theta^{-1}\eta a\over \delta t}\Big\}\int_0^t\!\int_0^t\gamma^0_\delta\Big({s-r\over t}\Big)\gamma\big(B(s)-B(r)\big)drds\bigg\}\\
&\ge b^{2\over 2-\alpha}\theta^{-{2\alpha_0\over 2-\alpha}}{\cal E}_\delta.
\end{aligned}
$$
Recall a well-known fact (e.g., (1.3), \cite{Li-Shao}) that
$$
\lim_{t\to\infty}{1\over t}\log\P_0\Big\{\max_{s\le t}\vert\beta(s)\vert\le a\Big\}
=-{\pi^2\over 8a^2}.
$$
In summary
$$
\begin{aligned}
&\liminf_{t\to\infty}{1\over t}
\log\E_0\exp\bigg\{{b\over t^{1-\alpha_0}}\int_0^t\!\int_0^t
\Big\vert\theta(s-r)+i\epsilon\big(\beta(s)-\beta(r)\big)\Big\vert^{-\alpha_0}\gamma\big(B(s)-B(r)\big)drds\bigg\}\nonumber\\
&\ge b^{2\over 2-\alpha}\theta^{-{2\alpha_0\over 2-\alpha}}{\cal E}_\delta -{\pi^2\over 8a^2}.
\end{aligned}
$$
Letting $\delta\to 0^+$ and $a\to\infty$ on the right hand side leads to the desired lower bound
\begin{align}\label{T-8}
&\liminf_{t\to\infty}{1\over t}
\log\E_0\exp\bigg\{{b\over t^{1-\alpha_0}}\int_0^t\!\int_0^t
\Big\vert\theta(s-r)+i\epsilon\big(\beta(s)-\beta(r)\big)\Big\vert^{-\alpha_0}\gamma\big(B(s)-B(r)\big)drds\bigg\}\nonumber\\
&\ge b^{2\over 2-\alpha}\theta^{-{2\alpha_0\over 2-\alpha}}{\cal E}_0.
\end{align}
\qed

To prove the upper bound of (\ref{T-7}), it requires several steps. First, we prove

\begin{lemma}\label{T-1-1}
  Let $B_1(t)$ and $B_2(t)$ be two independent $d$-dimensional Brownian
motions. Under $\alpha <2$,  there is a constant $C>0$ such that
\begin{align}\label{T-9}
\limsup_{t\to\infty}{1\over t}\log
\E_0\exp\bigg\{bt^{2-\alpha\over 2}\int_0^{1}\!\!\int_0^{1}(s+r)^{-\alpha_0}
\gamma\big(B_1(s)+B_2(s)\big)dsdr\bigg\}\le Cb^{2\over 2-\alpha}
\end{align}
for any $b>0$.
\end{lemma}

\proof Perform
the decomposition
$$
\int_0^t\!\!\int_0^t(s+r)^{-\alpha_0}
\gamma\big(B_1(s)+B_2(s)\big)dsdr=\int_0^{ct}\!\!\int_0^{ct}+\int\!\!\int_{[0,t]^2\setminus[0, ct]^2}
$$
where the constant $0<c<1$ will be specified later. By Cauchy-Schwartz inequality,
$$
\begin{aligned}
&\E_0\exp\bigg\{{b\over t^{1-\alpha_0}}\int_0^t\!\int_0^t
(s+r)^{-\alpha_0}\gamma\big(B_1(s)+B_2(r)\big)drds\bigg\}\\
&\le\Bigg(\E_0\exp\bigg\{{2b\over t^{1-\alpha_0}}\int_0^{ct}\!\!\int_0^{ct}
(s+r)^{-\alpha_0}\gamma\big(B_1(s)+B_2(r)\big)drds\bigg\}\Bigg)^{1/2}\\
&\times\Bigg(\E_0\exp\bigg\{{2b\over t^{1-\alpha_0}}\int\!\!\int_{[0,t]^2\setminus[0, ct]^2}
(s+r)^{-\alpha_0}\gamma\big(B_1(s)+B_2(r)\big)drds\bigg\}\Bigg)^{1/2}.
\end{aligned}
$$
Notice that
$$
\int_0^{ct}\!\int_0^{ct}
(s+r)^{-\alpha_0}\gamma\big(B_1(s)+B_2(r)\big)drds
\buildrel d\over =c^{4-\alpha-2\alpha_0\over 2}
\int_0^t\!\int_0^t
(s+r)^{-\alpha_0}\gamma\big(B_1(s)+B_2(r)\big)drds.
$$
Choose $c$ such that
$$
2c^{4-\alpha-2\alpha_0\over 2}=1.
$$
So we have 
$$
\begin{aligned}
&\E_0\exp\bigg\{{b\over t^{1-\alpha_0}}\int_0^t\!\int_0^t
(s+r)^{-\alpha_0}\gamma\big(B_1(s)+B_2(r)\big)drds\bigg\}\\
&\le\Bigg(\E_0\exp\bigg\{{b\over t^{1-\alpha_0}}\int_0^{t}\!\!\int_0^{t}
(s+r)^{-\alpha_0}\gamma\big(B_1(s)+B_2(r)\big)drds\bigg\}\Bigg)^{1/2}\\
&\times\Bigg(\E_0\exp\bigg\{{2b\over t^{1-\alpha_0}}\int\!\!\int_{[0,t]^2\setminus[0, ct]^2}
(s+r)^{-\alpha_0}\gamma\big(B_1(s)+B_2(r)\big)drds\bigg\}\Bigg)^{1/2}.
\end{aligned}
$$
which leads to
\begin{align}\label{T-10}
&\E_0\exp\bigg\{{b\over t^{1-\alpha_0}}\int_0^t\!\int_0^t
(s+r)^{-\alpha_0}\gamma\big(B_1(s)+B_2(r)\big)drds\bigg\}\\
&\le \E_0\exp\bigg\{{2b\over t^{1-\alpha_0}}\int\!\!\int_{[0,t]^2\setminus[0, ct]^2}
(s+r)^{-\alpha_0}\gamma\big(B_1(s)+B_2(r)\big)drds\bigg\}.\nonumber
\end{align}
By the fact that $s+r\ge ct$ on $[0,t]^2\setminus[0, ct]^2$,
$$
\begin{aligned}
&\E_0\exp\bigg\{{2b\over t^{1-\alpha_0}}\int\!\!\int_{[0,t]^2\setminus[0, ct]^2}
(s+r)^{-\alpha_0}\gamma\big(B_1(s)+B_2(r)\big)drds\\
&\le \E_0\exp\bigg\{{2b\over t^{1-\alpha_0}}(ct)^{-\alpha_0}\int_0^t\!\!\int_0^t
\gamma\big(B_1(s)+B_2(r)\big)drds\bigg\}\\
&=\E_0\exp\bigg\{{2bc^{-\alpha_0}\over t}\int_0^t\!\!\int_0^t
\gamma\big(B_1(s)+B_2(r)\big)drds\bigg\}.
\end{aligned}
$$
By (\ref{T-2}), therefore,
$$
\limsup_{t\to\infty}{1\over t}\log \E_0\exp\bigg\{{2b\over t^{1-\alpha_0}}\int\!\!\int_{[0,t]^2\setminus[0, ct]^2}
(s+r)^{-\alpha_0}\gamma\big(B_1(s)+B_2(r)\big)drds\bigg\}
\le (2b)^{2\over 2-\alpha}{\cal H}.
$$
In view of (\ref{T-10}), 
$$
\limsup_{t\to\infty}{1\over t}\log \E_0\exp\bigg\{{b\over t^{1-\alpha_0}}\int_0^t\!\int_0^t
(s+r)^{-\alpha_0}\gamma\big(B_1(s)+B_2(r)\big)drds\bigg\}\le 2^{\alpha\over 2-\alpha}b^{2\over 2-\alpha}{\cal H}.
$$
Finally, the proof follows from the identity in law:
$$
\int_0^t\!\int_0^t
(s+r)^{-\alpha_0}\gamma\big(B_1(s)+B_2(r)\big)drds\buildrel d\over =
t^{4-\alpha-\alpha_0\over 2}\int_0^1\!\int_0^1
(s+r)^{-\alpha_0}\gamma\big(B_1(s)+B_2(r)\big)drds.
$$
\qed

Next, we establish a weaker version of the upper bound

\begin{lemma}\label{T-2-1} Under $\alpha_0+\alpha<2$, there is a constant $C>0$ such that
\begin{align}\label{T-11}
&\limsup_{t\to\infty}{1\over t}
\log\E_0\exp\bigg\{{b\over t^{1-\alpha_0}}\int_0^t\!\int_0^t
\Big\vert\theta(s-r)+i\eta\big(\beta(s)-\beta(r)\big)\Big\vert^{-\alpha_0}\gamma\big(B(s)-B(r)\big)drds\bigg\}\nonumber\\
&\le Cb^{2\over 2-\alpha}\theta^{-{2\alpha_0\over 2-\alpha}}
\end{align}
for any  $b>0$, $\theta>0$ and $\eta>0$.
\end{lemma}

\proof By the relation
$$
\begin{aligned}
&\int_0^t\!\int_0^t
\Big\vert\theta(s-r)+i\eta\big(\beta(s)-\beta(r)\big)\Big\vert^{-\alpha_0}\gamma\big(B(s)-B(r)\big)drds\\
&=2\int\!\int_{[0,t]_<^2}
\Big\vert\theta(s-r)+i\eta\big(\beta(s)-\beta(r)\big)\Big\vert^{-\alpha_0}\gamma\big(B(s)-B(r)\big)drds,
\end{aligned}
$$
(\ref{T-11}) is equivalent to
\begin{align}\label{T-12}
&\limsup_{t\to\infty}{1\over t}
\log\E_0\exp\bigg\{{b\over t^{1-\alpha_0}}\int\!\int_{[0,t]_<^2}
\Big\vert\theta(s-r)+i\eta\big(\beta(s)-\beta(r)\big)\Big\vert^{-\alpha_0}\gamma\big(B(s)-B(r)\big)drds\bigg\}\nonumber\\
&\le Cb^{2\over 2-\alpha}\theta^{-{2\alpha_0\over 2-\alpha}}.
\end{align}

Define the random measures $Q(\cdot)$, $Q_1(\cdot)$ and $Q_2(\cdot)$ on $(\R_+)^2_<$ as follows
$$
Q(A)=\int\!\int_A\Big\vert\theta(s-r)+i\eta\big(\beta(s)-\beta(r)\big)\Big\vert^{-\alpha_0}\Gamma\big(B(s)-B(r)\big)dsdr
\hskip.2in A\subset(\R_+)^2_<;
$$
$$
Q_1(A)=\int\!\int_A\big\vert \beta(s)-\beta(r)\big\vert^{-\alpha_0}\gamma\big(B(s)-B(r)\big)dsdr
\hskip.2in A\subset(\R_+)^2_<;
$$
$$
Q_2(A)=\int\!\int_A\vert s-r\vert^{-\alpha_0}\gamma\big(B(s)-B(r)\big)dsdr
\hskip.2in A\subset(\R_+)^2_<.
$$
By the facts that
$$
\Big\vert\theta(s-r)+i\big(\beta(s)-\beta(r)\big)\Big\vert^{-\alpha_0}
\le \theta^{-\alpha_0}\vert s-r\vert^{-\alpha_0}
$$
and that
$$
\Big\vert\theta(s-r)+i\big(\beta(s)-\beta(r)\big)\Big\vert^{-\alpha_0}
\le\eta^{-\alpha_0}\big\vert \beta(s)-\beta(r)\big\vert^{-\alpha_0}
$$
we have
$$
Q(A)\le \eta^{-\alpha_0}Q_1(A)\hskip.1in\hbox{and}\hskip.1in
Q(A)\le \theta^{-\alpha_0}Q_2(A).
$$

Consider the triangular decomposition: For a integral $N\ge 1$,
$$
[0,t]_<^2=\bigg(\bigcup_{l=0}^{2^{N+1} -1}
\Big[{l\over 2^{N+1}}t, {l+1\over 2^{N+1}}t\Big]_<^2\bigg)\cup
\bigg(\bigcup_{k=0}^{N-1}\bigcup_{l=0}^{2^k-1}A_l^k\bigg)
$$
where
$$
A_l^k=\Big[{2l\over 2^{k+1}}t, {2l+1\over 2^{k+1}}t\Big]\times\Big[{2l+1\over 2^{k+1}}t, {2l+2\over 2^{k+1}}t\Big]
\hskip.2in l=0, 1,\cdots, 2^k-1;\hskip.1in k=0, 1,\cdots, N-1.
$$
See the diagram Figure 1 for the case $N=2$. In our proof, $N$ increases along with $t$ in a way
that will be specified later.
\begin{figure}\label{Fig:Chen1}
  \includegraphics[scale=0.99]{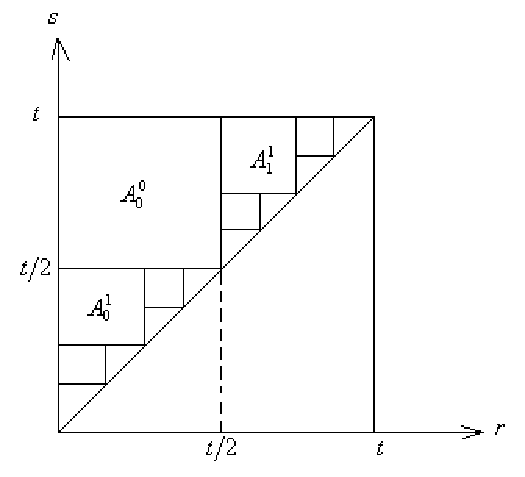}\\
  \caption{triangular approximation\index{triangular approximation}} 
\end{figure}
We have
$$
\begin{aligned}
&Q([0,t]_<^2)=Q\bigg(\bigcup_{l=0}^{2^{N+1}-1}
\Big[{l\over 2^{N+1}}, {l+1\over 2^{N+1}}\Big]_<^2\bigg)+
Q\bigg(\bigcup_{k=0}^{N-1}\bigcup_{l=0}^{2^k-1}A_l^k\bigg)\\
&\le \eta^{-\alpha_0}Q_1\bigg(\bigcup_{l=0}^{2^{N+1}-1}
\Big[{l\over 2^{N+1}}t, {l+1\over 2^{N+1}}t\Big]_<^2\bigg)+
\theta^{-\alpha_0}Q_2\bigg(\bigcup_{k=0}^{N-1}\bigcup_{l=0}^{2^k-1}A_l^k\bigg)\nonumber\\
&=\eta^{-\alpha_0}\sum_{l=0}^{2^{N+1}-1}Q_1\Big(\Big[{l\over 2^{N+1}}t, {l+1\over 2^{N+1}}t\Big]_<^2\Big)
+\theta^{-\alpha_0}\sum_{k=0}^{N-1}Q_2\Big(\bigcup_{l=0}^{2^k-1}A_l^k\Big).\nonumber
\end{aligned}
$$
By Cauchy-Schwartz inequality
\begin{align}\label{T-13}
&\E_0\exp\Big\{{b\over t^{1-\alpha_0}}Q\big([0, t]_<^2\big)\Big\}\\
&\le\Bigg(\E_0\exp\bigg\{2\eta^{-\alpha_0}{b\over t^{1-\alpha_0}}
\sum_{l=0}^{2^{N+1}-1}Q_1\Big(\Big[{l\over 2^{N+1}}t, {l+1\over 2^{N+1}}t\Big]_<^2\Big)\bigg\}\Bigg)^{1/2}\nonumber\\
&\times\Bigg(\E_0\exp\bigg\{2\theta^{-\alpha_0}{b\over t^{1-\alpha_0}}
\sum_{k=0}^{N-1}Q_2\Big(\bigcup_{l=0}^{2^k-1}A_l^k\Big)\bigg\}\Bigg)^{1/2}.\nonumber
\end{align}
Notice that the random variables
$$
Q_1\Big(\Big[{l\over 2^{N+1}}t, {l+1\over 2^{N+1}}t\Big]_<^2\Big)\hskip.2in l=0, 1,\cdots, 2^{N+1}-1
$$
form an i.i.d. sequence. Therefore,
$$
\begin{aligned}
&\E_0\exp\bigg\{2\eta^{-\alpha_0}{b\over t^{1-\alpha_0}}
\sum_{l=0}^{2^{N+1}-1}Q_1\Big(\Big[{l\over 2^{N+1}}t, {l+1\over 2^{N+1}}t\Big]_<^2\Big)\bigg\}\\
&=\Bigg(\E_0\exp\bigg\{2\eta^{-\alpha_0}{b\over t^{1-\alpha_0}}
Q_1\Big(\Big[0, {t\over 2^{N+1}}\Big]_<^2\Big)\bigg\}\Bigg)^{2^{N+1}}
\end{aligned}
$$
and
$$
\begin{aligned}
&Q_1\Big(\Big[0, {t\over 2^{N+1}}\Big]_<^2\Big)
={1\over 2}\int_0^{t\over 2^{N+1}}\int_0^{t\over 2^{N+1}}
\vert\beta(s)-\beta(r)\vert^{-\alpha_0}\gamma\big(B(s)-B(r)\big)dsdr\\
&\buildrel d\over ={1\over 2}\Big({t\over 2^{N+1}}\Big)^{4-\alpha-\alpha_0\over 2}\int_0^1\!\!\int_0^1
\vert\beta(s)-\beta(r)\vert^{-\alpha_0}\gamma\big(B(s)-B(r)\big)dsdr.
\end{aligned}
$$
In summary,
$$
\begin{aligned}
&\E_0\exp\bigg\{2\eta^{-\alpha_0}{b\over t^{1-\alpha_0}}
\sum_{l=0}^{2^{N+1}-1}Q_1\Big(\Big[{l\over 2^{N+1}}t, {l+1\over 2^{N+1}}t\Big]_<^2\Big)\bigg\}\\
&=\Bigg(\E_0\exp\bigg\{\eta^{-\alpha_0}b\Big({t\over 2^{N+1}}
\Big)^{2-\alpha-\alpha_0\over 2}{t^{\alpha_0}\over 2^{N+1}}
\int_0^1\!\!\int_0^1
\vert\beta(s)-\beta(r)\vert^{-\alpha_0}\gamma\big(B(s)-B(r)\big)dsdr\bigg\}\Bigg)^{2^{N+1}}.
\end{aligned}
$$
Given $\epsilon>0$, we now post our constraint
$$
\eta^{-\alpha_0}{t^{\alpha_0}\over 2^{N+1}}\le \epsilon
$$
for which we let
$$
N=\Big[{\log \epsilon^{-1}(\eta^{-1}t)^{\alpha_0}\over\log 2}\Big].
$$
So we have
$$
\begin{aligned}
&\E_0\exp\bigg\{2\eta^{-\alpha_0}{b\over t^{1-\alpha_0}}
\sum_{l=0}^{2^{N+1}-1}Q_1\Big(\Big[{l\over 2^{N+1}}t, {l+1\over 2^{N+1}}t\Big]_<^2\Big)\bigg\}\\
&\le \Bigg(\E_0\exp\bigg\{\epsilon b\Big({t\over 2^{N+1}}\Big)^{2-\alpha-\alpha_0\over 2}
\int_0^1\!\!\int_0^1
\vert\beta(s)-\beta(r)\vert^{-\alpha_0}\gamma\big(B(s)-B(r)\big)dsdr\bigg\}\Bigg)^{2^{N+1}}.
\end{aligned}
$$
In view of (\ref{T-2}) and of the relation
$$
\begin{aligned}
&\int_0^t\!\!\int_0^t
\vert\beta(s)-\beta(r)\vert^{-\alpha_0}\gamma\big(B(s)-B(r)\big)dsdr\\
&\buildrel d\over =t^{4-\alpha-\alpha_0\over 2}\int_0^1\!\!\int_0^1
\vert\beta(s)-\beta(r)\vert^{-\alpha_0}\gamma\big(B(s)-B(r)\big)dsdr,
\end{aligned}
$$
on the other hand, we have
$$
\lim_{t\to\infty}{1\over t}\log\E_0\exp\bigg\{b t^{2-\alpha-\alpha_0\over 2}\int_0^1\!\!\int_0^1
\vert\beta(s)-\beta(r)\vert^{-\alpha_0}\gamma\big(B(s)-B(r)\big)dsdr\bigg\}
=b^{2\over 2-\alpha-\alpha_0}\widetilde{H}
$$
Replacing $b$ by $\epsilon b$ and $t$ by $2^{-(N+1)}t$
\begin{align}\label{T-14}
&\limsup_{t\to\infty}{1\over t}\log\E_0\exp\bigg\{2\eta^{-\alpha_0}{b\over t^{1-\alpha_0}}
\sum_{l=0}^{2^{N+1}-1}Q_1\Big(\Big[{l\over 2^{N+1}}t, {l+1\over 2^{N+1}}t\Big]_<^2\Big)\bigg\}\\
&\le\lim_{t\to\infty}{2^{N+1}\over t}\log
\E_0\exp\bigg\{\epsilon b\Big({t\over 2^{N+1}}\Big)^{2-\alpha-\alpha_0\over 2}
\int_0^1\!\!\int_0^1
\vert\beta(s)-\beta(r)\vert^{-\alpha_0}\gamma\big(B(s)-B(r)\big)dsdr\bigg\}\nonumber\\
&= (\epsilon b)^{2\over 2-\alpha-\alpha_0}\widetilde{H}\nonumber
\end{align}
where we used the fact that $2^{N+1}t^{-1}\to 0$ (as $t\to\infty$) in the last
equality.
\medskip

Let 
$$
a_k=\prod_{j=2}^k\Big(1- 2^{-(1-\alpha_0)j}\Big)\hskip.2in k=2,3,\cdots, N\hskip.1in\hbox{and}\hskip.1in
C_0=\prod_{j=2}^\infty\Big(1- 2^{-(1-\alpha_0)j}\Big)^{-1}
$$
Since $C_0a_N\ge 1$,
$$
\E_0\exp\bigg\{2\theta^{-\alpha_0}{b\over t^{1-\alpha_0}}
\sum_{k=0}^{N-1}Q_2\Big(\bigcup_{l=0}^{2^k-1}A_l^k\Big)
\bigg\}
\le \E_0\exp\bigg\{2\theta^{-\alpha_0}{b\over t^{1-\alpha_0}}
C_0a_N\sum_{k=0}^{N-1}Q_2\Big(\bigcup_{l=0}^{2^k-1}A_l^k\Big)
\bigg\}.
$$
By H\"older's inequality
$$
\begin{aligned}
&\E_0\exp\bigg\{2\theta^{-\alpha_0}{b\over t^{1-\alpha_0}}
C_0 a_N\sum_{k=0}^{N-1}Q_2\Big(\bigcup_{l=0}^{2^k-1}A_l^k\Big)
\bigg\}\\
&\le \Bigg(\E_0\exp\bigg\{2\theta^{-\alpha_0}{b\over t^{1-\alpha_0}}\theta^{-\alpha_0}C_0a_{N-1}
\sum_{k=0}^{N-2}Q_2\Big(\bigcup_{l=0}^{2^k-1}A_l^k\Big)
\bigg\}\Bigg)^{1-2^{-(1-\alpha_0)N}}\\
&\times\Bigg(\E_0\exp\bigg\{2^{(1-\alpha_0)N}2\theta^{-\alpha_0}
C_0a_N
{b\over t^{1-\alpha_0}}Q_2\Big(\bigcup_{l=0}^{2^N-1}A_l^{N-1}\Big)\bigg\}\Bigg)^{2^{-(1-\alpha_0)N}}.
\end{aligned}
$$
To continue, we remove the power "$1-2^{-(1-\alpha_0)N}$" from the first factor. Since $a_N\le 1$, we remove
$a_N$ from the second factor. So we have
\begin{align}\label{T-15}
&\E_0\exp\bigg\{2\theta^{-\alpha_0}{b\over t^{1-\alpha_0}}
C_0 a_N\sum_{k=0}^{N-1}Q_2\Big(\bigcup_{l=0}^{2^k-1}A_l^k\Big)
\bigg\}\\
&\le \E_0\exp\bigg\{2\theta^{-\alpha_0}{b\over t^{1-\alpha_0}}\theta^{-\alpha_0}C_0a_{N-1}
\sum_{k=0}^{N-2}Q_2\Big(\bigcup_{l=0}^{2^k-1}A_l^k\Big)
\bigg\}\nonumber\\
&\times\Bigg(\E_0\exp\bigg\{2^{(1-\alpha_0)N}2\theta^{-\alpha_0}
C_0{b\over t^{1-\alpha_0}}Q_2\Big(\bigcup_{l=0}^{2^N-1}A_l^{N-1}\Big)\bigg\}\Bigg)^{2^{-(1-\alpha_0)N}}.\nonumber
\end{align}

Write
$$
Q_2\Big(\bigcup_{l=0}^{2^N-1}A_l^{N-1}\Big)=\sum_{l=0}^{2^N-1}Q_2(A_l^{N-1})
$$
and notice that the right hand side is a sum of i.i.d. Thus,
$$
\begin{aligned}
&\E_0\exp\bigg\{2^{(1-\alpha_0)N}2C_0\theta^{-\alpha_0}{b\over t^{1-\alpha_0}}
Q_2\Big(\bigcup_{l=0}^{2^N-1}A_l^{N-1}\Big)\bigg\}\\
&=\bigg(\E_0\exp\Big\{2^{(1-\alpha_0)N}2C_0\theta^{-\alpha_0}{b\over t^{1-\alpha_0}}Q_2(A_0^{N-1})\Big\}\bigg)^{2^N}.
\end{aligned}
$$
Further,
$$
\begin{aligned}
&Q_2(A_0^{N-1})=\int_0^{t\over 2^N}dr\int_{t\over 2^N}^{2t\over 2^N}ds\vert s-r\vert^{-\alpha_0}\gamma\big(B(s)-B(r)\big)\\
&\buildrel d\over =\Big({t\over 2^N}\Big)^{4-\alpha-2\alpha_0\over 2^N}
\int_0^{1}dr\int_1^{2}ds\vert s-r\vert^{-\alpha_0}\gamma\big(B(s)-B(r)\big)\\
&\buildrel d\over =\Big({t\over 2^N}\Big)^{4-\alpha-2\alpha_0\over 2^N}
\int_0^{1}\!\!\int_0^1(s+r)^{-\alpha_0}\gamma\big(B_1(s)+B_2(r)\big)dsdr
\end{aligned}
$$
where $B_1(\cdot)$ and $B_2(\cdot)$ are two independent $d$-dimensional Brownian motions,
and where the last identity in law follows from the variable substitution $r\mapsto 1-r$ and $s\mapsto 1+s$ so
$$
\begin{aligned}
&\int_0^{1}dr\int_1^{2}ds\vert s-r\vert^{-\alpha_0}\gamma\big(B(s)-B(r)\big)\\
&=\int_0^{1}dr\int_0^1ds\vert (1+s)-(1-r)\vert^{-\alpha_0}\gamma\big(B(1+s)-B(1-r)\big)\\
&=\int_0^1\!\int_0^1(s+r)^{-\alpha_0}\gamma\Big(\big(B(1+s)-B(1)\big)+\big(B(1)-B(1-r)\big)\Big)dsdr
\end{aligned}
$$
and $B_1(s)\equiv B(1+s)-B(1)$ ($0\le s\le 1$) and
$B_2(r)\equiv B(1)-B(1-r)$ ($0\le r\le 1$) are 
 two independent $d$-dimensional Brownian motions. Thus,
 $$
 \begin{aligned}
 &\E_0\exp\Big\{2^{(1-\alpha_0)N}2C_0\theta^{-\alpha_0}{b\over t^{1-\alpha_0}}Q_2(A_0^{N-1})\Big\}\\
 &=\E_0\exp\bigg\{2Cb\theta^{-\alpha_0}
\Big({t\over 2^N}\Big)^{2-\alpha\over 2}\int_0^1\!\int_0^1(s+r)^{-\alpha_0}
\gamma\big(B_1(s)+B_2(r)\big)dsdr\bigg\}.
\end{aligned}
$$
 Therefore,
$$
\begin{aligned}
&\E_0\exp\bigg\{2^{(1-\alpha_0)N}2\theta^{-\alpha_0}
C_0{b\over t^{1-\alpha_0}}Q_2\Big(\bigcup_{l=0}^{2^N-1}A_l^{N-1}\Big)\bigg\}\\
&=\Bigg(\E_0\exp\bigg\{2C_0b\theta^{-\alpha_0}
\Big({t\over 2^N}\Big)^{2-\alpha\over 2}\int_0^1\!\int_0^1(s+r)^{-\alpha_0}
\gamma\big(B_1(s)+B_2(r)\big)dsdr\bigg\}\Bigg)^{2^N}.
\end{aligned}
$$
Bringing it back to (\ref{T-15}),
$$
\begin{aligned}
&\E_0\exp\bigg\{2\theta^{-\alpha_0}{b\over t^{1-\alpha_0}}
C_0 a_N\sum_{k=0}^{N-1}Q_2\Big(\bigcup_{l=0}^{2^k-1}A_l^k\Big)
\bigg\}\\
&\le \E_0\exp\bigg\{2\theta^{-\alpha_0}{b\over t^{1-\alpha_0}}\theta^{-\alpha_0}C_0a_{N-1}
\sum_{k=0}^{N-2}Q_2\Big(\bigcup_{l=0}^{2^k-1}A_l^k\Big)
\bigg\}\\
&\times\Bigg(\E_0\exp\bigg\{2C_0b\theta^{-\alpha_0}
\Big({t\over 2^N}\Big)^{2-\alpha\over 2}\int_0^1\!\int_0^1(s+r)^{-\alpha_0}
\gamma\big(B_1(s)+B_2(r)\big)dsdr\bigg\}\Bigg)^{2^{N\alpha_0}}.
\end{aligned}
$$
Repeating the same game
$$
\begin{aligned}
&\E_0\exp\bigg\{{b\over t^{1-\alpha_0}}\theta^{-\alpha_0}\sum_{k=0}^{N-1}Q_2\Big(\bigcup_{l=0}^{2^k-1}A_l^k\Big)
  \bigg\}\le\E_0\exp\bigg\{2\theta^{-\alpha_0}{b\over t^{1-\alpha_0}}
C_0 a_N\sum_{k=0}^{N-1}Q_2\Big(\bigcup_{l=0}^{2^k-1}A_l^k\Big)
\bigg\}\\
&\le \prod_{k=1}^{N}\Bigg(\E_0\exp\bigg\{2C_0b\theta^{-\alpha_0}
\Big({t\over 2^k}\Big)^{2-\alpha\over 2}\int_0^1\!\int_0^1(s+r)^{-\alpha_0}
\gamma\big(B_1(s)+B_2(r)\big)dsdr\bigg\}\Bigg)^{2^{k\alpha_0}}.
\end{aligned}
$$
By Lemma \ref{T-1-1} with $b$ being replaced by $2C_0b\theta^{-\alpha_0}$ and $t$ being replaced by
$2^{-k}t$,
$$
\begin{aligned}
&\limsup_{t\to\infty}{1\over t}\log\E_0\exp\bigg\{{b\over t^{1-\alpha_0}}\theta^{-\alpha_0}\sum_{k=0}^{N-1}Q_2\Big(\bigcup_{l=0}^{2^k-1}A_l^k\Big)
  \bigg\}\le C (2C_0b)^{2\over 2-\alpha}\theta^{-{2\alpha_0\over 2-\alpha}}
  \sum_{k=1}^\infty 2^{-(1-\alpha_0)k}.
\end{aligned}
$$
Together with (\ref{T-12}) and (\ref{T-13}),
$$
\begin{aligned}
&\limsup_{t\to\infty}{1\over t}\log\E_0\exp\Big\{{b\over t^{1-\alpha_0}}Q\big([0, t]_<^2\big)\Big\}\cr
&\le {1\over 2}C(\epsilon b)^{2\over 2-\alpha-\alpha_0}+
{1\over 2}C (2C_0b)^{2\over 2-\alpha}\theta^{-{2\alpha_0\over 2-\alpha}}\sum_{k=1}^\infty 2^{-(1-\alpha_0)k}.
\end{aligned}
$$
Letting $\epsilon\to 0^+$ on the right hand side finally completes the proof of (\ref{T-12}). \qed

\medskip

{\it Proof of the upper bound}.
To tighten (\ref{T-11}) into the demanded upper bound, we first write it as
\begin{align}\label{T-16}
&\limsup_{t\to\infty}{1\over t}
\log\E_0\exp\bigg\{{b\over t}\int_0^t\!\int_0^t
\Big\vert{\theta(s-r)+i\eta\big(\beta(s)-\beta(r)\big)\over t}\Big\vert^{-\alpha_0}\gamma\big(B(s)-B(r)\big)drds\bigg\}\nonumber\\
&\le Cb^{2\over 2-\alpha}\theta^{-{2\alpha_0\over 2-\alpha}}
\end{align}
and prove
\begin{align}\label{T-17}
\lim_{\delta\to 0^+}
\limsup_{t\to\infty}{1\over t}\log\E_0&\exp\bigg\{{b\over t}\int_0^t\!\int_0^t
\tilde{\gamma}^0_{\delta}\Big({\theta(s-r)+i\eta\big(\beta(s)-\beta(r)\big)\over t}\Big)\\
&\times \gamma\big(B(s)-B(r)\big)drds\bigg\}=0\nonumber
\end{align}
where
$$
\tilde{\gamma}^0_\delta(u)=\Gamma(\alpha_0)^{-1}\int_{\delta^{-1}}^\infty e^{-\lambda\vert u\vert}{d\lambda\over\lambda^{1-\alpha_0}}.
$$
Indeed, given $\epsilon>0$, break $[0, t]^2$ into two parts: On the first part
$$
\Big\vert{\theta(s-r)+i\eta\big(\beta(s)-\beta(r)\big)\over t}\Big\vert\ge\epsilon
$$
where
$$
\tilde{\gamma}^0_{\delta}\Big({\theta(s-r)+i\eta\big(\beta(s)-\beta(r)\big)\over t}\Big)
\le\tilde{\gamma}^0_{\delta}(\epsilon).
$$
On the second part
$$
\Big\vert{\theta(s-r)+i\eta\big(\beta(s)-\beta(r)\big)\over t}\Big\vert<\epsilon
$$
where
$$
\begin{aligned}
&\tilde{\gamma}^0_{\delta}\Big({\theta(s-r)+i\eta\big(\beta(s)-\beta(r)\big)\over t}\Big)
\le\Big\vert{\theta(s-r)+i\eta\big(\beta(s)-\beta(r)\big)\over t}\Big\vert^{-\alpha_0}\\
&\le\epsilon^{\tilde{\alpha_0}-\alpha_0}\Big\vert{\theta(s-r)
+i\eta\big(\beta(s)-\beta(r)\big)\over t}\Big\vert^{-\tilde{\alpha}_0}
\end{aligned}
$$
and where $\alpha_0<\tilde{\alpha}_0<1$ satisfies $\tilde{\alpha}_0+\alpha<2$. Thus,
$$
\begin{aligned}
&\int_0^t\!\int_0^t
\tilde{\gamma}^0_{\delta}\Big({\theta(s-r)+i\eta\big(\beta(s)-\beta(r)\big)\over t}\Big)\gamma\big(B(s)-B(r)\big)drds\\
&\le\tilde{\gamma}^0_{\delta}(\epsilon)\int_0^t\!\int_0^t\gamma\big(B(s)-B(r)\big)drds\\
&+\epsilon^{\tilde{\alpha}_0-\alpha_0}\int_0^t\!\int_0^t
\Big\vert{\theta(s-r)+i\eta\big(\beta(s)-\beta(r)\big)\over t}\Big\vert^{-\tilde{\alpha}_0}\gamma\big(B(s)-B(r)\big)drds.
\end{aligned}
$$
By Cauchy-Schwartz's inequality
$$
\begin{aligned}
&\E_0\exp\bigg\{{b\over t}\int_0^t\!\int_0^t
\tilde{\gamma}^0_{\delta}\Big({\theta(s-r)+i\eta\big(\beta(s)-\beta(r)\big)\over t}\Big)
\gamma\big(B(s)-B(r)\big)drds\bigg\}\\
&\le\Bigg(\E_0\exp\bigg\{{2b\over t}\tilde{\gamma}_\delta^0(\epsilon)\int_0^t\!\int_0^t
\gamma\big(B(s)-B(r)\big)drds\bigg\}\Bigg)^{1/2}\\
&\times\Bigg(\E_0\exp\bigg\{{2b\over t}\epsilon^{\tilde{\alpha}_0-\alpha_0}
\int_0^t\!\int_0^t\Big\vert{\theta(s-r)+i\eta\big(\beta(s)-\beta(r)\big)\over t}
\Big\vert^{-\tilde{\alpha}_0}\gamma\big(B(s)-B(r)\big)drds\bigg\}\Bigg)^{1/2}.
\end{aligned}
$$
By (\ref{T-2}) and (\ref{T-16}) (with $\alpha_0$ being replaced by $\tilde{\alpha}_0$),
$$
\begin{aligned}
&\limsup_{t\to\infty}{1\over t}\log\E_0\exp\bigg\{{b\over t}\int_0^t\!\int_0^t
\tilde{\gamma}^0_{\delta}\Big({\theta(s-r)+i\eta\big(\beta(s)-\beta(r)\big)\over t}\Big)
\gamma\big(B(s)-B(r)\big)drds\bigg\}\\
&\le {1\over 2}\Big\{C(2b)^{2\over 2-\alpha}\big(\tilde{\gamma}^0_\delta(\epsilon)\big)^{2\over 2-\alpha}
+C(2b)^{2\over 2-\alpha}\epsilon^{2(\tilde{\alpha_0}-\alpha_0)\over 2-\alpha}
\Big\}.
\end{aligned}
$$
The claim (\ref{T-17}) follows from that $\tilde{\gamma}^0_\delta(\epsilon)\to 0$ ($\delta\to 0^+$) for any $\epsilon>0$.

We finally come to the proof for the upper bound of (\ref{T-7}).
By the relation $\vert\cdot\vert^{-\alpha_0}=\gamma_\delta^0(\cdot)+\tilde{\gamma}_\delta^0(\cdot)$ and by H\"older's
inequality
$$
\begin{aligned}
&\E_0\exp\bigg\{{b\over t}\int_0^t\!\int_0^t\Big\vert{\theta(s-r)+i\eta\big(\beta(s)-\beta(r)\big)\over t}\Big\vert^{-\alpha_0}\gamma\big(B(s)-B(r)\big)drds\bigg\}\\
&\le\Bigg(\E_0\exp\bigg\{{bp\over t}\int_0^t\!\int_0^t
\gamma_\delta^0\Big({\theta(s-r)+i\eta\big(\beta(s)-\beta(r)\big)\over t}\Big)\gamma\big(B(s)-B(r)\big)drds\bigg\}\Bigg)^{1/p}\\
&\times\Bigg(\E_0\exp\bigg\{{bq\over t}\int_0^t\!\int_0^t\tilde{\gamma}_\delta^0\Big({\theta(s-r)+i\eta\big(\beta(s)-\beta(r)\big)\over t}\Big)\gamma\big(B(s)-B(r)\big)drds\bigg\}\Bigg)^{1/q}.
\end{aligned}
$$
Using the bound
$$
\gamma_\delta^0\Big({\theta(s-r)+i\eta\big(\beta(s)-\beta(r)\big)\over t}\Big)\le \gamma_\delta^0\Big({\theta(s-r)\over t}\Big)
=\theta^{-\alpha_0}\gamma_{\theta^{-1}\delta}\Big({s-r\over t}\Big)
$$
and (\ref{T-6}) (with $b$ being replaced by $bp\theta^{-\alpha_0}$ and $\delta$ by $\theta^{-1}\delta$),
$$
\begin{aligned}
&\limsup_{t\to\infty}{1\over t}\log\E_0\exp\bigg\{{bp\over t}\int_0^t\!\int_0^t\gamma_\delta^0\Big({\theta(s-r)+i\eta\big(\beta(s)-\beta(r)\big)\over t}\Big)\gamma\big(B(s)-B(r)\big)drds\bigg\}\\
&\le (pb)^{2\over 2-\alpha}\theta^{-{2\alpha_0\over 2-\alpha}}{\cal E}_{\theta^{-1}\delta}
\le(pb)^{2\over 2-\alpha}\theta^{-{2\alpha_0\over 2-\alpha}}{\cal E}_0.
\end{aligned}
$$
Thus,
$$
\begin{aligned}
&\limsup_{t\to\infty}{1\over t}
\log\E_0\exp\bigg\{{b\over t}\int_0^t\!\int_0^t\Big\vert{\theta(s-r)+i\eta\big(\beta(s)-\beta(r)\big)\over t}\Big\vert^{-\alpha_0}\gamma\big(B(s)-B(r)\big)drds\bigg\}\\
&\le {1\over p}(pb)^{2\over 2-\alpha}\theta^{-{2\alpha_0\over 2-\alpha}}{\cal E}_0\\
&+{1\over q}\limsup_{t\to\infty}{1\over t}\log\E_0\exp\bigg\{{bq\over t}\int_0^t\!\int_0^t
\tilde{\gamma}^0_{\delta}\Big({\theta(s-r)+i\eta\big(\beta(s)-\beta(r)\big)\over t}\Big)
\gamma\big(B(s)-B(r)\big)drds\bigg\}.
\end{aligned}
$$
Letting $\delta\to 0^+$ on the right hand side, by (\ref{T-17}) (with $b$ being replaced by $qb$),
$$
\begin{aligned}
&\limsup_{t\to\infty}{1\over t}
\log\E_0\exp\bigg\{{b\over t}\int_0^t\!\int_0^t\Big\vert{\theta(s-r)+i\eta\big(\beta(s)-\beta(r)\big)\over t}\Big\vert^{-\alpha_0}\gamma\big(B(s)-B(r)\big)drds\bigg\}\\
&\le {1\over p}(pb)^{2\over 2-\alpha}\theta^{-{2\alpha_0\over 2-\alpha}}{\cal E}_0.
\end{aligned}
$$
Letting $p\to 1^+$ on the right hand side leads to the demanded upper bound:
\begin{align}\label{T-18}
&\limsup_{t\to\infty}{1\over t}
\log\E_0\exp\bigg\{{b\over t^{1-\alpha_0}}\int_0^t\!\int_0^t\Big\vert\theta(s-r)+i\eta\big(\beta(s)-\beta(r)\big)\Big\vert^{-\alpha_0}\gamma\big(B(s)-B(r)\big)drds\bigg\}\nonumber\\
&\le b^{2\over 2-\alpha}\theta^{-{2\alpha_0\over 2-\alpha}}{\cal E}_0.
\end{align}
\qed

We end this section with the following lemma.

\begin{lemma}\label{T-L} 
For any $\delta>0$
\begin{align}\label{T-19}
  \lim_{n\to\infty}{1\over n}\log (n!)^{-\alpha/2}
  \E_0\bigg[\int_0^1\!\!\int_0^1\gamma_\delta^0(s-r)
\gamma\big(B(s)-B(r)\big)dsdr\bigg]^n=\log\Big({2{\cal E}_\delta\over 2-\alpha}\Big)^{2-\alpha\over 2}.
\end{align}
For any $\theta>0$ and $\eta>0$,
\begin{align}\label{T-20}
&\lim_{n\to\infty}{1\over n}\log (n!)^{-\alpha/2}\E_0\bigg[\int_0^1\!\!\int_0^1
\Big\vert \theta(s-r)+i{\eta\over\sqrt{n}}\big(\beta(s)-\beta(r)\big)\Big\vert^{-\alpha_0}
\gamma\big(B(s)-B(r)\big)dsdr\bigg]^n\nonumber\\
&=\log\Big({2{\cal E}_0\over 2-\alpha}\Big)^{2-\alpha\over 2}-\alpha_0\log\theta.
\end{align}
\end{lemma}

\proof First notice
$$
\int_0^t\!\int_0^t\gamma_\delta^0\Big({s-r\over t}\Big)
\gamma\big(B(s)-B(r)\big)dsdr\buildrel d\over =t^{4-\alpha\over 2}
\int_0^1\!\int_0^1\gamma_\delta^0(s-r)
\gamma\big(B(s)-B(r)\big)dsdr.
$$
Therefore, (\ref{T-6}) can be written as
$$
\lim_{t\to\infty}{1\over t}\log\E_0\exp\bigg\{b t^{2-\alpha\over 2}\int_0^1\!\int_0^1\gamma_\delta^0(s-r)
\gamma\big(B(s)-B(r)\big)dsdr\bigg\}
=b^{2\over 2-\alpha}{\cal E}_\delta\hskip.2in b>0.
$$
By G\"artner-Ellis theorem on $\R^+$ (Theorem 1.2.4, p. 11, \cite{Chen-1}),
the process
$$
t^{-\alpha/2}\int_0^t\!\int_0^t\gamma_\delta^0(s-r)
\gamma\big(B(s)-B(r)\big)dsdr\hskip.2in t>0
$$
obeys the large deviation on $\R_+$ with deviation scale $t$ and the
rate function
$$
I(\lambda)=\sup_{b>0}\Big\{\lambda b-b^{2\over 2-\alpha}{\cal E}_\delta\Big\}
={\alpha\over 2}\Big({2-\alpha\over 2{\cal E}_\delta}\Big)^{2-\alpha\over \alpha}\lambda^{2\over\alpha}
\hskip.2in \lambda>0.
$$
By Varadhan's integral lemma (Theorem 4.3, p.137, \cite{DZ})
$$
\begin{aligned}
&\lim_{t\to\infty}{1\over t}\log\E_0\exp\bigg\{t\log\bigg( t^{-\alpha/2}
\int_0^1\!\!\int_0^1\tilde{\gamma}_a^0(s-r)
\gamma\big(B(s)-B(r)\big)dsdr\bigg)\bigg\}\\
&=\sup_{\lambda>0}\bigg\{\log \lambda-{\alpha\over 2}\Big({2-\alpha\over 2{\cal E}_\delta}\Big)^{2-\alpha\over \alpha}\lambda^{2\over\alpha}\bigg\}=-{\alpha\over 2}
+\log\Big({2{\cal E}_\delta\over 2-\alpha}\Big)^{2-\alpha\over 2}.
\end{aligned}
$$
Take $t=n$:
$$
\begin{aligned}
&\lim_{n\to\infty}{1\over n}\log n^{-{\alpha n\over 2}}\E_0\bigg[\int_0^1\!\!\int_0^1\tilde{\gamma}_a^0(s-r)
\gamma\big(B(s)-B(r)\big)dsdr\bigg]^n
=-{\alpha\over 2}+\log\Big({2{\cal E}_\delta\over 2-\alpha}\Big)^{2-\alpha\over 2}.
\end{aligned}
$$
Applying Stirling formula to the above gives (\ref{T-19}).

\medskip

Notice that
$$
\begin{aligned}
&\int_0^t\!\int_0^t
\Big\vert\theta(s-r)+i\eta\big(\beta(s)-\beta(r)\big)\Big\vert^{-\alpha_0}\gamma\big(B(s)-B(r)\big)drds\\
&\buildrel d\over =t^{4-\alpha-2\alpha_0\over 2}\int_0^1\!\int_0^1
\Big\vert\theta(s-r)+i{\eta\over \sqrt{t}}\big(\beta(s)-\beta(r)\big)\Big\vert^{-\alpha_0}\gamma\big(B(s)-B(r)\big)drds.
\end{aligned}
$$
Therefore, (\ref{T-7}) can be written as
$$
\begin{aligned}
&\lim_{t\to\infty}{1\over t}
\log\E_0\exp\bigg\{bt^{2-\alpha\over 2}\int_0^1\!\int_0^1
\Big\vert\theta(s-r)+i{\eta\over \sqrt{t}}\big(\beta(s)-\beta(r)\big)
\Big\vert^{-\alpha_0}\gamma\big(B(s)-B(r)\big)drds\bigg\}\\
&= b^{2\over 2-\alpha}\theta^{-{2\alpha_0\over 2-\alpha}}{\cal E}_0.
\end{aligned}
$$
The remaining of the proof for (\ref{T-20})
follows a completely parallel argument.
\qed

\section{Proof of Theorem \ref{th-2}}\label{P}

The central piece of the proof is to establish
\begin{align}\label{P-1}
&\lim_{n\to\infty}{1\over n}\log (n!)^{3-\alpha}\E S_{2n}\big(g_{2n}(\cdot, 1, 0)\big)
=\log \bigg({2(4-\alpha-2\alpha_0)^{4-\alpha-2\alpha_0\over 2}\over
(4-\alpha-\alpha_0)^{4-\alpha-\alpha_0}}\Big({{\cal M}\over 4-\alpha}\Big)^{4-\alpha\over 2}\bigg).
\end{align}
Indeed,  by (\ref{M-9}) $\E S_{2n-1}\big(g_{2n-1}(\cdot, t, 0)\big)=0$.
From (\ref{M-10-20}), (\ref{intro-8}) and (\ref{M-13}) one can derive that
\begin{align}\label{P-2}
\E S_{2n}\big(g_{2n}(\cdot, t, 0)\big)=t^{(4-\alpha-\alpha_0)n}\E S_{2n}\big(g_{2n}(\cdot, 1, 0)\big)
\hskip.2in t>0.
\end{align}
By the Stratonovich expansion (\ref{M-3}), therefore,
$$
\E u(t,x)=\E u(t,0)=\sum_{n=0}^\infty\E S_{2n}\big(g_{2n}(\cdot, t, 0)\big)
=\sum_{n=0}^\infty t^{(4-\alpha-\alpha_0)n}\E S_{2n}\big(g_{2n}(\cdot, 1, 0)\big)
$$
where the first equality comes from the stationarity in $x$.
By (\ref{P-1}), therefore, Theorem \ref{th-2} follows from the following computation:
$$
\begin{aligned}
&\lim_{t\to\infty}t^{-{4-\alpha-\alpha_0\over 3-\alpha}}\log\E u(t,x)
=\lim_{t\to\infty}t^{-{4-\alpha-\alpha_0\over 3-\alpha}}\log\sum_{n=0}^\infty t^{(4-\alpha-\alpha_0)n}\E S_{2n}\big(g_{2n}(\cdot, 1, 0)\big)\\
&=\lim_{t\to\infty}t^{-{4-\alpha-\alpha_0\over 3-\alpha}}\log\sum_{n=0}^\infty 
{t^{(4-\alpha-\alpha_0)n}\over (n!)^{3-\alpha}}\bigg({2(4-\alpha-2\alpha_0)^{4-\alpha-2\alpha_0\over 2}\over
(4-\alpha-\alpha_0)^{4-\alpha-\alpha_0}}\Big({{\cal M}\over 4-\alpha}\Big)^{4-\alpha\over 2}\bigg)^n\\
&=(3-\alpha)\bigg({2(4-\alpha-2\alpha_0)^{4-\alpha-2\alpha_0\over 2}\over
(4-\alpha-\alpha_0)^{4-\alpha-\alpha_0}}\Big({{\cal M}\over 4-\alpha}\Big)^{4-\alpha\over 2}\bigg)^{1\over 3-\alpha}
\end{aligned}
$$
where the last step follows from the elementary identity on asymptotics of Mittag-Leffler function (Lemma A.3, \cite{BCC}):
$$
\lim_{b\to\infty} b^{-1/\gamma}\log\sum_{n=0}^\infty {\theta^nb^n\over(n!)^\gamma}
=\gamma\theta^{1/\gamma}\hskip.2in \theta,\gamma>0
$$
with $b=t^{4-\alpha -\alpha_0}$, $\gamma =3-\alpha$ and 
$$
\theta =\bigg({2(4-\alpha-2\alpha_0)^{4-\alpha-2\alpha_0\over 2}\over
(4-\alpha-\alpha_0)^{4-\alpha-\alpha_0}}\Big({{\cal M}\over 4-\alpha}\Big)^{4-\alpha\over 2}\bigg).
$$
Recall that the variation ${\cal E}_0$ is defined in (\ref{T-5}) and define
$$
\begin{aligned}
\widetilde{\cal M}&=\sup_{g\in{\cal A}_d}\bigg\{\bigg(\int_0^1\!\!\int_0^1\int_{\R^d\times\R^d}{\gamma(x-y)\over \vert s-r\vert^{\alpha_0}}g^2(s, x)g^2(r, y)dxdydsdr\bigg)^{1/2}\\
&-{1\over 2}\int_0^1\int_{\R^d}\vert\nabla_x g(s, x)\vert^2dxds\bigg\}.
\end{aligned}
$$
By rescaling, $\widetilde{\cal M}=2^{\alpha\over 4-\alpha}{\cal M}$. By Lemma A-4, \cite{Chen-6},
\begin{align}\label{P-3}
{\cal E}_0={2-\alpha\over 2}2^{\alpha\over 2-\alpha}
\Big({4\widetilde{\cal M}\over 4-\alpha}\Big)^{4-\alpha\over 2-\alpha}
={2-\alpha\over 2}2^{2\alpha\over 2-\alpha}
\Big({4 {\cal M}\over 4-\alpha}\Big)^{4-\alpha\over 2-\alpha}.
\end{align}
Therefore, (\ref{P-1}) can be rewritten as\footnote{For comparison to Theorem 1.2, \cite{CH} in the setting
of time-independent Gaussian field, we formulate Theorem \ref{th-2} in terms of ${\cal M}$ instead of ${\cal E}_0$}
\begin{align}\label{P-4}
&\lim_{n\to\infty}{1\over n}\log (n!)^{3-\alpha}\E S_{2n}\big(g_{2n}(\cdot, 1, 0)\big)
=\log \bigg({(4-\alpha-2\alpha_0)^{4-\alpha-2\alpha_0\over 2}\over
2^3(4-\alpha-\alpha_0)^{4-\alpha-\alpha_0}}\Big({2{\cal E}_0\over 2-\alpha}\Big)^{2-\alpha\over 2}\bigg).
\end{align}
In the following subsections, we shall establish the upper and lower bounds, separately, for (\ref{P-4}).

\subsection{Upper bound for (\ref{P-4})}

In view of (\ref{B-1}) and (\ref{T-1}),
\begin{align}\label{P-5}
&\int_0^\infty e^{-\theta t} \E S_{2n}\big(g_{2n}(\cdot, t,0)\big)dt\\
&\le {\theta\over 2}\Big({1\over 2}\Big)^{3n}{1\over n!}\int_0^\infty dt\exp\Big\{-{\theta^2\over 2}t\Big\}\cr
&\times\E_0\bigg[\int_0^t\!\int_0^t\Big\vert \theta( s-r)+i \big(\beta(s)-\beta(r)\big)\Big\vert^{-\alpha_0}
\gamma\big(B(s)-B(r)\big)dsdr\nonumber
\bigg]^n.
\end{align}

Let $\theta>0$ be fixed but arbitrary (for a while). Given $\eta>0$, write
\begin{align}\label{P-6}
&\int_0^\infty dt\exp\Big\{-{\theta^2\over 2}t\Big\}
\E_0\bigg[\int_0^t\!\int_0^t\Big\vert \theta( s-r)+i \big(\beta(s)-\beta(r)\big)\Big\vert^{-\alpha_0}
\gamma\big(B(s)-B(r)\big)dsdr\bigg]^n\nonumber\\
&=\bigg\{\int_0^{\eta^{-2}n}+\int_{\eta^{-2}n}^\infty\bigg\}dt\exp\Big\{-{\theta^2\over 2}t\Big\}\\
&\times\E_0\bigg[\int_0^t\!\int_0^t\Big\vert \theta( s-r)+i \big(\beta(s)-\beta(r)\big)\Big\vert^{-\alpha_0}
\gamma\big(B(s)-B(r)\big)dsdr\bigg]^n.\nonumber
\end{align}
For the first term on the right hand side, we use the scaling identity
$$
\begin{aligned}
&\int_0^t\!\int_0^t\Big\vert \theta( s-r)+i \big(\beta(s)-\beta(r)\big)\Big\vert^{-\alpha_0}
\gamma\big(B(s)-B(r)\big)dsdr\\
&\buildrel d\over =t^{4-\alpha-2\alpha_0\over 2}
\int_0^1\!\int_0^1\Big\vert \theta( s-r)+i {1\over\sqrt{t}}\big(\beta(s)-\beta(r)\big)\Big\vert^{-\alpha_0}
\gamma\big(B(s)-B(r)\big)dsdr.
\end{aligned}
$$
So we have
$$
\begin{aligned}
&\int_0^{\eta^{-2}n} dt\exp\Big\{-{\theta^2\over 2}t\Big\}
  \E_0\bigg[\int_0^t\!\int_0^t\Big\vert \theta( s-r)
  +i \big(\beta(s)-\beta(r)\big)\Big\vert^{-\alpha_0}
\gamma\big(B(s)-B(r)\big)dsdr\bigg]^n\\
&=\int_0^{\epsilon^{-2}n} dt\exp\Big\{-{\theta^2\over 2}t\Big\}
  t^{{4-\alpha-2\alpha_0\over 2}n}\\
&\times\E_0\bigg[\int_0^1\!\int_0^1\Big\vert \theta( s-r)+i
  {1\over\sqrt{t}}\big(\beta(s)-\beta(r)\big)\Big\vert^{-\alpha_0}
\gamma\big(B(s)-B(r)\big)dsdr\bigg]^n\\
&\le \E_0\bigg[\int_0^1\!\int_0^1\Big\vert \theta( s-r)
  +i {\eta\over\sqrt{n}}\big(\beta(s)-\beta(r)\big)\Big\vert^{-\alpha_0}
\gamma\big(B(s)-B(r)\big)dsdr\bigg]^n\\
&\times\int_0^\infty \exp\Big\{-{\theta^2\over 2}t\Big\}t^{{4-\alpha-2\alpha_0\over 2}n}dt\\
&=\E_0\bigg[\int_0^1\!\int_0^1\Big\vert \theta( s-r)+i {\eta\over\sqrt{n}}\big(\beta(s)-\beta(r)\big)\Big\vert^{-\alpha_0}
\gamma\big(B(s)-B(r)\big)dsdr\bigg]^n\\
&\times\Big({2\over\theta^2}\Big)^{1+{4-\alpha-2\alpha_0\over 2}n}
\Gamma\Big(1+{4-\alpha-2\alpha_0\over 2}n\Big)\\
&=\big(1+o(1)\big)^n(n!)^{2-\alpha_0}(4-\alpha-2\alpha_0)^{{4-\alpha-2\alpha_0\over 2}n}
\theta^{-(4-\alpha-\alpha_0)n}\Big({2{\cal E}_0\over 2-\alpha}\Big)^{{2-\alpha\over 2}n}\hskip.2in (n\to\infty)
\end{aligned}
$$
where the last step follows from (\ref{T-20}) and Stirling formula.

As for the second term, we use the bound of Taylor expansion
$$
\begin{aligned}
&\E_0\bigg[\int_0^t\!\int_0^t\Big\vert \theta( s-r)+i \big(\beta(s)-\beta(r)\big)\Big\vert^{-\alpha_0}
\gamma\big(B(s)-B(r)\big)dsdr\bigg]^n\\
&\le n!t^{(1-\alpha_0)n}
\E_0\exp\bigg\{{1\over t^{1-\alpha_0}}\int_0^t\!\int_0^t\Big\vert \theta( s-r)+i \big(\beta(s)-\beta(r)\big)\Big\vert^{-\alpha_0}
\gamma\big(B(s)-B(r)\big)dsdr\bigg\}\\
&\le \big(1+o(1)\big)^nn!t^{(1-\alpha_0)n}\exp\Big\{\theta^{-{2\alpha_0\over 2-\alpha}}{\cal E}_0t\Big\}
\hskip.2in (n\to\infty)
\end{aligned}
$$
for large $t$, where the last step follows from Theorem \ref{th-6} with
$b=\eta=1$. Take $\theta>0$ sufficiently large so that
$$
c(\theta)\equiv{\theta^2\over 2}-\theta^{-{2\alpha_0\over 2-\alpha}}{\cal E}_0>0.
$$
We have
$$
\begin{aligned}
&\int_{\eta^{-2}n}^\infty dt\exp\Big\{-{\theta^2\over 2}t\Big\}
  \E_0\bigg[\int_0^t\!\int_0^t\Big\vert \theta( s-r)+i
  \big(\beta(s)-\beta(r)\big)\Big\vert^{-\alpha_0}
\gamma\big(B(s)-B(r)\big)dsdr\bigg]^n\\
&\le \big(1+o(1)\big)^nn!\int_{\eta^{-2}n}^\infty t^{(1-\alpha_0)n}\exp\{-c(\theta)t\}dt\hskip.2in (n\to\infty).
\end{aligned}
$$
We now claim that
\begin{align}\label{P-7}
  \lim_{\eta\to 0^+}\limsup_{n\to\infty}{1\over n}\log (n!)^{-(1-\alpha_0)}
  \int_{\eta^{-2}n}^\infty t^{(1-\alpha_0)n}\exp\{-c(\theta)t\}dt=-\infty.
\end{align}
Indeed, consider an i.i.d. sequence $X_1,\cdots, X_n$ with common distribution
$\Gamma \Big(1-\alpha_0, c(\theta)\Big)$, i.e., they have the common
density
$$
f(x)=\Gamma\big(1-\alpha_0\big)^{-1} c(\theta)^{(1-\alpha_0)}
x^{-\alpha_0}\exp\Big\{-c(\theta)x\Big\}\hskip.2in x>0
$$
and $X_0\sim\exp(c(\theta))$
 independent of $\{X_1,\cdots, X_n\}$. We have that
$$
X_0+X_1+\cdots X_n\sim\Gamma\Big(1+n(1-\alpha_0), \hskip.05in
c(\theta)\Big).
$$
Therefore
$$
\begin{aligned}
  &\Gamma\big(1+n(1-\alpha_0)\big)^{-1} c(\theta)^{1+(1-\alpha_0)n}
    \int_{\eta^{-2}n}^\infty t^{(1-\alpha_0)n}\exp\{-c(\theta)t\}dt\\
&=\P\Big\{{X_0+X_1+\cdots +X_n\over n}\ge\eta^{-2}\Big\}.
\end{aligned}
$$
Therefore, (\ref{P-7}) follows form Cram\'er large deviation
(Theorem 2.2.3, p.27, \cite{DZ}) which particularly leads to
$$
\lim_{\eta\to 0^+}\limsup_{n\to\infty}{1\over n}\log
\P\Big\{{X_0+X_1+\cdots +X_n\over n}\ge\eta^{-2}\Big\}=-\infty.
$$
In summary,
$$
\begin{aligned}
  &\int_{\eta^{-2}n}^\infty dt \exp\Big\{-{\theta^2\over 2} t\Big\}
    \int _0^t\!\int_0^t\Big\vert \theta( s-r)+i \big(\beta(s)-\beta(r)\big)\Big\vert^{-\alpha_0}
\gamma\big(B(s)-B(r)\big)dsdr\bigg]^n\\
&\le \big(1+o(1)\big)^n(n!)^{2-\alpha_0}\exp\{-L_\eta n\}
\end{aligned}
$$
where $L_\eta>0$ can be sufficiently large if $\eta$ is sufficiently small.

So we reach the point that in the decomposition (\ref{P-6}),
the bound of the first term
dominates the bound of the second term as $\eta>0$ is small. Consequently, by (\ref{P-5}) we have
$$
\begin{aligned}
&\int_0^\infty e^{-\theta t} \E S_{2n}\big(g_{2n}(\cdot, t,0)\big)dt\\
&\le \big(1+o(1)\big)^n\Big({1\over 2}\Big)^{3n}(n!)^{1-\alpha_0}
(4-\alpha-2\alpha_0)^{{4-\alpha-2\alpha_0\over 2}n}
\theta^{-(4-\alpha-\alpha_0)n}\Big({2{\cal E}_0\over 2-\alpha}\Big)^{{2-\alpha\over 2}n}.
\end{aligned}
$$
On the other hand, by (\ref{P-2})
$$
\begin{aligned}
&\int_0^\infty e^{-\theta t} \E S_{2n}\big(g_{2n}(\cdot, t,0)\big)dt=\E S_{2n}\big(g_{2n}(\cdot, 1,0)\big)
\int_0^\infty e^{-\theta t}t^{(4-\alpha-\alpha_0)n}dt\\
&=\theta^{-(1+(4-\alpha-\alpha_0)n}\Gamma\big(1+(4-\alpha-\alpha_0)n\big)\E S_{2n}\big(g_{2n}(\cdot, 1,0)\big)\\
&=\big(1+o(1)\big)^n\theta^{-(4-\alpha-\alpha_0)n}(4-\alpha-\alpha_0)^{(4-\alpha-\alpha_0)n}
(n!)^{4-\alpha-\alpha_0}\E S_{2n}\big(g_{2n}(\cdot, 1,0)\big).
\end{aligned}
$$
So we have the upper bound of (\ref{P-4}):
\begin{align}\label{P-8}
&\limsup_{n\to\infty}{1\over n}\log (n!)^{3-\alpha}\E S_{2n}\big(g_{2n}(\cdot, 1, 0)\big)
\le\log \bigg({(4-\alpha-2\alpha_0)^{4-\alpha-2\alpha_0\over 2}\over
2^3(4-\alpha-\alpha_0)^{4-\alpha-\alpha_0}}\Big({2{\cal E}_0\over 2-\alpha}\Big)^{2-\alpha\over 2}\bigg).
\end{align}
\qed

\subsection{Lower bound for (\ref{P-4})}

Take $\theta=1$. By (\ref{B-1}) and (\ref{B-2})
$$
\begin{aligned}
&\int_0^\infty e^{-t} \E S_{2n}\big(g_{2n}(\cdot, t,0)\big)dt
=\Gamma(\alpha_0)^{-n}{1\over 2}\Big({1\over 2}\Big)^{3n}{1\over n!}
\int_0^\infty dt\exp\Big\{-{t\over 2}\Big\}\\
&\times\E_0\bigg[\E^\kappa\int_{\R_+\times\R^d}\bigg\vert\int_0^t\exp\Big\{i\lambda\big(\kappa(s)+\beta(s)\big)+i\xi\cdot B(s)\Big\}ds\bigg\vert^2{d\lambda\over \lambda^{1-\alpha_0}}\mu(d\xi)\bigg]^n\\
&\ge \Gamma(\alpha_0)^{-n}{1\over 2}\Big({1\over 2}\Big)^{3n}{1\over n!}
\int_0^\infty dt\exp\Big\{-{t\over 2}\Big\}\\
&\times\E_0\bigg[\E^\kappa\int_{\R_+\times\R^d}\bigg\vert\E^\beta\int_0^t\exp\Big\{i\lambda\big(\kappa(s)+\beta(s)\big)+i\xi\cdot B(s)\Big\}ds\bigg\vert^2{d\lambda\over \lambda^{1-\alpha_0}}\mu(d\xi)\bigg]^n\\
&=\Gamma(\alpha_0)^{-n}{1\over 2}\Big({1\over 2}\Big)^{3n}{1\over n!}
\int_0^\infty dt\exp\Big\{-{t\over 2}\Big\}\\
&\times\E_0\bigg[\E^\kappa\int_{\R_+\times\R^d}\bigg\vert\int_0^t\exp\Big\{i\lambda\kappa(s)-{\lambda^2\over 2}s
+i\xi\cdot B(s)\Big\}ds\bigg\vert^2{d\lambda\over \lambda^{1-\alpha_0}}\mu(d\xi)\bigg]^n
\end{aligned}
$$
where the inequality follows from Jensen's inequality. Further,
$$
\begin{aligned}
&\E^\kappa\int_{\R_+\times\R^d}\bigg\vert\int_0^t\exp\Big\{i\lambda\kappa(s)-{\lambda^2\over 2}s
+i\xi\cdot B(s)\Big\}ds\bigg\vert^2{d\lambda\over \lambda^{1-\alpha_0}}\mu(d\xi)\\
&=\int_{\R_+}{d\lambda\over\lambda^{1-\alpha_0}}
\int_0^t\!\!\int_0^t\exp\Big\{-\lambda\vert s-r\vert-{\lambda^2\over 2 }(s+r)\Big\}
\gamma\big(B(s)-B(r)\big)dsdr\\
&\buildrel d\over =t^{4-\alpha-2\alpha_0\over 2}\int_{\R_+}{d\lambda\over\lambda^{1-\alpha_0}}
\int_0^1\!\!\int_0^1\exp\Big\{-\lambda\vert s-r\vert-{\lambda^2\over 2 t}(s+r)\Big\}
\gamma\big(B(s)-B(r)\big)dsdr\\
&\ge t^{4-\alpha-2\alpha_0\over 2}\exp\Big\{-(\delta^2 t)^{-1}\Big\}\int_0^{\delta^{-1}}{d\lambda\over\lambda^{1-\alpha_0}}
\int_0^1\!\!\int_0^1\exp\Big\{-\lambda\vert s-r\vert\Big\}\gamma\big(B(s)-B(r)\big)dsdr\\
&= \Gamma(\alpha_0)
t^{4-\alpha-2\alpha_0\over 2}\exp\Big\{-(\delta^2 t)^{-1}\Big\}
\int_0^1\!\!\int_0^1\gamma^0_\delta(s-r)\gamma\big(B(s)-B(r)\big)dsdr.
\end{aligned}
$$
Here we recall (\ref{M-17}) for the definition of $\gamma^0_\delta(\cdot)$.

Therefore,
$$
\begin{aligned}
&\int_0^\infty e^{-t} \E S_{2n}\big(g_{2n}(\cdot, t,0)\big)dt\\
&\ge {1\over 2}\Big({1\over 2}\Big)^{3n}{1\over n!}
\E_0\bigg[\int_0^1\!\!\int_0^1\gamma^0_\delta(s-r)\gamma\big(B(s)-B(r)\big)dsdr\bigg]^n\\
&\times\int_0^\infty \exp\Big\{-{t\over 2}\Big\}\exp\Big\{-n(\delta^2 t)^{-1}\Big\}t^{{4-\alpha-2\alpha_0\over 2}n}dt\\
&\ge {1\over 2}\Big({1\over 2}\Big)^{3n}{1\over n!}
\E_0\bigg[\int_0^1\!\!\int_0^1\gamma^0_\delta(s-r)\gamma\big(B(s)-B(r)\big)dsdr\bigg]^n\\
&\times\exp\Big\{-{n^{1-\eta}\over \delta^2}\Big\}\int_{n^\eta}^\infty \exp\Big\{-{t\over 2}\Big\}
t^{{4-\alpha-2\alpha_0\over 2}n}dt
\end{aligned}
$$
where $0<\eta<1$.
By constructing relevant  independent Gamma-distributed random variables (as we did in the proof of (\ref{P-7})) and
by the law of large numbers, one can show that
$$
\int_{n^\eta}^\infty \exp\Big\{-{t\over 2}\Big\}
t^{{4-\alpha-2\alpha_0\over 2}n}dt=\big(1+o(1)\big) 2^{1+{4-\alpha-2\alpha_0\over 2}n}
\Gamma\Big(1+{4-\alpha-2\alpha_0\over 2}n\Big)\hskip.2in (n\to\infty).
$$

By (\ref{T-19}) and Stirling formula, therefore,
$$
\begin{aligned}
&\int_0^\infty e^{-t} \E S_{2n}\big(g_{2n}(\cdot, t,0)\big)dt
\ge \big(1+o(1)\big)^n\Big({1\over 2}\Big)^{3n}(n!)^{1-\alpha_0}
(4-\alpha -2\alpha_0)^{{4-\alpha-2\alpha_0\over 2}n}
\Big({2{\cal E}_\delta\over 2-\alpha}\Big)^{{2-\alpha\over 2}n}
\end{aligned}
$$
as $n\to\infty$. 

On the other hand, by (\ref{P-2})
$$
\begin{aligned}
&\int_0^\infty e^{-t} \E S_{2n}\big(g_{2n}(\cdot, t,0)\big)dt
=\E S_{2n}\big(g_{2n}(\cdot, 1,0)\big)\int_0^\infty e^{-t}t^{(4-\alpha-\alpha_0)n}dt\\
&=\Gamma\Big(1+(4-\alpha-\alpha_0)n\big)\E S_{2n}\big(g_{2n}(\cdot, 1,0)\big)\\
&=\big(1+o(1)\big)^n (n!)^{4-\alpha-\alpha_0}(4-\alpha-\alpha_0)^{(4-\alpha-\alpha_0)n}
\E S_{2n}\big(g_{2n}(\cdot, 1,0)\big).
\end{aligned}
$$
Combining them together, we have 
\begin{align}\label{P-9}
&\liminf_{n\to\infty}{1\over n}\log (n!)^{3-\alpha}\E S_{2n}\big(g_{2n}(\cdot, 1, 0)\big)
\ge\log \bigg({(4-\alpha-2\alpha_0)^{4-\alpha-2\alpha_0\over 2}\over
2^3(4-\alpha-\alpha_0)^{4-\alpha-\alpha_0}}\Big({2{\cal E}_\delta\over 2-\alpha}\Big)^{2-\alpha\over 2}\bigg).
\end{align}
Letting $\delta\to 0^+$ on the righ hand side gives the expected lower bound. \qed

\section{Appendix}

Let $\theta>0$. Corresponding to the norms $\|\cdot\|_{j,k}^{(1)}$ and $\|\cdot\|_{j,k}^{(2)}$ introduced in (\ref{D-4})
and (\ref{D-5}), the following lemma is concerned with the bound given in Theorem \ref{th-3} with
$G_l(t,x)=e^{-\theta t}G(t,x)$.

\begin{lemma}\label{A} Assume (\ref{intro-5}).
\begin{enumerate}
	\item[(i)] 	For any $\theta>0$,
	\begin{align}\label{A-1}
	&\int_{\R_+\times\R^d}\bigg\vert\int_0^\infty\!\!\int_{\R^d}
e^{-(\theta+\lambda) t+i\xi\cdot x}G(t,x)dxdt\bigg\vert {d\lambda\over \lambda^{1-\alpha_0}}\mu(d\xi)\\
&=\int_{\R^+\times\R^d}{1\over (\theta+\lambda)^2+\vert\xi\vert^2}{d\lambda\over \lambda^{1-\alpha_0}}
\mu(d\xi)<\infty.\nonumber
\end{align}
	\item[(ii)] For any $\theta>0$,
\begin{align}\label{A-2}
 &\int_0^\infty\!\int_0^\infty dsdt e^{-\theta(t+s)}\int_{\R^d\times\R^d}\vert s-t\vert^{-\alpha_0}
 \gamma(x-y)G(t,x)G(s,y)dxdy\\
  &={1\over 2}\int_{\R^+\times\R^d}{d\lambda\over\lambda^{1-\alpha_0}}\mu(d\xi)
{1\over (\theta+\lambda)^2+\vert\xi\vert^2}{1\over \theta^2+\vert\xi\vert^2}\nonumber\\
&+{1\over 2\theta}\int_{\R^+\times\R^d}{d\lambda\over\lambda^{1-\alpha_0}}\mu(d\xi)
{\theta+\lambda\over (\theta+\lambda)^2+\vert\xi\vert^2}{1\over \theta^2+\vert\xi\vert^2}\nonumber\\
&<\infty.\nonumber
\end{align}
Further, there is a constant $C>0$ such that
	\begin{align}\label{A-3}
 \int_0^\infty\!\int_0^\infty dsdt e^{-\theta(t+s)}\int_{\R^d\times\R^d}\vert s-t\vert^{-\alpha_0}
 \gamma(x-y)G(t,x)G(s,y)dxdy\le C\theta^{-2}
\end{align}
for large $\theta$.
	\end{enumerate}

\end{lemma}

\proof The identity in (\ref{A-1}) follows from
$$
\int_0^\infty\!\!\int_{\R^d}e^{-(\theta+\lambda) t+i\xi\cdot x}G(t,x)dxdt\
=\int_0^\infty e^{-(\theta+\lambda) t}{\sin (t\vert\xi\vert)\over\vert\xi\vert}dt
={1\over (\theta+\lambda)^2+\vert\xi\vert^2}
$$
where the first equality follows from (\ref{M-11}) and the second from integration by
parts. To show the finiteness, 
$$
\int_{\R^+\times\R^d}{1\over (\theta+\lambda)^2+\vert\xi\vert^2}{d\lambda\over \lambda^{1-\alpha_0}}
\mu(d\xi)
\le \int_{\R^+\times\R^d}{1\over \lambda^2+(\theta^2+\vert\xi\vert^2)}{d\lambda\over \lambda^{1-\alpha_0}}
\mu(d\xi)
$$
and by variable substitution
$$
\int_{\R^+}{1\over \lambda^2+(\theta^2+\vert\xi\vert^2)}{d\lambda\over \lambda^{1-\alpha_0}}
=\bigg({1\over \theta^2+\vert\xi\vert^2}\bigg)^{2-\alpha_0\over 2}\int_{\R_+}{1\over 1+\lambda^2}
{d\lambda\over \lambda^{1-\alpha_0}}.
$$
By (\ref{intro-5}), therefore,
$$
\int_{\R^+\times\R^d}{1\over (\theta+\lambda)^2+\vert\xi\vert^2}{d\lambda\over \lambda^{1-\alpha_0}}
\mu(d\xi)
\le C\int_{\R^d}\bigg({1\over \theta^2+\vert\xi\vert^2}\bigg)^{2-\alpha_0\over 2}\mu(d\xi)<\infty.
$$

We now prove (\ref{A-2}). By (\ref{intro-5}),
$$
\begin{aligned}
&\int_0^\infty\!\int_0^\infty dsdte^{-\theta(t+s)}\int_{\R^d\times\R^d}\vert s-t\vert^{-\alpha_0}\gamma(x-y)G(t,x)G(s,y)dxdy\cr
&=2\int_{\R^+\times\R^d}{d\lambda\over\lambda^{1-\alpha_0}}\mu(d\xi)
\int\!\!\int_{\{s\le t\}}dsdt e^{-\theta(s+t)}
e^{-\lambda (t-s)}\int_{\R^d\times\R^d}e^{i\xi\cdot(x-y)}G(t x)G(s,y)dxdy\\
&=2\int_{\R^+\times\R^d}{d\lambda\over\lambda^{1-\alpha_0}}\mu(d\xi)
\int\!\!\int_{\{s\le t\}}dsdt
e^{-\lambda (t-s)}e^{-\theta(s+t)}{\sin (t\vert\xi\vert)\sin (s\vert\xi\vert)\over\vert\xi\vert^2}\\
&=2\int_{\R^+\times\R^d}{d\lambda\over\lambda^{1-\alpha_0}}{\mu(d\xi)\over\vert\xi\vert^2}
\int_0^\infty ds e^{-2\theta s}\sin (s\vert\xi\vert)\int_s^\infty e^{-(\theta+\lambda)(t-s)}\sin (t\vert\xi\vert)dt.
\end{aligned}
$$
By the relation
$$
\sin (t\vert\xi\vert)=\sin ((t-s)\vert\xi\vert)\cos (s\vert\xi\vert)+\cos((t-s)\vert\xi\vert)\sin(s\vert\xi\vert)
$$
and by integration by parts
$$
\begin{aligned}
&\int_s^\infty e^{-(\theta+\lambda)(t-s)}\sin (t\vert\xi\vert)dt\cr
&=\cos(s\vert\xi\vert)\int_0^\infty e^{-(\theta+\lambda)t}\sin (t\vert\xi\vert)dt
+\sin (s\vert\xi\vert)\int_0^\infty e^{-(\theta+\lambda)t}\cos (t\vert\xi\vert)dt\\
&=\cos(s\vert\xi\vert){\vert\xi\vert\over (\theta+\lambda)^2+\vert\xi\vert^2}
+\sin(s\vert\xi\vert){\theta+\lambda\over (\theta+\lambda)^2+\vert\xi\vert^2}.
\end{aligned}
$$
Thus,
$$
\begin{aligned}
&\int_0^\infty\!\int_0^\infty dsdte^{-\theta(t+s)}
\int_{\R^d\times\R^d}\vert s-t\vert^{-\alpha_0}\gamma(x-y)G(t,x)G(s,y)dxdy\\
&=2\int_{\R^+\times\R^d}{d\lambda\over\lambda^{1-\alpha_0}}{\mu(d\xi)\over\vert\xi\vert^2}
{\vert\xi\vert\over (\theta+\lambda)^2+\vert\xi\vert^2}
\int_0^\infty  e^{-2\theta s}\sin (s\vert\xi\vert)\cos(s\vert\xi\vert)ds\\
&+2\int_{\R^+\times\R^d}{d\lambda\over\lambda^{1-\alpha_0}}{\mu(d\xi)\over\vert\xi\vert^2}
{\theta+\lambda\over (\theta^2+\lambda)^2+\vert\xi\vert^2}
\int_0^\infty  e^{-2\theta s}\sin^2 (s\vert\xi\vert)ds.
\end{aligned}
$$
Using integration by parts again
$$
\int_0^\infty  e^{-2\theta s}\sin (s\vert\xi\vert)\cos(s\vert\xi\vert)ds
={1\over 2}\int_0^\infty e^{-2\theta s}\sin (2s\vert\xi\vert)ds={1\over 2}{2\vert\xi\vert\over 4\theta^2+4\vert\xi\vert^2}
={1\over 4}{\vert\xi\vert\over \theta^2+\vert\xi\vert^2}.
$$
A similar treatment also leads to
$$
\int_0^\infty  e^{-2\theta s}\sin^2 (s\vert\xi\vert)ds
={\vert\xi\vert\over 2\theta}\int_0^\infty  e^{-2\theta s}\sin(2s\vert\xi\vert)ds
={1\over 4\theta}{\vert\xi\vert^2\over \theta^2+\vert\xi\vert^2}.
$$
Bringing them together leads to the identity leads to the identity in (\ref{A-2}).

Establishing the finiteness in (\ref{A-2}) is an easy job and can be seen from the following estimate for (\ref{A-3}).
To show (\ref{A-3}), all we need is to bound the two terms on the right hand side of (\ref{A-2}) separately. We first
 work on the second term.

By variable substitution
$$
\begin{aligned}
&\int_0^\infty{d\lambda\over\lambda^{1-\alpha_0}}{\theta+\lambda\over (\theta+\lambda)^2+\vert\xi\vert^2}
=\int_\theta^\infty{d\lambda\over(\lambda -\theta)^{1-\alpha_0}}{\lambda\over \lambda^2+\vert\xi\vert^2}\\
&={1\over\vert \xi\vert^{1-\alpha_0}}\int_{\vert\xi\vert^{-1}\theta}^\infty 
{d\lambda\over(\lambda-\vert\xi\vert^{-1}\theta)^{1-\alpha_0}}{\lambda\over \lambda^2+1}.
\end{aligned}
$$
Consider the decomposition
$$
\int_{\vert\xi\vert^{-1}\theta}^\infty 
{d\lambda\over(\lambda-\vert\xi\vert^{-1}\theta)^{1-\alpha_0}}{\lambda\over \lambda^2+1}
\le \bigg\{\int_{\vert\xi\vert^{-1}\theta}^{2\vert\xi\vert^{-1}\theta}+\int_{2\vert\xi\vert^{-1}\theta}^\infty \bigg\}
{d\lambda\over(\lambda-\vert\xi\vert^{-1}\theta)^{1-\alpha_0}}{\lambda\over \lambda^2+1}.
$$
For the first term
$$
\begin{aligned}
&\int_{\vert\xi\vert^{-1}\theta}^{2\vert\xi\vert^{-1}\theta}{d\lambda\over(\lambda-\vert\xi\vert^{-1}\theta)^{1-\alpha_0}}{\lambda\over \lambda^2+1}\le\Big({\vert\xi\vert\over\theta}\Big)\int_{\vert\xi\vert^{-1}\theta}^{2\vert\xi\vert^{-1}\theta}{d\lambda\over(\lambda-\vert\xi\vert^{-1}\theta)^{1-\alpha_0}}\\
&=\Big({\vert\xi\vert\over\theta}\Big)\int_0^{\vert\xi\vert^{-1}\theta}{d\lambda\over\lambda^{1-\alpha_0}}
={1\over\alpha_0}\Big({\vert\xi\vert\over\theta}\Big)^{1-\alpha_0}.
\end{aligned}
$$
As for the second term
$$
\int_{2\vert\xi\vert^{-1}\theta}^\infty 
{d\lambda\over(\lambda-\vert\xi\vert^{-1}\theta)^{1-\alpha_0}}{\lambda\over \lambda^2+1}
\le \int_{2\vert\xi\vert^{-1}\theta}^\infty 
{d\lambda\over(\lambda/2)^{1-\alpha_0}}{\lambda\over \lambda^2+1}\le C\Big({\vert\xi\vert\over\theta}\Big)^{1-\alpha_0}.
$$
In summary, we have the bound (the constant $C$ can be different from place to place in our argument)
$$
\int_0^\infty{d\lambda\over\lambda^{1-\alpha_0}}{\theta+\lambda\over (\theta+\lambda)^2+\vert\xi\vert^2}
\le C\theta^{-(1-\alpha_0)}.
$$
By Fubini's theorem, therefore,
$$
\begin{aligned}
&\int_{\R^+\times\R^d}{d\lambda\over\lambda^{1-\alpha_0}}\mu(d\xi)
{\theta+\lambda\over (\theta+\lambda)^2+\vert\xi\vert^2}{1\over \theta^2+\vert\xi\vert^2}
\le C\theta^{-(1-\alpha_0)}\int_{\R^d}{1\over \theta^2+\vert\xi\vert^2}\mu(d\xi)\\
&\le C\theta^{-(3-\alpha_0)}\mu(\{\vert\xi\vert\le 1\})+C\theta^{-(1-\alpha_0)}\int_{\{\vert\xi\vert\ge 1\}}{1\over \theta^2+\vert\xi\vert^2}\mu(d\xi).
\end{aligned}
$$
By Minkowski inequality with $p=2/\alpha_0$ and $q=2(2-\alpha_0)^{-1}$,
$$
\theta^2+\vert\xi\vert^2\ge C^{-1}\theta^{\alpha_0}\vert\xi\vert^{2-\alpha_0}.
$$
So we have 
$$
\theta^{-(1-\alpha_0)}\int_{\{\vert\xi\vert\ge 1\}}{1\over \theta^2+\vert\xi\vert^2}\mu(d\xi)
\le C\theta^{-1}\int_{\{\vert\xi\vert\ge 1\}}\vert\xi\vert^{-(2-\alpha_0)}\mu(d\xi).
$$
The integral on the right hand side is finite under (\ref{intro-5}). Therefore, we have established the expected bound
$$
{1\over 2\theta}\int_{\R^+\times\R^d}{d\lambda\over\lambda^{1-\alpha_0}}\mu(d\xi)
{\theta+\lambda\over (\theta+\lambda)^2+\vert\xi\vert^2}{1\over \theta^2+\vert\xi\vert^2}
\le C\theta^{-2}.
$$
Notice that
$$
{1\over (\theta+\lambda)^2+\vert\xi\vert^2}{1\over \theta^2+\vert\xi\vert^2}
\le{1\over\theta}{\theta+\lambda\over (\theta+\lambda)^2+\vert\xi\vert^2}{1\over \theta^2+\vert\xi\vert^2}.
$$
The first term in (\ref{A-2}) yields the same bound. \qed

\vskip 1.in

\begin{tabular}{lll}
Xia Chen  \\
Department of Mathematics  \\
University of Tennessee \\
Knoxville TN 37996, USA  \\
{\tt xchen3@tennessee.edu} 
\end{tabular}


\begin{thebibliography}{999}

\bibitem{Anderson}
  Anderson, P. W. Localized Magnetic States in Metals.
  {\sl Phys. Rev. \bf 124 } (1961) 41–53. 


\bibitem{BC}
 Balan, M. R. and Conus, D. Intermittency for wave and heat equations
 with fractional noise in time. {\sl Ann. Probab. \bf 44} (2016) 1488-1534.

\bibitem{BCC}
  Balan, R. M., Chen, L and Chen, X.
  Exact asymptotics of the stochastic wave equation with
  time-independent noise. {\sl Annales de l'Institut Henri Poincare \bf 58}
  (2022),  1590-1620.





\bibitem{Chen-1} 
Chen, X. {\sl Random Walk Intersections: Large Deviations
and Related Topics.} Mathematical Surveys and Monographs, {\bf 157}.
American Mathematical Society, Providence 2009.




\bibitem{Chen-2} 
  Chen, X. Quenched asymptotics for Brownian motion in generalized
  Gaussian potential.
  {\sl Ann. Probab. \bf 42} (2014),  576-622.

\bibitem{Chen-6} 
 Chen, X. Spatial asymptotics for the parabolic Anderson models with generalized time-space Gaussian noise. 
 {\sl Ann. Probab. \bf 44} (2016), 1535-1598


\bibitem{Chen-5}
Chen, X.  Moment asymptotics for parabolic Anderson equation with fractional time-space noise: in 
Skorokhod regime \sl{Annales de l'Institut Henri Poincare \bf 53}  (2017) 819-841.




\bibitem{Chen-4}
Chen, X.  Parabolic Anderson model with rough or critical Gaussian noise. \sl{Annales de l'Institut Henri Poincare \bf 55} (2019),  941-976. 


\bibitem{Chen-3}
Chen, X. Exponential asymptotics for Brownian self-intersection local times under Dalang's condition.
{\sl Electronic J. P. \bf 28} (2023),  1-17.

\bibitem{CDST-1}
Chen, X., Deya, A. Song, J. and Tindel, S. 
Hyperbolic Anderson model 2: {\sl Strichartz estimates and Stratonovich setting PDF International Mathematics Research Notices.  \bf 00} (2023),   1-54.

\bibitem{CDST}
Chen, X., Deya, A. Song, J. and Tindel, S. Solving the hyperbolic Anderson model I: Skorohod
setting. {\sl Annales de l'Institut Henri Poincare \bf 61} (2025), 1794-1814

\bibitem{CHSX}
Chen, X., Hu, Y. Song, J. and Xing, F. Exponential asymptotics for time-space Hamiltonians.  
\sl{Annales de l'Institut Henri Poincare \bf 51} (2015), 1529-1561 

\bibitem{CH}
Chen, X. and Hu, Y. Hyperbolic Anderson equations with general time-independent
	Gaussian noise: Stratonovich regime {\sl Ann. Probab.} (to appear)





\bibitem{Dalang}
Dalang, R. C. Extending martingale measure stochastic integral with
applications to spatially homogeneous S.P.D.E's. {\sl Electron. J. Probab.
\bf 4} (1999),  1-29.


\bibitem{DMT}
    Dalang, R. C., Mueller, C. and Tribe, R.
    A Feynman-Kac-type formula for the deterministic and stochastic
    wave equations and other P.D.E.'s. {\sl Trans. Amer. Math. Soc. \bf 360}
    (2008),  4681-4703





\bibitem{DZ}
  Dembo, A. and Zeitouni, O. (1997). {\sl Large Deviations Techniques  and
    Applications.} 2nd ed. Springer, New York.



\bibitem{hubook} Hu, Y.  Analysis on Gaussian spaces. World Scientific Publishing Co. Pte. Ltd., Hackensack, NJ, 2017. 


\bibitem{humeyer} Hu, Y. Z.  and  Meyer, P.-A. Sur les int\'egrales multiples de Stratonovitch.  S\'eminaire de Probabilit\'es, XXII, 72–81, Lecture Notes in Math., 1321, Springer, Berlin, 1988. 

\bibitem{humeyer93} Hu, Y. and Meyer, P. A. On the approximation of multiple Stratonovich integrals.
Stochastic processes, 141-147, Springer, New York, 1993.

\bibitem{HHNS}
  Hu, Y. , Huang, J.,  Nualart, D.  and Tindel, S.
  Stochastic heat equations with general multiplicative Gaussian noise: 
  H\"older continuity and intermittency.
  {\sl Electron. J. Probab. \bf 20} (2015) 1-50

\bibitem{Evans} Evans, L.  C. Partial differential equations. Second edition. Graduate Studies in Mathematics, 19. American Mathematical Society, Providence, RI, 2010.  


\bibitem{Feller}
  Feller, W. (1971). {\sl An introduction to Probability Theory and Its Applications \bf II}
  2nd ed. Wiley, New York.



\bibitem{Hairer}
  Hairer, M.  Solving the KPZ equation.
  {\sl Ann. of Math. \bf 178} (2013) 559–664.








\bibitem{K}
  Kallenberg, O. (2002). {\sl Foundations of Modern Probability},
  2nd ed. Springer, New York.

\bibitem{KPZ} Kardar, M.,  Parisi, G.   and Zhang, Y.-C.   
Dynamic Scaling of Growing Interfaces. 
Physical Review Letters. 56 (9), 1986,  889–892.

\bibitem{Li-Shao}
Li, W.V. and Shao, Q.-M.  Gaussian processes: inequalities, small ball probabilities and applications.
{\sl Handbook Statist. \bf 19} (2001), 533–597.



\bibitem{LSL}
  Lipschutz, S., Spiegel, M. R. and  Liu, J.
  Mathematical Handbook of Formulas and Tables,
  Schaum's Outline Series (3rd ed.), (2009)  McGraw-Hill.



\bibitem{MR}
  Marcus, M. B. and Rosen, J. {\sl Markov Processes, Gaussian Processes, and Local Times}
  Cambridge Studies in Advanced Mathematics {\bf 100}, Cambridge Univ. Press, Cambridge (2006)

 \bibitem{pipiras} Pipiras, V. and  Taqqu, M.  S. Integration questions related to fractional Brownian motion. Probab. Theory Related Fields 118 (2000), no. 2, 251-291.





  
\bibitem{Song}  
Song, J.  On a class of stochastic partial differential equations. {\sl Stochastic Process. Appl. \bf 127 }
(2017),  37–79.

\end{thebibliography}
\end{document}